%% file: qtt.tex
\documentclass[a4paper]{amsart}
\input{./preamble}
\title{Tensor Rank bounds for Point Singularities in $\mathbb{R}^3$
}
\author{C. Marcati, M. Rakhuba, \and Ch. Schwab}
\email{\{carlo.marcati, maksim.rakhuba, christoph.schwab\}@sam.math.ethz.ch}
\address{Seminar f\"ur Angewandte Mathematik (SAM), 
ETH Z\"urich\\ R{\"a}mistrasse 101, 8092 Z{\"u}rich, Switzerland}
\thanks{The work of M. Rakhuba was supported by ETH Grant ETH-44 17-1. 
Ch. Schwab acknowledges useful discussions at the meeting WS1936 at the 
Mathematical Research Institute Oberwolfach (MFP) from 02-06Sept2019.}

\subjclass[2010]{Primary 35A35; 15A69, 35J15, 41A25, 41A46, 65N30}
\keywords{Quantized Tensor Train, tensor networks, low-rank approximation, exponential convergence, Schrödinger equation}
\begin{document}
\begin{abstract}
We analyze rates of approximation 
by quantized, tensor-structured representations
of functions with isolated point singularities in ${\mathbb R}^3$.
We consider functions in countably normed Sobolev spaces with radial weights and  
analytic- or Gevrey-type control of weighted semi-norms.
Several classes of boundary value and eigenvalue problems from science and engineering
are discussed whose solutions belong to the countably normed spaces.

It is shown that quantized, tensor-structured approximations of functions
in these classes exhibit tensor ranks bounded polylogarithmically with respect
to the accuracy $\epsilon\in(0,1)$ in the Sobolev space $H^1$. 
We prove  \emph{exponential convergence rates} 
of three specific types of quantized tensor decompositions:
quantized tensor train (QTT), transposed QTT and Tucker-QTT.
In addition, the bounds for the patchwise decompositions are uniform with respect to the position of the
point singularity.
An auxiliary result of independent interest is
the proof of exponential convergence of $hp$-finite element approximations
for Gevrey-regular functions with point singularities 
in the unit cube $Q=(0,1)^3$.
Numerical examples of function approximations 
and of Schrödinger-type eigenvalue problems
illustrate the theoretical results.
\end{abstract}
\maketitle
{\footnotesize
\tableofcontents
}

%%%%%%%%%%%%%%%%%%%%%%%%%%%%%%%%%%%%%%%%%%%%%%%%%%%%%%%%%%%%%%%%%%%%%%%%%%%%%%%%%%%%%
\section{Introduction}
\label{sec:Intr}
%%%%%%%%%%%%%%%%%%%%%%%%%%%%%%%%%%%%%%%%%%%%%%%%%%%%%%%%%%%%%%%%%%%%%%%%%%%%%%%%%%%%%
Recent years have seen the emergence of 
\emph{structured numerical linear algebra} 
in scientific computing and data science.
We mention only formatted matrix algebras, such as 
${\mathcal H}$-matrices (e.g. \cite{Hackb_HMatBook} and the references there)
and tensor formats 
(e.g. \cite{Kolda2009,OST09,Osele_DMRGLA,Hackb_TensorBook,BKTensBook2018} 
 and the references there).
To date, the impact of these methods was, first and foremost, 
on the corresponding scientific computing applications: 
being abstracted from fast multipole methods,
formatted computational matrix algebras impact directly
the numerical solution of elliptic and parabolic 
partial differential equations (see, e.g., \cite{HKT2008,BSU_TensSurv,GKT13}).
Numerical tensor algebras, 
derived from quantum chemistry (e.g. \cite{Sch2011,BSU_TensSurv} and the references there)
have obvious applications in data-science,
where massive $n$-way data naturally arises
and needs to be efficiently handled numerically.
Furthermore, tensor-structured formats have, 
in recent years, been linked to deep neural networks 
(see \cite{NCohenEtAl,DBLP:journals/corr/abs-1901-10801} and the references there).
We now comment on more specific developments 
in these areas which are directly related to the 
present paper, and the mathematical results obtained in it.

We are concerned with the approximation of functions with isolated
point singularities using tensor-structured representations.
In particular, we approximate, using quantized tensor decompositions,
three-dimensional arrays of coefficients associated
with the finite element projection of functions over trilinear Lagrange basis functions.

\emph{Quantization} refers to the reshaping of an array of coefficients of size $2^\ell \times 2^\ell \times 2^\ell$ into a multidimensional array of size $2\times \dots \times 2$.
The application of tensor decompositions (e.g., the Tensor-Train
decomposition~\cite{IO2011a}, which leads to the QTT---quantized tensor train
decomposition) to such an array can lead to a reduction in complexity
and number of parameters.

The number of parameters in a decomposition is related to the
\emph{rank} of the decomposition---i.e., the generalization of matrix rank to
multi-dimensional arrays.
Having a priori knowledge that a function of interest, e.g., the solution
to a partial differential equation, can be approximated by a low rank tensor
decomposition, allows for the application of tensor-structured algorithms 
that avoid working with full $2^\ell \times 2^\ell \times 2^\ell$ arrays of coefficients.

In particular, here we consider functions in weighted Sobolev spaces with 
radial weights and analytic- or Gevrey-type control of weighted semi-norms.
Such functions arise in a variety of scientific applications: 
nonlinear Schr\"odinger equations (e.g. \cite{Cances2006,Cances2010} and the references there), 
Hartree-Fock and density functional theory equations, 
continuum models of point defects \cite{OrtLusk},
blowup solutions in evolution equations with critical nonlinearity 
(e.g. \cite{BlowUp95} and the references there)
to name but a few.

The main result of the present paper is \emph{exponential convergence of
tensor-structured approximations of point singularities in ${\mathbb R}^3$},
ie., they admit tensor ranks bounded polylogarithmically with respect
to the accuracy $\epsilon\in(0,1)$ of the approximation, measured in the Sobolev space $H^1$. 
%This leads to the overall exponential convergence of the tensor-structured
%approximation with respect to the effective number of parameters of the tensor decomposition.
 
An auxiliary result of independent interest is the 
exponential convergence of $hp$-finite element (FE) approximations for 
the class of functions considered.
Due to the piecewise-polynomial structure of $hp$-FE approximants, we
can obtain their quantized representations with exact rank bounds that depend
only on the dimensions of $hp$-spaces.
This, in turn, leads to the desired rank bounds of the functions of interest.

One of the advantages of using quantized tensor decompositions---compared with
the direct application of $hp$-FE approximations---is the relative ease of implementation.
The adaptation of the number of parameters in the decomposition to the
  approximated function is based on well-known numerical linear algebra tools
such as QR and SVD decompositions.
Moreover, there exist open source codes with the implementation of 
basic linear algebra operations including solution of linear systems, 
which can be used independently of a particular application.

Note also that we do not need to know a priori the type and exact location 
of the singularity of the solution to solve PDEs in quantized tensor-structured formats.
The nonlinear structure of the decomposition allows for an ``automatic'' adaptation of the
tensor compressed representation to the regularity of the function.
This is by contrast to $hp$ methods, where mesh and polynomial degree
refinements are programmed explicitly depending on the type of singularity.
Furthermore, while the mesh of an $hp$ space has to be constructed so that
the refinement happens towards the singular point, this a priori knowledge is
not necessary in the computation of quantized tensor-structured representation.

%%%%%%%%%%%%%%%%%%%%%%%%%%%%%%%%%%%%%%%%%%%%%%%%%%%%%%%%%%%%%%%%%%%%%%%%%%%%%%%%%%%%%
\subsection{Tensor Structured Function Approximation}
\label{sec:TTFnAppr}
%%%%%%%%%%%%%%%%%%%%%%%%%%%%%%%%%%%%%%%%%%%%%%%%%%%%%%%%%%%%%%%%%%%%%%%%%%%%%%%%%%%%%
With the availability of efficient numerical realizations
of tensor-structured numerical linear algebra, 
a new perspective has been opened towards 
\emph{computational function approximation}.
Here, 
one compresses arrays of function values
in tensor formats; early work in this 
direction is \cite{TT03}, and \cite{BKTensBook2018}
contains a bibliography with a large list of 
ensuing developments based on this idea.
An (incomplete) list of references contains
\cite{Khoromskij:2011:QuanticsApprox,VKBK2012,BK+IOQuantColloc,IOConstRepFns,NouyNM2019}
where tensor rank bounds for specific functions 
have been obtained, both analytically and computationally,
in the so-called 
\emph{quantized tensor train (QTT) format}.
QTT formatted numerics for electron structure computations
were presented in \cite{KKS2011}.

Subsequently, and more directly related to the present work,
rather than rank bounds for individual functions,
\emph{
tensor rank bounds for solution classes of elliptic PDEs} 
in one and two spatial dimensions
were obtained in \cite{K15_2125,Kazeev2018,KORS17_2264}.
In \cite{Kazeev2018}, in particular, it was 
proved first that functions in 
countably normed, analytic function classes in polygons 
$D\subset  {\mathbb R}^2$
admit QTT structured tensor approximations 
with tensor ranks bounded polylogarithmically 
in terms of the approximation accuracy $\varepsilon$.
The key mathematical argument in the references cited above is 
based on \emph{analytic regularity results} for solutions of 
elliptic PDEs in polytopal domains. 
Such regularity results, implying solutions belong to countably normed spaces,
have been obtained in the past two decades for several broad classes
of (boundary value and eigenvalue problems of) elliptic PDEs, 
in \cite{Babuska1989,MS19_2743,Maday2019}.
%%%%%%%%%%%%%%%%%%%%%%%%%%%%%%%%%%%%%%%%%%%%%%%%%%%%%%%%%%%%%%%%%%%
\subsection{Problem Formulation}
\label{sec:problem-notation}
%%%%%%%%%%%%%%%%%%%%%%%%%%%%%%%%%%%%%%%%%%%%%%%%%%%%%%%%%%%%%%%%%%%
The tensor-formatted function approximation 
considered in this paper aims at establishing 
tensor rank bounds for functions in certain classes 
of locally smooth functions that admit a point singularity.
In this paper, we confine ourselves to the case
that the function under consideration admits 
singular support consisting only of one isolated point 
(we therefore speak of ``point singularities'').
Naturally, functions whose
singular support comprises of a finite number
of \emph{well-separated points} can equally be
approximated in the tensor-formats discussed here, 
with the same tensor-rank bounds, by a localization
and superposition argument. 
%We shall briefly indicate this in Section~\ref{sec:Concl}.

Weighted Sobolev spaces for functions 
with isolated point singularities have been
introduced for the analysis of elliptic problems in polygonal domains, see
\cite{Kondratev1967}, since they allow for the extension of classical
elliptic regularity theory to domains with corners.
For an overview of regularity results for elliptic boundary value problems
in conical domains, in weighted Sobolev spaces, we refer to the monographs 
\cite{Grisvard1985,Kozlov1997,Kozlov2001,Mazya2010}.

For elliptic boundary value problems in three space dimensions, 
weighted Sobolev spaces that accommodate isolated point singularities 
have also proven important in the mathematical regularity analysis 
of problems with singular potentials, 
such as electron structure calculations in quantum physics and quantum chemistry, 
see, e.g., \cite{Flad2008,FHHOS2009,FHHOS2009b}.

When a function is regular in weighted Sobolev spaces---specifically, when
analytic-type bounds can be derived on the norms of its
derivatives---piecewise polynomial approximations can be constructed, 
for example by $hp$ finite elements which converge 
exponentially (in terms of the number of parameters)
\cite{Scherer80,Guo1986a,Guo1986b,Schotzau2015,Schotzau2017}. 
This suggests the existence of an underlying low-rank structure
in suitable tensor formats;
for this reason, we are here interested in the derivation of rank
bounds for functions that belong to weighted analytic- and Gevrey- type classes.

A theory of analytic regularity in weighted Sobolev spaces has been developed
for several classes of important physical problems and we mention an incomplete
  list.
  Solutions to scalar elliptic problems with
  constant coefficients belong to analytic-type weighted spaces
  \cite{Costabel2012a,Costabel2014}, as do
  the flow and pressure obtained with the Stokes \cite{Guo2006a} and
  Navier-Stokes \cite{MS19_2743} equations in polygons. Furthermore,
  eigenfunctions to three-dimensional linear \cite{Maday2019} and nonlinear
  \cite{Maday2019b} Schrödinger equations are weighted analytic.
  In quantum chemistry, the
  wave functions computed with the non-relativistic Hartree-Fock models for
  electronic structure calculations are
  also analytic in weighted Sobolev spaces \cite[Section 7.4]{Marcati_thesis},
  \cite{ChemistryPrimer},
  with point singularities at the nuclei. We refer to Section
  \ref{sec:model-problems} for some explicit examples in this sense.
  Other instances of the occurrence of point singularities in otherwise
  smooth solutions comprise general relativity (see, e.g., \cite{visser2009kerr,Boyer1967}
  and the references there) and solutions of parabolic evolution equations 
with critical nonlinearity (see, e.g., \cite{BlowUp95} and the references there).
The results of the present work apply to all the problems cited above, whose
  solutions are weighted analytic, in the spaces that we detail in Section
  \ref{sec:GevAnFct} below.

We consider the following setting for 
quantized, tensor train (TT)-formatted function approximation in $Q=(0,1)^3$,
with one point singularity at the origin, where the functions
belong to countably-normed, weighted Sobolev spaces, where the 
weights are powers of $r = |x|$,
the euclidean distance of the point $x\in Q$ from the origin.
%%%%%%%%%%%%%%%%%%%%%%%%%%%%%%%%%%%%%%%%%%%%%%%%%%%%%%%%%%%%%%%%%%%%%%%%%%%%%%%%%%%%%%%%%%%%
\subsubsection{Kondrat'ev type weighted Sobolev spaces}
\label{sec:weighted}
%%%%%%%%%%%%%%%%%%%%%%%%%%%%%%%%%%%%%%%%%%%%%%%%%%%%%%%%%%%%%%%%%%%%%%%%%%%%%%%%%%%%%%%%%%
For integer $s\geq 0$, a real parameter $\gamma\in \mathbb{R}$, 
and summability exponent $1\leq q < \infty$, 
we introduce the \emph{homogeneous weighted Sobolev spaces} 
$\cK^{s,q}_\gamma(Q)$. 
Given the seminorm
\begin{equation}
  \label{eq:Kseminorm}
|w|_{\cK^{s,q}_\gamma(Q)} 
= 
\left( \sum_{\alpham= s} \|r^{\alpham -\gamma}\dalpha w\|^q_{L^q(Q)}  \right)^{1/q},
\end{equation}
so that the spaces $\cK^{s,q}_\gamma(Q)$ are normed by
\begin{equation*}
\|w\|_{\cK^{s,q}_\gamma(Q)}  
= 
\left( \sum_{k = 0}^{ s} |w|^q_{\cK^{k,q}_\gamma(Q)} \right)^{1/q}.
\end{equation*}
Our focus will be mostly on \emph{non-homogeneous weighted Sobolev spaces}
$\cJ^{s,q}_\gamma(Q)$, with norm given by (remark the different weight exponent)
\begin{equation*}
\| w \|_{\cJ^{s,q}_\gamma(Q)}  
= 
\left( \sum_{\alpham\leq s} \|r^{s-\gamma}\dalpha w\|^q_{L^q(Q)} \right)^{1/q}.
\end{equation*}
In the following, we will always consider the case where $q=2$, 
$0<\gamma-3/2<1$, and $s>\gamma-3/2$.
Under those hypotheses, as shown in \cite[Proposition 3.18]{Costabel2010a}, 
the above norm is equivalent to
\begin{equation}
    \label{eq:Jspace}
\| w \|_{\cJ^{s,2}_\gamma(Q)}  
\simeq 
\left( \|w \|^2_{L^2(Q)} 
+ 
\sum_{k = 1}^{s} |w|^2_{\cK^{k, 2}_\gamma(Q)}
\right)^{1/2}.
\end{equation}
Non-homogeneous spaces allow for functions with non trivial Taylor expansion 
at the
singularity and have been used, for this reason, in the analysis of 
problems in non smooth domains with Neumann boundary conditions and
of elliptic problems with singular potentials. For a thorough analysis of the
relationship between homogeneous and non homogeneous spaces, we refer the reader
to \cite{Kozlov1997} and \cite{Costabel2010a}.
%%%%%%%%%%%%%%%%%%%%%%%%%%%%%%%%%%%%%%%%%%%%%%%%%%%%%%%%%%%%%%%%%%%%%%%%%%%%%%%%%%%%%%%%%%5
\subsubsection{Gevrey and analytic function classes}
\label{sec:GevAnFct}
%%%%%%%%%%%%%%%%%%%%%%%%%%%%%%%%%%%%%%%%%%%%%%%%%%%%%%%%%%%%%%%%%%%%%%%%%%%%%%%%%%%%%%%%%%%
We denote the weighted Kondrat'ev type spaces of infinite regularity by
\begin{equation*}
\cK^{\infty,q}_\gamma(Q) = \bigcap_{s\in\mathbb{N}} \cK^{s,q}_\gamma(Q).
\end{equation*}
Furthermore, for constants $C, A>0$ and $\gev\geq 1$, 
we introduce the \emph{countably normed, homogeneous weighted Gevrey-type} 
(analytic-type when $\gev=1$) 
class
\begin{equation*}
%\label{eq:Kanalytic}
\cK^{\varpi,q}_\gamma(Q; C, A, \gev) 
= 
\left\{ v\in \cK^{\infty,q}_\gamma(Q): \left| v\right|_{\cK^{s,q}_\gamma(Q)}
\leq 
C A^{s}(s!)^\gev, \, \text{for all }s\in\mathbb{N}_0 \right\}.
\end{equation*}
The \emph{countably normed, non-homogeneous weighted classes} 
$\cJ^{\infty, q}_\gamma(Q)$ 
are then defined as in the homogeneous case, 
while the \emph{non-homogeneous Gevrey/analytic classes} 
are given by
\begin{equation}
\label{eq:Janalytic}
\cJ^{\varpi,q}_\gamma(Q; C, A, \gev) 
= 
\left\{ v\in \cJ^{\infty,q}_\gamma(Q): | v |_{\cK^{s,q}_\gamma(Q)}
\leq 
C A^{s}(s!)^\gev, \text{integer }s>\gamma-3/2 \right\}.
\end{equation}
We write $\cK^{s,2}_\gamma(Q) = \cK^s_\gamma(Q)$ and
$\cJ^{s,2}_\gamma(Q) = \cJ^s_\gamma(Q)$; similarly we omit the 
summability exponent $q$ when it equals $2$
in the notation for the weighted Gevrey classes.
%%%%%%%%%%%%%%%%%%%%%%%%%%%%%%%%%%%%%%%%%%%%%%%%%%%%%%%%%5
\subsubsection{Model problems}
\label{sec:model-problems}
%%%%%%%%%%%%%%%%%%%%%%%%%%%%%%%%%%%%%%%%%%%%%%%%%%%%%%%%%5
We illustrate the scope of problems by 
listing several concrete boundary-value and eigenvalue problems 
whose solutions are known to belong
to the weighted analytic classes 
$\cK^\varpi_\gamma(\Omega)$ and $\cJ^\varpi_\gamma(\Omega)$.
Although the focus here is on three-dimensional problems, 
we start by considering
a polygon $\Omega\subset {\mathbb R}^2$ with $n\geq 3$
straight sides and corners $\fc_i$, $i=1,
\dots, n$. In this setting, the space $\cK^\varpi_\gamma(\Omega)$ contains the
\emph{corner weight function} $r_P = \prod_{i=1}^n | x - \fc_i|$, 
i.e., the seminorm \eqref{eq:Kseminorm} is replaced by
\[
| v |_{\cK^{s, q}_\gamma(\Omega)} 
= 
\left( \sum_{\alpham= s} \|r_P^{\alpham -\gamma}\dalpha w\|^q_{L^q(\Omega)}  \right)^{1/q}.
\]
Then, given an analytic (in $\overline{\Omega}$) 
external force field $f$, the \emph{Stokes equations}
\begin{equation*}
  -\nu \Delta u + \nabla p = f  \text{ in }\Omega,
  \qquad
  \nabla \cdot u = 0  \text{ in }\Omega
\end{equation*}
and the viscous, incompressible \emph{Navier-Stokes equations}
\begin{equation}
  \label{eq:NS}
  -\nu\Delta u + (u\cdot\nabla) u  + \nabla p = f  \text{ in }\Omega,
  \qquad
  \nabla \cdot u = 0  \text{ in }\Omega
\end{equation}
with homogeneous Dirichlet (``no-slip'') 
boundary conditions have been shown in \cite{MS19_2743,Guo2006a} 
to admit solutions in 
$\cK^\varpi_\gamma(\Omega)$ with $\gamma>3/2$. 
Specifically, for the nonlinear boundary value problem \eqref{eq:NS} 
we require a ``small data assumption'' which is well-known to 
ensure uniqueness of Leray-Hopf solutions, 
see, e.g., \cite[Chapter IV, Theorem 2.2]{Girault1986a}.
See Remark \ref{rmk:QTTcrnr} for further comments on 
the implication of the present work on two-dimensional problems.

In the three-dimensional setting, energy minimization problems in quantum
physics/chemistry can be transformed into eigenvalue problems whose solutions
are in the weighted analytic class \eqref{eq:Janalytic}. 
We consider here a set of isolated point singularities situated at 
$n$ nuclei in positions $R_i\in \mathbb{R}^3$,  $i=1, \dots,n$,
and function spaces with weight function $r$ such that 
$r\simeq |x-R_i|$ in the vicinity of each $R_i$, 
and $r\simeq 1$ far from all singularities resp. all nuclei.

A first example is given by a \emph{nonlinear Schrödinger equation} with polynomial nonlinearity.
Consider a compact domain without
boundary $\Omega$ (e.g., a periodic unit cell) and a potential $V$ 
such that there exists $\beta < 2$ and a constant $A_V>0$ such that 
\begin{equation*}
\forall \alpha\in {\mathbb N}_0^3: \qquad 
\|r^{\beta+\alpham} \dalpha V \|_{L^\infty(\Omega)} \leq A_V^{\alpham+1} \alpham! \;.
\end{equation*}
Then, 
the eigenfunction $u$ corresponding to the smallest eigenvalue (i.e., the ``ground state'')
of the nonlinear Schrödinger equation
\begin{equation}
  \label{eq:NLSch}
  -\frac{1}{2}\Delta u + V u + |u|^2 u = \lambda u, \qquad \| u \|_{L^2(\Omega)}=1
\end{equation}
is in $\cJ^\varpi_\gamma(\Omega)$ for some $\gamma>3/2$, see \cite[Section 7.3]{Marcati_thesis}.
Note that \eqref{eq:NLSch} is the Euler-Lagrange equation of the minimization problem
\begin{equation*}
 \inf\left\{ \int_{\Omega} |\nabla v|^2 + Vv^2 + \frac{1}{2} v^4,\, v\in H^1(\Omega), \,\|v\|_{L^2(\Omega)}=1 \right\}.
\end{equation*}

As a second example we consider the \emph{Hartree-Fock equation}. 
Let
$V_\mathrm{C}$ be the potential of the Coulomb interaction exerted on
  electrons by nuclei with charge $Z_i$ assumed to be pointlike and situated
at positions $R_i\in {\mathbb R}^3$, $i=1, \dots, n$,
i.e.,
\begin{equation*}
  V_{\mathrm{C}}(x) = -\sum_{i=1}^n \frac{Z_i}{|x-R_i|}.
\end{equation*}
The Hartree-Fock model consists in finding the smallest 
$N$ eigenvalues $\epsilon_i$ and 
the corresponding $L^2({\mathbb R}^3)$-orthonormal 
eigenfunctions $\psi_i$, $i=1, \dots, N$, 
such that 
\begin{equation}
  \label{eq:HF}
  \left( -\frac{1}{2}\Delta + V_{\mathrm{C}} \right)\psi_i 
+ \left( \frac{1}{|x|}\star\rho_\Psi \right)\psi_i 
- \sum_{j=1}^N\left(  \frac{1}{|x|}\star (\psi_j\psi_i)\right)\psi_j 
= \epsilon_i \psi_i,\quad i=1, \dots, N
\end{equation}
with $\rho_\Psi = \sum_{i=1}^N\psi_i^2$.
Then, under some conditions on the potential $V_C$ so that the solution exists \cite{Lieb1977}, 
the eigenfunctions are weighted analytic: 
\[\psi_i\in \cJ^\varpi_\gamma(\mathbb{R}^3),\qquad i=1, \dots, N,\]
see \cite[Section 7.4]{Marcati_thesis}.
Problem \eqref{eq:HF} is the Euler-Lagrange equation of the minimization problem
(see \cite[Section 9]{ChemistryPrimer})
\begin{equation*}
  %\label{eq:HF-min}
  \inf\left\{ E^{\mathrm{HF}}(\psi_1, \dots, \psi_N), \psi_i\in H^1(\mathbb{R}^3): \int_{\mathbb{R}^3}\psi_i\psi_j = \delta_{ij} \right\},
\end{equation*}
where
\begin{multline*}
  E^{\mathrm{HF}}(\psi_1, \dots, \psi_N) = \sum_{i=1}^N \int_{\mathbb{R}^3}|\nabla \psi_i|^2
  + \int_{\mathbb{R}^3} V \rho_{\Psi}
  + \frac{1}{2}\int_{\mathbb{R}^3} \rho_\Psi(x)\left( \frac{1}{|x|}\star\rho_\Psi \right)
  \\
 - \frac{1}{2} \int_{\mathbb{R}^3}\int_{\mathbb{R}^3} \frac{\tau_{\Psi}(x,y)}{|x-y|},
\end{multline*}
with $\tau_\Psi(x,y) = \sum_{i=1}^N\psi_i(x)\psi_i(y)$.
\begin{remark}[Near-Singularity]\label{rmk:NearSing}
    While functions of the form ($r\in \mathbb{R}_+, \omega\in \mathbb{S}_2$ spherical coordinates)
    \begin{equation}
      \label{eq:r+a}
      u_a(r, \omega) = (r^2+a^2)^{\beta/2} v(\omega), \qquad v \text{ analytic in }\mathbb{S}_2
    \end{equation}
    are, for $a\neq 0$ and for $\beta>0$, formally (mathematically) smooth,
    their behavior approaches that of functions with point singularities when $|a| \ll 1$. 
    Specifically, if $|a|\leq a_{\max{}}$, there exist positive constants
    $C$ and $A$ \emph{independent of $a$} such that $u_a \in
    \cJ^\varpi_\gamma(Q; C, A, 1)$ for $\gamma<\beta+3/2$; 
    hence, the bounds obtained in the present
    paper allow for the derivation of rank bounds for the  
    quantized tensor-formatted approximations considered,
    which are \emph{uniform as the parameter $a\downarrow 0$} 
    for functions of the form \eqref{eq:r+a} .

The same remark applies to certain \emph{merging singularities}
as also arise, for example, in binary star or black hole models.
Consider e.g. two nuclei situated at locations
$R_1 = -\varepsilon e_1$, $R_2 = \varepsilon e_1$ in 
${\mathbb R}^3$ at distance $2\varepsilon$ for small $\varepsilon > 0$.
Denoting by $r_i = |x-R_i|$, $i=1,2$, 
and $r = |x|$, 
we find 
$ v(x) = r_1^2+r_2^2 = 2(r^2 + \varepsilon^2)$
i.e., 
once more a function of the above form with $a = \varepsilon$.
\end{remark}
%%%%%%%%%%%%%%%%%%%%%%%%%%%%%%%%%%%%%%%%%%%%%%%%%%%%%%%%%%%%%%%%%%%%%%%%%%%%%%%%%%%%%%%%%%%
\subsection{Contributions}
\label{sec:contr}
%%%%%%%%%%%%%%%%%%%%%%%%%%%%%%%%%%%%%%%%%%%%%%%%%%%%%%%%%%%%%%%%%%%%%%%%%%%%%%%%%%%%%%%%%%%
In the present work, the following novel mathematical results are obtained.
First, we analyze approximation rates of tensor-structured approximations 
of smooth functions with isolated point singularities. 
As compared to exponential convergence results results for \emph{analytic} 
functions with point singularities, we here establish exponential convergence
of $hp$-finite element (FE) approximations on geometric meshes of axiparallel quadrilaterals resp.
hexahedra analogous to \cite{Feischl2018} also for Gevrey regular functions.

We then address tensor-formatted approximations.
Generalizing results also in two variables, 
in the present paper we extend the analysis in \cite{Kazeev2018} to 
quantized, TT-structured function approximation
of functions from the countably weighted, Gevrey-type classes 
$\cK^{\varpi,q}_\gamma(Q; C, A, \gev)$ 
and
$\cJ^{\varpi,q}_\gamma(Q; C, A, \gev)$ 
as defined above.

The corresponding results in three spatial variables are novel.
They also extend the QTT rank bounds in \cite{Kazeev2018} 
to Gevrey-$\gev$ regular functions (see Remark \ref{rmk:QTTcrnr}). 
They also constitute a building block for the derivation of
corresponding QTT rank bounds for edge and face singularities
in three space dimensions, which we do not detail here.

In particular, 
we prove in three physical variables
for analytic (resp. Gevrey) functions with point singularities, 
for the classical tensor format asymptotic upper bounds on quantized tensor ranks 
at prescribed accuracy $\varepsilon$ 
which are better than the corresponding bounds for the transposed TT format 
introduced in \cite{Kazeev2018} (in two dimensions).

We show numerical results indicating the correctness of the presently obtained results,
and also strongly suggesting that similar ranks are achieved in tensor-formatted
PDE solvers, provided the PDE solutions belong to the countably normed classes
introduced in Section~\ref{sec:GevAnFct}.
%%%%%%%%%%%%%%%%%%%%%%%%%%%%%%%%%%%%%%%%%%%%%%%%%%%%%%%%%%%%%%%%%%%%%%%%%%%%%%%%%%%%%%%%%%%
\subsection{Structure of this paper}
\label{sec:Struct}
%%%%%%%%%%%%%%%%%%%%%%%%%%%%%%%%%%%%%%%%%%%%%%%%%%%%%%%%%%%%%%%%%%%%%%%%%%%%%%%%%%%%%%%%%%%
In Section~\ref{sec:tensor}, we review the definitions and notation
of quantized, tensor structured function approximation which are
to be employed throughout the remainder of the article,
extending the concepts of \cite{IOConstRepFns}.
In Sections~\ref{sec:TT}-\ref{sec:TuckQTT}, 
in particular, we introduce the \emph{tensor train} (TT), 
the \emph{quantized TT format} (QTT), 
\emph{transposed quantized TT format} (QT3)
and the \emph{Tucker quantized TT format} (TQTT),
some of which allow to prove better rank bounds 
on functions with point singularities.

Section~\ref{sec:mesh} introduces tools from numerical analysis which
we require in the arguments for the TT-rank bounds for function approximation.
Section~\ref{sec:VirtUnifSpc} introduces in particular
the notion of \emph{``uniform background mesh''} (never directly accessed in the
QTT formats) which is the basis for all quantized TT function representations.
Section~\ref{sec:hp} recapitulates several notions and auxiliary results from
the theory of so-called $hp$-approximation from \cite{Schwabphp98,Schotzau2015,Schotzau2017}.
Section~\ref{sec:interp-error} introduces a combined (quasi) interpolation projector,
which was introduced in \cite{Kazeev2018} (in two dimensions) and which is crucial
in establishing the rank bounds.
Section~\ref{sec:analytic-approximation} then contains statements and proofs of
the main results of the present paper: 
tensor-rank bounds for generic functions in the various
countably-normed classes introduced in Section~\ref{sec:problem-notation} above.
These bounds are obtained for functions with the singularity at a corner of
the domain; they are extended to the case of an internal singular point (and
to a patchwise formulation that allows for more complex domains) in Section
\ref{sec:ExtRkBdIntSing} in the appendix.

Section~\ref{sec:NumExp} presents detailed numerical experiments 
which exhibit actual TT rank bounds in the various formats for model singular
functions in three space dimensions.
The Section~\ref{sec:Concl} provides a brief summary of the main results,
and possible further research directions.
The Appendix, Section~\ref{sec:gevrey} contains (novel) auxiliary results
on exponential rates of convergence of $hp$-approximations for 
Gevrey-regular functions in $\mathbb{R}^3$ with point singularities, 
generalizing \cite{Feischl2018} to axiparallel
geometric meshes of hexahedra with $1$-irregular edges and faces.
%
%%%%%%%%%%%%%%%%%%%%%%%%%%%%%%%%%%%%%%%%%%%%%%%%%%%%%%%%%%%%%%%%%%%%%%%%%%%%%%%%%%%%%%%%%%%%%%
\section{Tensor structured representations}
\label{sec:tensor}
%%%%%%%%%%%%%%%%%%%%%%%%%%%%%%%%%%%%%%%%%%%%%%%%%%%%%%%%%%%%%%%%%%%%%%%%%%%%%%%%%%%%%%%%%%%%%%
The mathematical issue in tensor-formatted function approximation
consists in finding a compressed representation/approximation 
of three-way tensors
\[A\in \mathbb{R}^{2^\ell\times2^\ell\times2^\ell},\]
for $\ell \in \mathbb{N}$. 
All techniques that we examine are based on the \emph{Quantized Tensor Train} (QTT) 
representation, see e.g. 
\cite{Kazeev2018,KORS17_2264,Oseledets:2010:QTT,Khoromskij:2011:QuanticsApprox,IOConstRepFns} 
and the references there. 
In particular, we will analyze three tensor compressed
representations, that we call here \emph{(classic) QTT}, \emph{transposed QTT} (QT3),
and \emph{Tucker-QTT} (TQTT) representation, respectively. 
The difference between these schemes lies
in the arrangement of the three physical dimensions of the tensor $A$ 
in the corresponding TT format 
in the following, after a brief introduction of QTT representations, 
we detail the three formats mentioned.
%%%%%%%%%%%%%%%%%%%%%%%%%%%%%%%%%%%%%%%%%%%%%%%%%%%%%%%%%%%%%%%%%%%%%%%%%%%%%%%%%%%%%%%%%%%
\subsection{Notation}
\label{sec:Notat}
%%%%%%%%%%%%%%%%%%%%%%%%%%%%%%%%%%%%%%%%%%%%%%%%%%%%%%%%%%%%%%%%%%%%%%%%%%%%%%%%%%%%%%%%%%%
Throughout, we adopt the following notation, from \cite{Kazeev2018}.
Given $n\in \mathbb{N}$ indices $i_1, \dots, i_n$ 
such that $i_j\in \{0, \dots, k_j-1\}$ for all $j=1, \dots, n$, we write
\begin{equation*}
  \overline{i_1\dots i_n} = i_1 \prod_{j=2}^n k_j + i_2 \prod_{j=3}^nk_j + \dots + i_n.
\end{equation*}
Hereafter, 
by \emph{tensor} we will mean multi-dimensional array.
Furthermore, 
for an axiparallel $d$-dimensional ($d\leq 3$) subset $K\in Q$, 
the space $\mathbb{Q}_p(K)$ is the tensor product space of 
$d$-variate polynomials in $K$ of maximum polynomial degree $p$ in each variable.
Furthermore, we will indicate by a colon~``:'' a whole slice of a tensor.
For example, given a four-dimensional tensor $A\in \mathbb{R}^{n_1\times n_2\times n_3\times n_4}$
with entries $a_{i,j,k,l}$, we will write
\begin{equation*}
A_{i, :, :, l} = \{a_{i, j, k, l}\}_{j=1, \dots, n_2, k=1, \dots, n_3}\in \mathbb{R}^{n_2, n_3}.
\end{equation*}
%%%%%%%%%%%%%%%%%%%%%%%%%%%%%%%%%%%%%%%%%%%%%%%%%%%%%%%%%%%%%%%%%%%%%%%%
\subsection{Tensor Train (TT) Format}
\label{sec:TT}
%%%%%%%%%%%%%%%%%%%%%%%%%%%%%%%%%%%%%%%%%%%%%%%%%%%%%%%%%%%%%%%%%%%%%%%%
Tensor Trains (TT)~\cite{IO2011a}, also known as Matrix Product States (MPS) 
in the computational physics community \cite{Sch2011}, 
provide an efficient way to represent high-dimensional tensors, 
provided these tensors have an underlying low-rank structure. 
Let $d \gg 1$, and consider the $d$-dimensional tensor
\begin{equation} \label{eq:TensB}
  B \in \mathbb{R}^{n_1\times \cdots\times n_d}.
\end{equation}
The \emph{Tensor Train representation} of the $d$-variate tensor 
$B$ in \eqref{eq:TensB} 
is given in terms of the \emph{core tensors}\footnote{
The cores $U^k$ can be naturally considered as three-dimensional arrays
$\widetilde U^{k}\in \mathbb{R}^{r_{i-1}\times n_i\times r_i}$ 
so that 
$\left(U^k(i_k)\right)_{\alpha_{k-1},\alpha_k} 
 = 
 \widetilde U^{k}_{\alpha_{k-1},i_k,\alpha_k}$, $\alpha_{k-1}=1,\dots,r_{k-1}$, $\alpha_{k}=1,\dots,r_k$.
Using the three-variate arrays notation, 
the TT decomposition of $B$ can be written as
\begin{equation}\label{eq:tt_sumsrepr}
	B_{i_1, \dots, i_d} 
        = 
        \sum_{\alpha_0=1}^{r_0}\dots \sum_{\alpha_d=1}^{r_d} 
        \widetilde U^1_{\alpha_0,i_1,\alpha_1} \cdots \widetilde U^d_{\alpha_{d-1},i_d,\alpha_d},
\end{equation}
which is also used in~\cite[Eq.~(1.3)]{IO2011a}.
For clarity of presentation, 
we use the representation~\eqref{eq:TT-def}, which is more compact
than~\eqref{eq:tt_sumsrepr} (and equivalent to it).
We also do not distinguish between the mappings $U^k$ and the three-dimensional arrays $\widetilde U^{k}$.
}
\begin{equation*}
  %\label{eq:TT-core}
  U^k : \{0, \dots, n_k-1\} \to \mathbb{R}^{r_{k-1}\times r_k} \qquad k=1, \dots, d
\end{equation*}
where $r_k\in \mathbb{N}$ (with the restriction $r_0=r_d=1$) 
and such that~\cite[Eq.~(1.2)]{IO2011a}
\begin{equation}
  \label{eq:TT-def}
  B_{i_1, \dots, i_d} = U^1(i_1) \cdots U^d(i_d).
\end{equation}
Suppose for ease of presentation that $r_i = r$ and $n_i = n$ for all $i$.
Then, the TT representation \eqref{eq:TT-def} has
\begin{equation*}
  \Ndof = \mathcal{O}(dnr^2)
\end{equation*}
parameters. 
The TT format is therefore an efficient decomposition if a 
$d$-way tensor can be written as a tensor train with \emph{low ranks} $r_i$.
Due to the equality in \eqref{eq:TT-def} the representation we have introduced
is an \emph{exact} TT representation; in practice, a matrix may not admit an
exact low rank TT representation, but a low rank \emph{approximation} could
instead be available.
\subsection{Rank bound analysis of TT representations}
\label{sec:RkBdAnTT}
To examine the issue of low rank approximation of high-dimensional tensors, 
we require the concept of \emph{unfolding matrices} (``unfoldings'' for short), 
used to derive \emph{rank bounds} on the TT representation of a tensor.
\begin{definition}[Unfolding matrix]
 Let $d\in \mathbb{N}$ and $n_i\in \mathbb{N}$ for $i=1, \dots, d$. Given a
 tensor $B\in\mathbb{R}^{n_1\times \cdots\times n_d}$, we define for all $q = 1,
 \dots, d-1$ its \emph{unfolding matrices} $B^{(q)}\in \mathbb{R}^{n_1\cdots n_q \times n_{q+1}\cdots n_d}$ as
 \begin{equation*}
   %\label{eq:unfolding}
   B^{(q)} _{\overline{i_1\dots i_q}, \overline{i_{q+1}\dots i_d}}= B_{i_1, \dots, i_d},
    \qquad \text{ for all }i_k = 1, \dots, n_k\text{ and }k=1, \dots, d
 \end{equation*}
i.e., the matrix with row index given by the concatenation of the first $q$
indices, and column index given by the concatenation of the remaining ones.
\end{definition}
In the case that the unfolding matrices of a tensor 
can be approximated by low-rank matrices,
then a low rank TT approximation exists. 
This is made precise in the following result.
\begin{proposition}{\cite[Theorem 2.2]{IO2011a}}
  Let $B\in \mathbb{R}^{n_1\times\dots\times n_d}$ such that its unfolding
  matrices $B^{(q)}$ can be decomposed as 
  \begin{equation*}
  B^{(q)} = R^q + E^q, 
  \quad \rank R^q = r_q,
  \quad \| E^q\|_F \leq \epsilon_q, 
  \quad \text{for all }q=1, \dots, d-1.
  \end{equation*}
There exists a tensor $C$ with 
TT representation \eqref{eq:TT-def} and TT-ranks $r_q$ 
such that
  \begin{equation*}
    \| B-C\|_F^2 \leq \sum_{q=1}^{d-1} \epsilon_q^2.
  \end{equation*}
\end{proposition}
The above theorem includes as a sub-case the rank bound of exact
TT-representation, 
by affirming the existence of an exact TT-rank $r_q$ representation
of a tensor with unfolding matrix rank bounded by $r_q$, $q=1,\dots,d-1$.
%%%%%%%%%%%%%%%%%%%%%%%%%%%%%%%%%%%%%%%%%%%%%%%%%%%%%%%%%%%%%
\subsection{Quantized Tensor Train (QTT) format in one physical dimension}
\label{sec:QQT1d}
%%%%%%%%%%%%%%%%%%%%%%%%%%%%%%%%%%%%%%%%%%%%%%%%%%%%%%%%%%%%%
We introduce QTT representations in the simplified setting 
of QTT approximation of vectors
\begin{equation}
  \label{eq:vector}
  v\in \mathbb{R}^{2^\ell},
\end{equation}
for $\ell\in\mathbb{N}$. 
Generalizations to the multi-dimensional
case will be the subject of the next sections.

The QTT decomposition introduced in~\cite{Oseledets:2010:QTT,Khoromskij:2011:QuanticsApprox}
extends the use of the TT approximation to the case of \emph{low-dimensional} tensors with 
a large number of elements. 
To do so, 
the low-dimensional tensor is reshaped into a high-dimensional one, 
which is subsequently TT-(re)approximated.
Applied to the vector in \eqref{eq:vector}, 
algorithmically this is achieved by reshaping it into the
$\ell$-dimensional tensor $\widetilde{v}$ 
such that
\begin{equation*}
  \widetilde{v}_{i_1, \dots, i_\ell} = v_{\overline{i_1\dots i_\ell}},
\end{equation*}
where $i_k \in \{0,1\}$ for all $k=1, \dots, \ell$. 
The tensor $\widetilde{v}$ can then be represented in TT form. 
We formalize this representation in the following definition.
\begin{definition}[Univariate QTT decomposition]
\label{def:qtt-1d}
 Given $\ell\in \mathbb{N}$ and 
 a vector $v\in \mathbb{R}^{2^\ell}$,
 $v$ admits a \emph{QTT representation}
 with QTT ranks $r_0, \dots, r_\ell$ and QTT cores $U^i : \{0,1\}\to
 \mathbb{R}^{r_{i-1}\times r_i}$
 % , $i=1, \dots, \ell$
 if
 \begin{equation*}
   %\label{eq:QTT-1d}
   v_{\overline{i_1\dots i_\ell}} = U^1(i_1)\cdots U^\ell(i_\ell), \qquad \text{ for all }(i_1, \dots, i_\ell)\in \{0,1\}^\ell.
 \end{equation*}
\end{definition}
As before, the tensor cores 
can also be interpreted as three-way arrays in
$\mathbb{R}^{r_{i-1}\times 2 \times r_i}$.
%%%%%%%%%%%%%%%%%%%%%%%%%%%%%%%%%%%%%%%%%%%%%%%%%%%%%%%%%%%%%%%%%%%%%%%%%%%%%%%%%%
\subsection{Classic QTT format in three physical space dimensions} 
\label{sec:ClassQTT}
%%%%%%%%%%%%%%%%%%%%%%%%%%%%%%%%%%%%%%%%%%%%%%%%%%%%%%%%%%%%%%%%%%%%%%%%%%%%%%%%%%
The ``classic QTT'' format is the straightforward generalization of the 
univariate QTT format in Definition \ref{def:qtt-1d} 
to the multivariate case. 

In this way, 
a three-dimensional tensor 
$A\in \mathbb{R}^{2^\ell \times 2^\ell \times 2^\ell}$ is reshaped into the tensor 
\begin{equation}
  \label{eq:QTT-reshape}
  \widetilde{A}^{\qtt} \in \mathbb{R}^{\overbrace{\scriptstyle 2\times \cdots \times 2}^{3 \ell\text{ times }}}
  \quad 
  \text{ such that }\quad \widetilde{A}^{\qtt}_{i_1,\dots, i_\ell, j_1, \dots, j_\ell, k_1, \dots, k_\ell} 
  = A_{\overline{i_1\dots i_\ell}, \overline{j_1\dots, j_\ell}, \overline{k_1\dots k_\ell}}  
\end{equation}
for all $i_n, j_n, k_n \in \{0,1\}$, 
which is subsequently TT-decomposed. 
\begin{definition}[Classic QTT decomposition]
  \label{def:QTT}
  Given $A\in \mathbb{R}^{2^\ell \times 2^\ell \times 2^\ell}$ for an $\ell\in \mathbb{N}$, 
  we say that $A$ admits a \emph{classic QTT decomposition} with ranks 
  $r_0, \dots, r_{\ell}$, $s_0, \dots, s_\ell$, $t_0, \dots, t_{\ell}$
  and cores $U^1, \dots, U^\ell$, $V^1, \dots, V^\ell$, $W^1, \dots, W^\ell$ if
  \begin{equation}
    \label{eq:QTT-def}
    A_{\overline{i_1\dots i_\ell}, \overline{j_1\dots, j_\ell}, \overline{k_1\dots k_\ell}} 
    = 
    U^1(i_1)\cdots U^\ell(i_\ell)V^1(j_1)\cdots V^\ell(j_\ell)W^1(k_1)\cdots W^\ell(k_\ell)
  \end{equation}
  for all $i_n, j_n, k_n\in \{0,1\}$, and where
  \begin{equation}
    \label{eq:qtt-cores-def}
    U^n:\{0, 1\}\to \mathbb{R}^{r_{n-1}\times r_n}, \quad
    V^n:\{0, 1\}\to \mathbb{R}^{s_{n-1}\times s_n}, \quad
    W^n:\{0, 1\}\to \mathbb{R}^{t_{n-1}\times t_n},
  \end{equation}
  for $n = 1, \dots, \ell$. We have the restriction on the ranks
  \begin{equation*}
    %\label{eq:QTT-ranks}
    r_0 = t_\ell = 1, \quad r_{\ell} = s_0, \quad s_\ell = t_0.
  \end{equation*}
  \end{definition}
We denote by 
$\fT^{\qtt} : \mathbb{R}^{2^\ell\times 2^\ell \times 2^\ell} \to \mathbb{R}^{\overbrace{\scriptstyle 2\times
        \cdots \times 2}^{3 \ell\text{ times }}}$ the ``classic QTT''
    tensorization given by
    \begin{equation}
      \label{eq:QTT-tensorization}
      \fT^{\qtt}(A) = \widetilde{A}^{\qtt},
    \end{equation}
    with $\widetilde{A}^{\qtt}$ defined in \eqref{eq:QTT-reshape}.

The (classic) QTT decomposition is symbolically depicted
in tensor network format in Figure \ref{fig:networks}, panel (A).
%%%%%%%%%%%%%%%%%%%%%%%%%%%%%%%%%%%%%%%%%%%%%%%%%%%%%%%%%%%%
\subsection{Transposed order QTT format in three physical space dimensions}
\label{sec:QT3}
%%%%%%%%%%%%%%%%%%%%%%%%%%%%%%%%%%%%%%%%%%%%%%%%%%%%%%%%%%%%%
In the \emph{transposed order QTT} format (referred to as ``QT3'' format)
introduced first in \cite{Kazeev2018},
after reshaping the tensor $A$ as in \eqref{eq:QTT-reshape}, 
the indices from the different (three) physical
dimensions are regrouped together, resulting in a tensor
\begin{equation}
  \label{eq:transposedQTT-reshape}
   \widetilde{A}^{\qttt} \in \mathbb{R}^{\overbrace{\scriptstyle 8\times
        \cdots \times 8}^{\ell\text{ times }}}\quad 
        \text{ such that }\quad 
         \widetilde{A}^{\qttt}_{\overline{i_1j_1k_1},\dots, \overline{i_\ell j_\ell k_\ell}}
   = 
   A_{\overline{i_1\dots i_\ell}, \overline{j_1\dots, j_\ell}, \overline{k_1\dots k_\ell}}  
\end{equation}
for all $i_n, j_n, k_n \in \{0,1\}$. 
The tensor $\widetilde{A}^{\qttt}$ is subsequently \emph{TT-decomposed},
as specified in the next definition.
\begin{definition}[Transposed order QTT]
  \label{def:QT3}
Let $\ell \in \mathbb{N}$ and let $A\in \mathbb{R}^{2^\ell \times 2^\ell \times 2^\ell}$. 
The tensor $A$ admits a transposed QTT decomposition with 
tensor ranks $r_0, \dots, r_{\ell}$ and cores $U_1, \dots, U_\ell$ if
\begin{equation}
 \label{eq:transposedQTT-def}
    A_{\overline{i_1\dots i_\ell}, \overline{j_1\dots, j_\ell}, 
      \overline{k_1\dots k_\ell}} = U^1(\overline{i_1j_1k_1})\cdots U^\ell(\overline{i_\ell j_\ell k_\ell}),
\end{equation}
for all $i_n, j_n, k_n\in \{0,1\}$, and where 
$U^n:\{0, \dots, 7\}\to \mathbb{R}^{r_{n-1}\times r_n}$, for $n = 1, \dots, \ell$. 
We have the restriction on the ranks $r_0 = r_\ell = 1$.
\end{definition}
We denote by $\fT^{\qttt} : \mathbb{R}^{2^\ell\times 2^\ell \times
  2^\ell} \to \mathbb{R}^{\overbrace{\scriptstyle 8\times
        \cdots \times 8}^{\ell\text{ times }}}$ the ``transposed QTT''
    tensorization given by
    \begin{equation}
      \label{eq:transposedQTT-tensorization}
      \fT^{\qttt}(A) = \widetilde{A}^{\qttt},
    \end{equation}
    with $\widetilde{A}^{\qttt}$ defined in \eqref{eq:transposedQTT-reshape}.

A representation of the transposed order QTT decomposition in tensor network format is
given in Figure \ref{subfig:transposedqtt}.

\subsection{Tucker-QTT}
\label{sec:TuckQTT}
The Tucker-QTT (TQTT) decomposition is a combination of the 
Tucker and the QTT decompositions, first considered in~\cite{Dolgov:2013:ttqtt}.
A tensor $A\in\mathbb{R}^{2^\ell\times 2^\ell \times 2^\ell}$ is 
\emph{represented in the Tucker decomposition}
if
\[
	A_{ijk} = \sum_{\beta_1 = 1}^{R_1}\sum_{\beta_2 = 1}^{R_2}\sum_{\beta_3 = 1}^{R_3} 
	G_{\beta_1 \beta_2 \beta_3}
	U_{\beta_1}(i) V_{\beta_2}(j) W_{\beta_3}(k)
\]
where $R_1,R_2,R_3 \in {\mathbb N}$ 
are the Tucker ranks, the tensor 
$G\in \mathbb{R}^{R_1\times R_2\times R_3}$ is the Tucker core 
and the Tucker factors $U,V,W$ can be considered as matrices 
$U\in\mathbb{R}^{2^\ell \times R_1}$, $V\in\mathbb{R}^{2^\ell \times R_2}$, $W\in\mathbb{R}^{2^\ell \times R_3}$.
In the TQTT decomposition, the factor matrices $U,V,W$ are given by QTT decompositions, 
where, e.g. for $U$, 
only one of the QTT cores depends on the corresponding column number $\beta_1$:
\begin{equation}
  \label{eq:TQTT-cores}
	\begin{split}
		& U_{\beta_1}(i) = U^1_{\beta_1}(i_1) U^2(i_2)\dots U^\ell(i_\ell), \quad i = \overline{i_1\dots i_\ell}, \\ 
		& V_{\beta_2}(j) = V^1_{\beta_2}(j_1) V^2(j_2)\dots V^\ell(j_\ell), \quad j = \overline{j_1\dots j_\ell}, \\ 
		& W_{\beta_3}(k) = W^1_{\beta_3}(k_1) W^2(k_2)\dots W^\ell(k_\ell), \quad k = \overline{k_1\dots k_\ell}, \\
	\end{split}
\end{equation}
We denote the QTT ranks of 
$U,V,W$ as $\{r_0, r_1, \dots,r_{\ell}\}$, $\{s_0,s_1,\dots,s_{\ell}\}$ and $\{t_0,t_1,\dots,t_{\ell}\}$ with the constraints $r_0 = R_1$, $s_0=R_2$, $t_0=R_3$ and $r_{\ell} = s_{\ell} = t_{\ell} = 1$.

\begin{definition}[Tucker-QTT (TQTT) representation] 
Let $\ell \in \mathbb{N}$ and let $A\in \mathbb{R}^{2^\ell \times 2^\ell \times 2^\ell}$. 
$A$ admits a Tucker-QTT decomposition with Tucker ranks
$R_1,R_2,R_3$ and QTT ranks $r_0,r_1,\dots,r_{\ell}$,
$s_0,\dots,s_{\ell}$, $t_0,\dots,t_{\ell}$
if there exist
a Tucker core $G\in \mathbb{R}^{R_1\times R_2 \times R_3}$ 
and 
QTT cores 
$U^1, \dots, U^\ell$, $V^1,\dots,V^\ell$, $W^1,\dots, W^\ell$ defined as in \eqref{eq:TQTT-cores}
such that
\begin{equation}
\label{eq:TQTT-def}
\begin{aligned}
  A_{\overline{i_1\dots i_\ell}, \overline{j_1\dots, j_\ell}, \overline{k_1\dots k_\ell}} &
  \\
   \quad = \sum_{\beta_1, \beta_2, \beta_3 = 1}^{R_1, R_2, R_3}
  &G_{\beta_1, \beta_2, \beta_3} 
 U^1_{{\beta_1}}(i_1) U^2(i_2) \ldots U^\ell(i_\ell) \\
 & V^1_{{\beta_2}}(j_1) V^2(j_2)\ldots V^\ell(j_\ell) 
 W^1_{{\beta_3}}(k_1) W^2(k_2)\ldots W^\ell(k_\ell).
\end{aligned}
\end{equation}
\end{definition}
\subsection{Degrees of freedom}
\label{sec:DOFs}
Supposing for ease of notation that  $r_n= s_n = t_n = r_{\qtt}$ for the classic QTT representation, 
$r_n = r_{\qttt}$ for the transposed one, and 
$r_n= s_n = t_n = r_{\tqtt}$ and $R_1 = R_2 = R_3 = R$ for Tucker-QTT, 
the number $\Ndof$ of parameters in the 
QTT representations is bounded as
\begin{equation}
  \label{eq:Ndof-TT}
  \Ndof =
  {\small 
  \begin{cases}
2\left(( 3\ell -2)r_{\qtt}^2 + 2r_{\qtt} \right) = {\mathcal{O}(\ell r_{\qtt}^2)}& \text{classic QTT}\\
8\left( (\ell-2)r_{\qttt}^2  +2r_{\qttt} \right) = {\mathcal{O}(\ell r_{\qttt}^2)}&  \text{transposed QTT}\\
R^3 + 6\left( (\ell-2)r_{\tqtt}^2 + (R+1) r_{\tqtt} \right)= \mathcal{O}(R^3 + \ell r_{\tqtt}^2 + Rr_{\tqtt})& \text{Tucker-QTT}.
  \end{cases}
} 
\end{equation}
\begin{figure}
  \centering
  \begin{subfigure}{.7\linewidth}
    \centering
   \includegraphics[width=.9\textwidth]{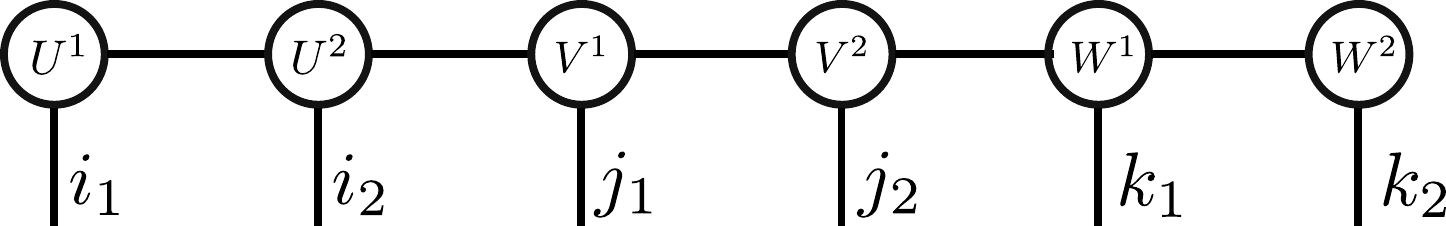}
   \caption{}
   \label{subfig:qtt}
 \end{subfigure}%
\begin{subfigure}{.3\linewidth}
    \centering
   \includegraphics[width=.7\textwidth]{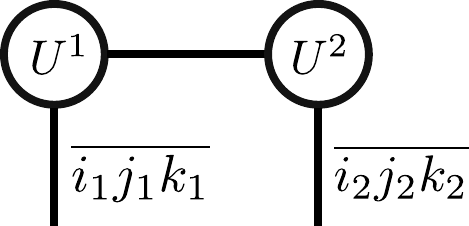}
   \caption{}
   \label{subfig:transposedqtt}
  \end{subfigure}
 \\
\begin{subfigure}{\linewidth}
    \centering
   \includegraphics[width=.4\textwidth]{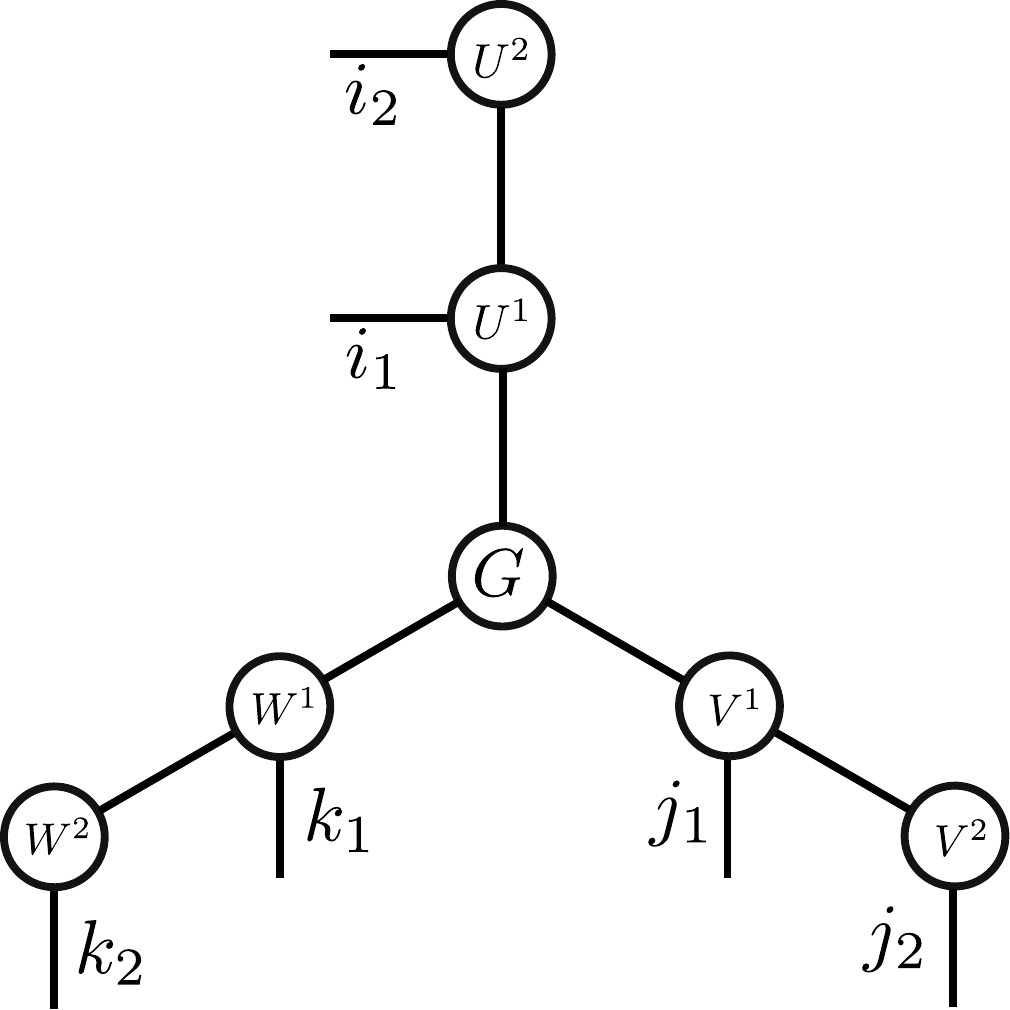}
   \caption{}
   \label{subfig:tqtt}
  \end{subfigure}
  \caption{Tensor networks for the QTT (\subref{subfig:qtt}), transposed QTT (QT3)
    (\subref{subfig:transposedqtt}), and Tucker-QTT ({TQTT}) representations (\subref{subfig:tqtt}). 
    Each node represents a tensor 
    with edges in the network indicating indices. An edge connecting two
    nodes is a contracted index (corresponding to tensor multiplication).
    This can be seen comparing the networks with equations \eqref{eq:QTT-def},
    \eqref{eq:transposedQTT-def}, and \eqref{eq:TQTT-def}.}
  \label{fig:networks}
\end{figure}
%%%%%%%%%%%%%%%%%%%%%%%%%%%%%%%%%%%%%%%%%%%%%%%%%%%%%%%%%%%%%%%%%%%%%%%%%%%%%%%%%%%
\section{Functional setting}
\label{sec:mesh}
%%%%%%%%%%%%%%%%%%%%%%%%%%%%%%%%%%%%%%%%%%%%%%%%%%%%%%%%%%%%%%%%%%%%%%%%%%%%%%%%%%%%
Our analysis will require the introduction of two different meshes and of two
respective finite element spaces in the cube $Q$.
The first one is a \emph{uniform
tensor product mesh} with distance between nodes given by $h_\ell = 2^{-\ell}$.
This mesh contains $2^\ell$ nodes in every physical direction; given a function
$f$ defined over $Q$, the point values of $f$ at the mesh points can be grouped
in a three-dimensional tensor of dimension $2^\ell\times 2^\ell \times 2^\ell$,
which can be QTT-approximated in the formats introduced in the previous section.
Note that, in practice, one does not need to compute the values of the function
at all $2^{3\ell}$ mesh points, 
see e.g. \cite{Oseledets2010}, as this would undermine the efficiency of
tensor compressed methods 
(hence the ``virtual'' character of the ``background-mesh''). 
Furthermore, 
tensor-formatted closed forms of some discrete differential operators exist, 
see e.g. \cite{Khoromskij:2011:QuanticsApprox,BK+IOQuantColloc,VKBK2012}.
This can be used to discretize certain partial differential equations in quantized tensor format, 
as it will be shown in the sequel.
The space of (tensor-formatted) functions on the uniform mesh 
is the space of $\mathbb{Q}_1$ finite elements, i.e., the tensor product of one-dimensional
Lagrange functions associated with mesh nodes.

The second finite dimensional space we introduce is the \emph{auxiliary $hp$ space}. 
This space is introduced here only for proving tensor rank bounds 
of the QTT-structured approximation. It is
never accessed during numerical computation in the tensor formats.
The $hp$ space is, in particular,
an $H^1$-conforming finite element space, on a 
mesh with elements geometrically refined towards the origin. 
The polynomial degree of functions in the $hp$ space is, instead,
increasing polynomially with the number of geometric mesh layers.
This is made more precise in Section \ref{sec:hp} below.
The role of the auxiliary 
$hp$ finite element approximation is to provide an
\emph{exponentially convergent, continuous and piecewise polynomial approximation}
on a \emph{bisection geometric partition which is compatible with the 
virtual mesh in the $\mathbb{Q}_1$-approximation} 
for generic functions in the weighted Sobolev space $\cJ^\varpi_{\gamma}(Q)$.

A function in 
$v\in\cJ^\varpi_\gamma(\Omega)$ can then be approximated---with exponential
accuracy---by its projection $v_{\mathsf{hp}}$ into the $hp$ space. 
By re-interpolating
$v_{\mathsf{hp}}$ on the virtual mesh and QTT-compressing the resulting tensor, 
we establish existence of quantized, tensor-structured approximations 
with polylogarithmic bounds on the QTT ranks and the number of QTT parameters. 
The quasi-interpolation operator from $v$ to its representation on
the virtual mesh is introduced in Section \ref{sec:quasi-interpolation}.

For simplicity, we will consider here functions that have zero trace on the
part of the boundary not abutting at the origin, i.e., 
on
\begin{equation*}
  %\label{eq:gamma-def} 
\Gamma 
= 
\left\{ (x_1,x_2,x_3)\in \partial Q:\, x_1x_2x_3 \neq 0\right\}
\;.
\end{equation*}
We denote by $H^1_\Gamma(Q)$ the subspace of $H^1(Q)$ functions 
with zero trace on $\Gamma$.
We then fix $\gamma \in\mathbb{R}$ such that $\gamma-3/2\in (0,1)$, 
two constants $C_X$, $A_X>0$, and 
a regularity exponent $\gev\geq 1$ and 
denote by
\begin{equation*}
  %\label{eq:Xdef}
X=\cJ^{\varpi}_\gamma(Q; C_X, A_X, \gev) \cap H^1_\Gamma(Q)
\end{equation*}
the weighted space of Gevrey-$\gev$-regular 
functions with zero trace on $\Gamma$ that will be considered henceforth.
%%%%%%%%%%%%%%%%%%%%%%%%%%%%%%%%%%%%%%%%%%%%%%%%%%%%%%%%%%%%%%%%%%%%%%%%%%%%%%
\subsection{Low order virtual FE space $\Xqttell$}
\label{sec:VirtUnifSpc}
%%%%%%%%%%%%%%%%%%%%%%%%%%%%%%%%%%%%%%%%%%%%%%%%%%%%%%%%%%%%%%%%%%%%%%%%%%%%%%
We introduce the so-called ``virtual'' FE space discussed above. 
Here, it will consist of the space of continuous, 
piecewise trilinear functions
on a uniform mesh of axiparallel hexahedral elements 
of size $2^{-\ell}$ (so-called $\mathbb{Q}_1$-FEM), 
which we now introduce.
%%%%%%%%%%%%%%%%%%%%%%%%%%%%%%%%%%%%%%%%%%%%%%%%%%%%%%%%%%%%%%%%%%%%%%%%%%%%%%
\subsubsection{Uniform background mesh $\cT^\ell$}
\label{sec:UnifGrd}
%%%%%%%%%%%%%%%%%%%%%%%%%%%%%%%%%%%%%%%%%%%%%%%%%%%%%%%%%%%%%%%%%%%%%%%%%%%%%%
In $Q = (0,1)^3$,
we introduce the uniform mesh $\cT^\ell$ with 
nodes $\bx_{i,j,k}\in 2^{-\ell} {\mathbb N}_0^3 \cap \bar{Q}$,
for $(i,j,k) \in \{0, \dots, 2^\ell\}^3$. 
For a refinement level $\ell\in \mathbb{N}$, 
we write 
$ I^\ell_j = (2^{-\ell}j, 2^{-\ell}(j+1)) $, $j=0, \dots, 2^\ell-1$.
Then,
$\cT^\ell = \{ I^\ell_i\times I^\ell_j \times I^\ell_k: i,j,k = 0, \dots, 2^\ell-1 \}$.
\subsubsection{Virtual finite element space}
For $(i,j,k) \in \{0, \dots, 2^\ell-1\}^3$, 
we denote by $\phi_{i,j,k}$ the 
locally trilinear, continuous nodal Lagrange functions 
which satisfy
\begin{equation*}
  \phi_{i,j,k}(\bx_{m,n,p}) = \delta_{im}\delta_{jn}\delta_{kp},
\quad 
(i,j,k) \in \{0, \dots, 2^\ell-1\}^3, \, (m,n,p)\in \{0, \dots, 2^\ell\}^3
\end{equation*}
where $\delta_{im}$ denotes the Kronecker delta symbol 
for indices $i$ and $m$.

The space of continuous, locally trilinear Lagrange functions 
on the (background) mesh $\cT^\ell$ is 
\begin{equation*}
  %\label{eq:Xqttell}
\Xqttell = \spn\{\phi_{i,j,k}: (i,j,k) \in \{0, \dots, 2^\ell-1\}^3\}.
\end{equation*}
Note that the basis functions 
$\phi_{i,j,k}$ and the space $\Xqttell$ are algebraic tensor
products of the corresponding univariate functions, resp. spaces.
We remark that for every $v\in \Xqttell$ holds $v_{|_\Gamma} = 0$.

  \begin{remark}
$\Xqttell$ contains functions that vanish on $\Gamma$. 
We limit ourselves to this case for simplicity of notation;
the extension of our analysis to functions with nonzero trace 
on $\Gamma$ involves additional technicalities.
We refer to \cite{Kazeev2018} for the two-dimensional case.
\end{remark}

%%%%%%%%%%%%%%%%%%%%%%%%%%%%%%%%%%%%%%%%%%%%%%%%%%%%%%%%%%%%%%%%%%%%%%%%%%%%%%%%%
\subsubsection{Lagrange interpolation operator $\cI^\ell$}
\label{sec:I}
%%%%%%%%%%%%%%%%%%%%%%%%%%%%%%%%%%%%%%%%%%%%%%%%%%%%%%%%%%%%%%%%%%%%%%%%%%%%%%%%%
We denote by 
$\cI^\ell$ the Lagrange interpolation operator on the uniform tensor mesh 
$\cT^\ell$. 
I.e., $\cI^\ell: C(\bar{Q}) \to \Xqttell$ is defined as 
\begin{equation*}
%\label{eq:Idef}
\left( \mathcal{I}^\ell v \right)(x) 
= 
\sum_{(i,j,k)\in \{0, \dots, 2^\ell-1\}^3}v(\bx_{i,j,k})\phi_{i,j,k}(x),
\; x\in \bar{Q}\;.
\end{equation*}
%%%%%%%%%%%%%%%%%%%%%%%%%%%%%%%%%%%%%%%%%%%%%%%%%%%%%%%%%%%%%%%%%%%%%%%%%%%%%%%%%
\subsubsection{Analysis and synthesis operators}
\label{sec:AnSyntOp}
%%%%%%%%%%%%%%%%%%%%%%%%%%%%%%%%%%%%%%%%%%%%%%%%%%%%%%%%%%%%%%%%%%%%%%%%%%%%%%%%%
For $\ell\in \mathbb{N}$, $\scrA^\ell : \Xqttell\to \mathbb{R}^{2^\ell\times
  2^\ell \times 2^\ell}$ and $\scrS^\ell : \mathbb{R}^{2^\ell\times
  2^\ell \times 2^\ell} \to \Xqttell $ are the analysis and synthesis operators, such that
\begin{equation} \label{eq:ansynth}
(\scrA^\ell v^\ell)_{i,j,k} 
=  
v^\ell(\bx_{i,j,k}), \qquad (\scrS^\ell\bv)(x) = \sum_{i,j,k = 0}^{2^\ell-1}\bv_{i,j,k}\phi_{i,j,k}(x).
\end{equation}
%%%%%%%%%%%%%%%%%%%%%%%%%%%%%%%%%%%%%%%%%%%%%%%%%%%%%%%%%%%%%%%%%%%%%%%%%%%%%%%%%
\subsection{Auxiliary $hp$ space}
\label{sec:hp}
%%%%%%%%%%%%%%%%%%%%%%%%%%%%%%%%%%%%%%%%%%%%%%%%%%%%%%%%%%%%%%%%%%%%%%%%%%%%%%%%%
We obtain the QTT-rank bounds on TT-formatted approximations
by comparison with $hp$-approximations.
To this end, we introduce the $hp$-FE spaces.
We start with \emph{$1$-irregular meshes of axiparallel hexahedra}
with geometric refinement towards the singularity of the function of
interest (``geometric meshes'' for short).
%%%%%%%%%%%%%%%%%%%%%%%%%%%%%%%%%%%%%%%%%%%%%%%%%%%%%%%%%%%%%%%%%%%%%%%%%%%%%%%%
\subsubsection{Geometric mesh}
\label{sec:GeoMes}
%%%%%%%%%%%%%%%%%%%%%%%%%%%%%%%%%%%%%%%%%%%%%%%%%%%%%%%%%%%%%%%%%%%%%%%%%%%%%%%%%
Let $\ell \in \mathbb{N}$. For $i=0, \dots, \ell$, let
\begin{equation*}
  J_{1,i}^\ell = (2^{i - \ell - 1}, 2^{i-\ell})\qquad\text{and}\qquad J_{0,i}^\ell = (0, 2^{i -\ell}).
\end{equation*}
Then, for $k\in \{0, \dots, \ell\}$ and $a,b,c\in \{0,1\}$, define
\begin{equation*}
  K_{abc, k}^\ell = J_{a,k}^\ell\times J_{b,k}^\ell\times J_{c,k}^\ell
\end{equation*}
see Figure \ref{fig:elements}.
\begin{figure}
  \centering
  \includegraphics[width=.6\textwidth]{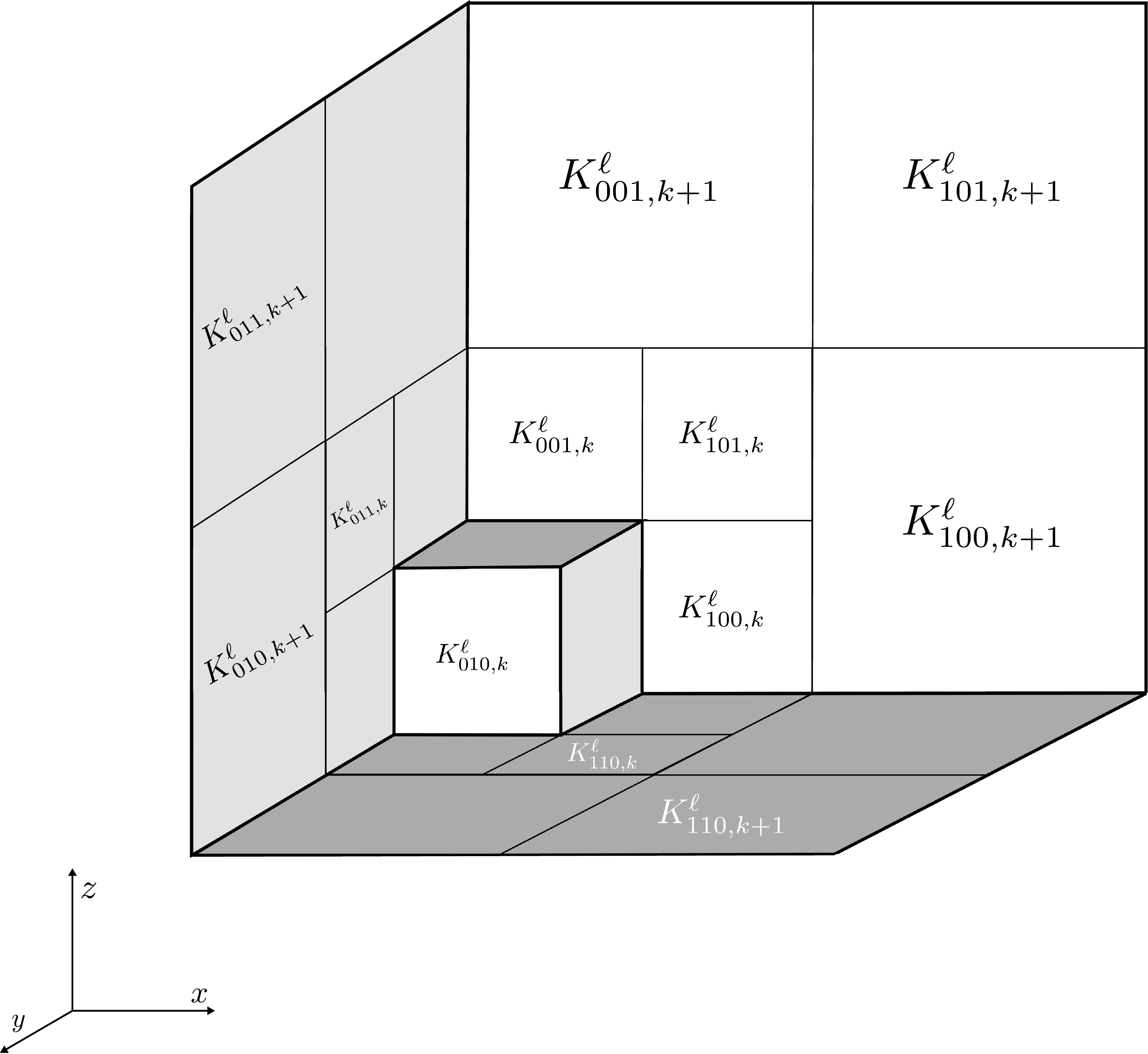}
  \caption{Elements $K_{n,j}^\ell$, for fixed $\ell$, for $j=k, k+1$, 
  and for $n\in \{001, \dots, 110\}$. 
  Element $K_{111}$ not visible in this projection.}
  \label{fig:elements}
\end{figure}
Denoting $\cN = \{001, \dots, 111\}$, the auxiliary geometric mesh 
is given by
\[
  \cG^\ell = \{ K^\ell_{n, k}, k\in \{1, \dots, \ell\}, n\in
  \cN\}\cup K^\ell_{000,0}.
\]
Element $K_{000, 0 }^\ell$ has one vertex coinciding with the origin.
We collect all elements at the same refinement level in 
\emph{mesh layers}
\begin{equation}
  \label{eq:mesh-level}
\cL^{\ell}_0 = \{K_{000, 0}\}, \qquad 
  \cL^\ell_j = \{K_{n,j}, \, n\in \cN\} \text{ for } j=1, \dots, \ell.
\end{equation}
We also introduce the one- and two-dimensional versions 
of the geometric mesh as
\begin{equation*}
  \cG^\ell_{\mathrm{2d}} 
  =  
  \{ K^\ell_{n, k}, k\in \{1, \dots, \ell\}, n\in \{01,10,11\}\}\cup K^\ell_{00,0},
\end{equation*}
where $K_{ab, k}^\ell = J_{a,k}^\ell\times J_{b, k}^\ell$,
and 
\begin{equation}
  \label{eq:hpmesh1d}
  \cG^\ell_{\mathrm{1d}} = \{ J^\ell_{1, k}, k\in \{1, \dots, \ell\}\}\cup J^\ell_{0,0},
\end{equation}
see Figure \ref{fig:G-12d}.
\begin{figure}
  \centering
  \begin{subfigure}{0.5\linewidth}
   \centering\includegraphics[width=.7\textwidth]{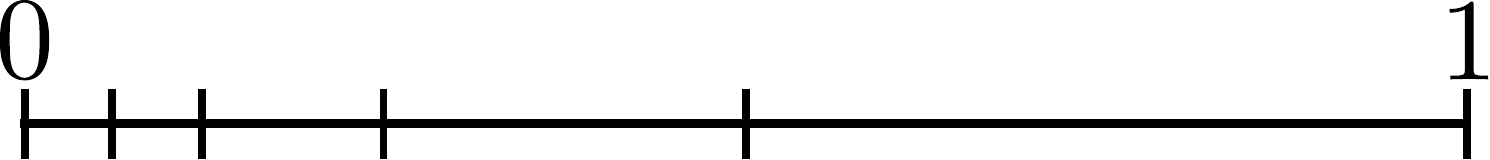}
  \end{subfigure}%
  \begin{subfigure}{0.5\linewidth}
   \centering\includegraphics[width=.7\textwidth]{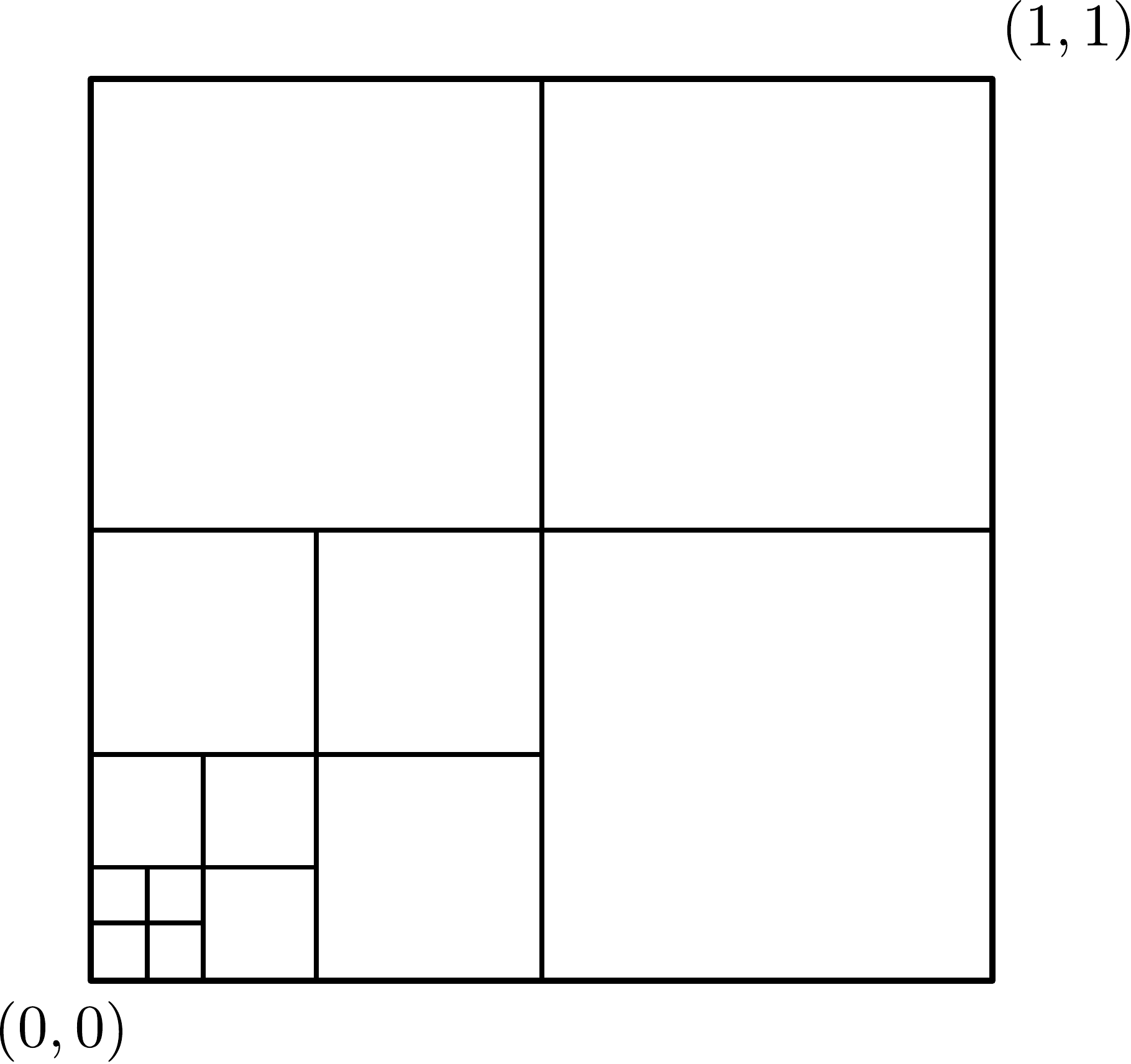}
  \end{subfigure}
  \caption{
    Univariate geometric mesh $\cG^\ell_{\mathrm{1d}}$ (left) and
    bivariate geometric mesh $\cG^\ell_{\mathrm{2d}}$ (right) with 
    subdivision ratio $1/2$.}
  \label{fig:G-12d}
\end{figure}
We remark that for $q\in \mathbb{N}$, $k = 1, \dots,  2^q-1$,  $j = 1, \dots,  q-1$,
and for all integer $\ell>q$
\begin{subequations}
\begin{equation*}
%\label{eq:IJa}
  I^q_k \subset J^q_{1,\lfloor \log_2 k \rfloor+1} \qquad J^q_{1,j}   = J^\ell_{1, \ell-q+j} \;.
\end{equation*}
Furthermore,
for $p = 1, \dots, q-1$ and all $m = p+1, \dots, q-1$, there holds
\begin{equation*}
  %\label{eq:IJb}
  J^q_{1,p} \subset J^q_{0, m}.
\end{equation*}
\end{subequations}
\subsubsection{$hp$ space}
\label{sec:hpSpc}
The $hp$ space is formally introduced as 
\begin{multline}
  \label{eq:hp-def}
  \Xhpell  = \{ v\in H^1(Q): v_{|_{K^\ell_{n, j}}} \in \mathbb{Q}_{p}(K^\ell_{n, j}), \\
                \text{ for all } n\in \cN, j = 1, \dots, \ell\text{ and }n=000, j=0\}.
\end{multline}
Note that, as a consequence of the existence of a continuous $hp$
approximation to functions in $\cJ^\infty_\gamma(Q)$ 
proved in Appendix \ref{sec:gevrey} and in \cite{Schotzau2015}, 
the space $\Xhpell$ is well-defined by \eqref{eq:hp-def}.
%%%%%%%%%%%%%%%%%%%%%%%%%%%%%%%%%%%%%%%%%%%%%%%%%%%%%%%%%%%%%%%%%%%%%%%%%%%%%%%%%%%%%%%%%%%%%%%%%%
\subsubsection{$hp$ approximation}
\label{sec:Pihp}
%%%%%%%%%%%%%%%%%%%%%%%%%%%%%%%%%%%%%%%%%%%%%%%%%%%%%%%%%%%%%%%%%%%%%%%%%%%%%%%%%%%%%%%%%%%%%%%%%%
We provide a brief presentation of (novel) 
$hp$-interpolation error bounds
which are exponential in the number of degree of freedom
for functions in the Gevrey-type classes 
$\cK^{\varpi,q}_\gamma(Q; C, A, \gev)$, 
$\cJ^{\varpi, q}_\gamma(Q; C, A, \gev)$ 
defined in Section~\ref{sec:GevAnFct}.
We consider here \emph{axiparallel, geometric
partitions} of $Q = (0,1)^3$ into hexahedral elements; this entails,
of course, irregular nodes and faces so that $hp$-interpolants are 
to be constructed in a two-stage process: first, an elementwise 
$hp$-(quasi)interpolant with analytic error bounds and second,
\emph{polynomial face jump liftings} which preserve the analytic
bounds. We refer to the appendix and to e.g. \cite{Schotzau2015}
for details on this.

In the analytic case, i.e., when $\gev = 1$
such exponential error bounds are well-known (e.g. \cite{Schotzau2015}).
However, for $\gev>1$, these bounds are novel; 
for regular geometric meshes of tetrahedra, corresponding
bounds have recently been established in \cite{Feischl2018}. 

We introduce in Appendix \ref{sec:gevrey} the projector $\Pihpell :
 \cJ^\infty_\gamma(Q)\to \Xhpell$, defined for $\gamma > 3/2$.
We recall here that given a function $v\in \cJ^{\infty}_\gamma(Q)$, for $\gamma>3/2$, then $\Pihpell
v \in H^1(Q)\cap C(\bar{Q})$, i.e., the projector is \emph{conforming} in $H^1(Q)$.
Furthermore for all $u\in \cJ^{\varpi}_\gamma(Q; C, A, \gev)$, with
$\gamma>3/2$, there exist $p\simeq \ell^\gev$, and positive
$C_{\mathsf{hp}}$, $b_{\mathsf{hp}}$ (depending on $C, A$, $\gev$) 
such that for every $\ell\in \mathbb{N}$ holds
\begin{equation}
  \label{eq:hp-error-text}
  \| u - \Pihpell u \|_{H^1(Q)} \leq C_{\mathsf{hp}} \exp(-b_{\mathsf{hp}}\ell),
\end{equation}
 and $\dim(\Xhpell)\simeq \ell p^3 \simeq \ell^{3\gev+1}$,
see Proposition \ref{prop:hp-error} in the appendix. 
%%%%%%%%%%%%%%%%%%%%%%%%%%%%%%%%%%%%%%%%%%%%%%%%%%%%%%%%%%%%%%%%
\subsection{Quasi interpolation operator $\fP^\ell$}
\label{sec:quasi-interpolation}
%%%%%%%%%%%%%%%%%%%%%%%%%%%%%%%%%%%%%%%%%%%%%%%%%%%%%%%%%%%%%%%%
We recall that
  \begin{equation*}
  % \label{eq:Xdef}
X=\left\{ v\in\cJ^{\varpi}_\gamma(Q; C_X, A_X, \gev) : v_{|_\Gamma} = 0\right\},
\end{equation*}
where 
$ \Gamma = \left\{ (x_1,x_2,x_3)\in \partial Q:\, x_1 x_2 x_3 \neq 0 \right\}$ 
and for $\gamma>3/2$, positive constants $C_X$ and $A_X$, and Gevrey exponent $\gev\geq 1$.
We also fix $p\simeq \ell^\gev$ such that \eqref{eq:hp-error-text} holds 
and define the quasi-interpolation operator $\fP^\ell: X\to \Xqttell$ as
\begin{equation} \label{eq:vqtt}
   \fP^\ell u = \cI^\ell \Pihpell u.
\end{equation}
%%%%%%%%%%%%%%%%%%%%%%%%%%%%%%%%%%%%%%%%%%%%%%%%%%%%%%%%%%%%%%%%
\section{Quasi interpolation error}
\label{sec:interp-error}
%%%%%%%%%%%%%%%%%%%%%%%%%%%%%%%%%%%%%%%%%%%%%%%%%%%%%%%%%%%%%%%%%%
We give here (specifically, in Proposition \ref{prop:hp-reinterp}) 
an estimate on the error introduced by the quasi-interpolation 
operator $\fP^\ell$ defined in Section \ref{sec:quasi-interpolation}. 
We start by estimating, in the following lemma, 
the error introduced by interpolating the $hp$ projection of a function in $X$.
\begin{lemma}
  \label{lemma:interp-error}
  Let $\ell \in \mathbb{N}$, $u\in X$, $\cI^\ell : C(\bar{Q})\to \Xqttell$ and $\Pihpell :
 X\to \Xhpell$ defined in Sections \ref{sec:I} and \ref{sec:Pihp}, respectively. 
 Then there exist constants $C, b_{\mathcal{I}} >0$ such that
 \begin{equation*}
   %\label{eq:interp-error}
  \|\left( \Id - \cI^\ell  \right)\Pihpell u \|_{H^1(Q)} \leq C \exp(-b_{\mathcal{I}}\ell).
  \end{equation*}
\end{lemma}
\begin{proof}
  There holds 
  \begin{equation}
    \label{eq:interp-error1}
    \| \left( \Id - \cI^\ell  \right)\Pihpell u \|_{H^1(Q)} ^2
    =
    \sum_{j=0}^\ell \sum_{K \in \cL^\ell_j}
    \| \left( \Id - \cI^\ell  \right)\Pihpell u \|_{H^1(K)} ^2,
  \end{equation}
where the mesh layers $\cL^\ell_j$ are defined in \eqref{eq:mesh-level}.
  The quantity on the right hand side of this equation 
  is an upper bound for the error of interpolation 
  over a uniform mesh of axiparallel cubes of edge length $2^{-\ell}$.
  The axiparallel cubes $K^\ell_{n,j}$ are affine equivalent 
  to the reference element $\hK = (-1,1)^3$. 
  Furthermore, since $u\in X$, by Remark \ref{remark:nullboundary} in the Appendix 
  there holds $\left( \Pihpell u \right)_{|_\Gamma} =0$.
  Hence there exists $C>0$ such that, for all $(n,j)\in \cN\times \{1, \dots, \ell\} \cup (000, 0)$ 
  and for all $\ell$
  \begin{equation*}
    %\label{eq:interp-error3}
    \| \left( \Id - \cI^\ell  \right)\Pihpell u \|_{H^1(K^\ell_{n,j})}
    \leq
    C 2^{-\ell}|\Pihpell u|_{H^2(K^\ell_{n,j})}.
  \end{equation*}
  By the polynomial inverse inequality
  \begin{equation*}
    |v |_{H^2(K)} \leq C \frac{p^2}{h_K} |v|_{H^1(K)},
  \end{equation*}
where $v$ is a polynomial of degree $p$ and $h_K$ is the diameter of $K$
  (see, e.g.  \cite{Schwabphp98,Georgoulis2008}), 
recalling that an element 
$K^\ell_{n,j}\in \cL^\ell_j$ is an axiparallel cube with 
diameter $h_j\simeq 2^{ - \ell +j}$) and using a triangle inequality,
there exists a constant $C>0$ independent of $\ell$  and of $p$
such that 
  \begin{align*}
    2^{-\ell}|\Pihpell u|_{H^2(K^\ell_{n,j})}
    &
      \leq
      C 2^{-\ell} h_j^{-1}p^2|\Pihpell u|_{H^1(K^\ell_{n,j})}
      \\
    &\leq
    C 2^{-j}p^2|\Pihpell u|_{H^1(K^\ell_{n,j})}\\
    &\leq
    C 2^{-j}p^2\left(|u|_{H^1(K^\ell_{n,j})}  + | u - \Pihpell u|_{H^1(K^\ell_{n,j})}\right) .
  \end{align*}
  Since $\gamma> 1$,
  there exists a uniform constant $C$ such that, on each $K^\ell_{n,j}$, 
  \[2^{(\ell - j)(\gamma-1)} \leq C r_{|_{K^\ell_{n,j}}}^{1-\gamma}.
\] 
From this last inequality and \eqref{eq:Jspace}
  \begin{equation}
    \label{eq:interp-error4}
    % 2^{-\ell}|\Pihpell u|_{H^2(K^\ell_{n,j})}
2^{-j}|u|_{H^1(K^\ell_{n,j})}
    \leq
    C 2^{-\ell(\gamma - 1) - j(2-\gamma)}\|u\|_{\cJ^1 _\gamma(K^\ell_{n,j})}. 
  \end{equation}
Combining equations \eqref{eq:interp-error1} to \eqref{eq:interp-error4}, 
there exists a constant $C>0$ such that for all $\ell$ holds
  \begin{align*}
    &\| \left( \Id - \cI^\ell  \right)\Pihpell u \|_{H^1(Q)}^2
       \\ 
    & \quad
      \leq
C \sum_{j=0}^\ell{\sum_{K\in \cL^\ell_j}} 2^{-2j}p^4
                    \left(|u|_{H^1(K)}  + | u - \Pihpell u|_{H^1(K)}\right)^2 
       \\ 
    & \quad
 \leq C p^4\left(\sum_{j=0}^\ell \sum_{K\in \cL^\ell_j} 
          2^{-2\ell(\gamma - 1) - 2j(2-\gamma)}\|u\|^2_{\cJ^1 _\gamma(K)} 
      +  | u - \Pihpell u|_{H^1(Q)}^2 \right) .
  \end{align*}
  Then, by \eqref{eq:hp-error-text} and since $p\simeq \ell^\gev$,
  \begin{align*}
  \| \left( \Id - \cI^\ell  \right)\Pihpell u \|_{H^1(Q)}^2
    &\leq
C \ell^{4\gev}\left( 2^{-2\ell\min(\gamma-1, 1)}\|u\|^2_{\cJ^1 _\gamma(Q)}  
  + C_{\mathsf{hp}}^2e^{-2b_{\mathsf{hp}}\ell}\right).
\end{align*}
Absorbing the terms algebraic in $\ell$ into the exponential by a change of constant
concludes the proof.\myqed
\end{proof}
\begin{proposition}
  \label{prop:hp-reinterp}
  Let $0<\epsilon_0 < 1$ and $u\in X$
  % , and $\vqttell = \fP^\ell u$ defined as in
  % \eqref{eq:vqtt}
  . Then for all $0 < \epsilon \leq \epsilon_0$ there exists
$\ell\in \mathbb{N}$ such that
\begin{equation*}
  %\label{eq:hp-reinterp-error}
  \| u - \fP^\ell u\|_{H^1(Q)} \leq \epsilon,
  \end{equation*}
  where $\fP^\ell u \in \Xqttell$ is defined in \eqref{eq:vqtt} and there exists $C>0$
  independent of $\epsilon$ such that
  \begin{equation*}
    \ell \leq C |\log\epsilon|.
  \end{equation*}
\end{proposition}
\begin{proof}
  By a triangle inequality, Lemma \ref{lemma:interp-error}, and Proposition
  \ref{prop:hp-error} (recalled above in equation \eqref{eq:hp-error-text}),
  \begin{equation*}
    \| u - \fP^\ell u \|_{H^1(Q)}
    \leq 
    \| u - \Pihpell u \|_{H^1(Q)}
    + \| (\Id - \cI^\ell)\Pihpell u \|_{H^1(Q)} \leq C\exp(-b\ell),
  \end{equation*}
  where $C$ and $b$ are independent of $\ell$.
  The choice $\ell =  \lceil b^{-1}\log\left(\frac{C}{\epsilon}\right)\rceil$
  and adjusting the value of $C$ concludes the proof.
  \myqed
\end{proof}
%%%%%%%%%%%%%%%%%%%%%%%%%%%%%%%%%%%%%%%%%%%%%%%%%%%%%%%%%%%%%%%%%%%%%%%%%%%%%%%%%%%
\section{QTT formatted approximation of $u \in \cJ^\varpi_\gamma(Q)$}
\label{sec:analytic-approximation}
%%%%%%%%%%%%%%%%%%%%%%%%%%%%%%%%%%%%%%%%%%%%%%%%%%%%%%%%%%%%%%%%%%%%%%%%%%%%%%%%%%%
We now state and prove our main results. 
For Gevrey-regular functions in $Q=(0,1)^3$ with point singularity at the
origin, and for each of the three tensor formats (QTT, QT3, TQTT), 
we prove bounds on the ranks  which are sufficient to achieve 
a prescribed approximation accuracy $\varepsilon \in  (0,1)$ in the norm $H^1(Q)$.
This is the relevant norm for linear, second order, elliptic PDEs.
%%%%%%%%%%%%%%%%%%%%%%%%%%%%%%%%%%%%%%%%%%%%%%%%%%%%%%%%%%%%%%%%%%%%%%%%%%%%%%%%%%%
\subsection{Tensor Rank Bounds for QTT Approximation}
\label{sec:TRkBdsQTT}
%%%%%%%%%%%%%%%%%%%%%%%%%%%%%%%%%%%%%%%%%%%%%%%%%%%%%%%%%%%%%%%%%%%%%%%%%%%%%%%%%%%
\begin{lemma}
  \label{lemma:qtt-rank}
Let $0<\epsilon_0< 1$ and $u\in X$. 
Then, for all $0 < \epsilon\leq \epsilon_0$ there exist
$\ell\in\mathbb{N}$ and $\vqttell = \fP^\ell u$ 
such that 
$\| u - \vqttell\|_{H^1(Q)} \leq \epsilon$
and $\vqttell$ admits a QTT formatted representation with
\begin{equation*}
    \Ndof \leq C |\log\epsilon|^{6\gev+1}
\end{equation*}
degrees of freedom, for $C>0$ independent of $\epsilon$. 
\end{lemma}
\begin{proof}
We consider the unfolding matrices 
$V^{(q)}$ of $\fT^{\qtt}(\scrA^\ell(\vqttell))$, 
with
$\fT^{\qtt}$ defined in \eqref{eq:QTT-tensorization} and $\scrA^\ell$ in \eqref{eq:ansynth}. 
We first consider the case $q\in \{1, \dots, \ell-1\}$. 
In this case, 
\begin{equation*}
V^{(q)}_{\xi_1, \overline{\xi_2\eta\zeta}} 
= 
\vqttell(x_{\overline{\xi_1\xi_2}, \eta, \zeta}),
\end{equation*}
for $\xi_1 = 0, \dots, 2^q-1$, $\xi_2 = 0, \dots, 2^{\ell-q}-1$, 
and for $\eta, \zeta = 0, \dots, 2^\ell-1$.
Now, introduce the reference line $S_1$ as
    $S_1 = (0, 1)\times \{0\}\times \{0\}$.
For each element $K\in \cG^q_{\mathrm{1d}}$, we denote its left and right
endpoints as $y_0^K$ and $y_1^K$, so that $(y^K_0, y^K_1) = K$. 
On $S_1$, we introduce a geometric mesh
  \begin{equation*}
    \cG_{S_1}^q = \left\{ K\times \{0\} \times \{0\},\ K\in \cG^q_{\mathrm{1d}}\right\}, 
  \end{equation*}
and the univariate discontinuous FE space 
\begin{multline*}
X^q_{S_1} 
= 
\bigg\{ v\in L^\infty(S_1): 
v_{|_{K}}\in \mathbb{Q}_{\pmax}(K)\text{ for all }K\in \cG^q_{S_1}
\\ 
\text{ and }v\text{ is right continuous at the nodes of }\cG^q_{\mathrm{1d}}\bigg\}.
\end{multline*}
We require the function to be right continuous at its discontinuity points,
i.e., for any two neighboring intervals $K_\sharp$
and 
$K_\flat$ with $y_0^{K_\flat} < y_1^{K_\flat} = y_0^{K_\sharp} < y_1^{K_\sharp}$ 
and a function $v\in X^q_{S_1}$ such that
\begin{equation*}
  v =
  \begin{cases}
    v_\flat &\text{in }K_\flat\\
    v_\sharp &\text{in }K_\sharp,
  \end{cases}
\end{equation*}
we have $v(y_1^{K_\flat}) = v(y_0^{K_\sharp}) = v_\sharp(y_0^{K_\sharp})$.
We also consider the affine transformation
\begin{equation*}
\phi_{\overline{ijk}} : (x_1,x_2,x_3)\mapsto (x_1 + 2^{-\ell}i, x_2 + 2^{-\ell}j, x_3 + 2^{-\ell}k),
\end{equation*}
so that, for all $\xi_1 = 0, \dots, 2^q-1$, $\xi_2 = 0, \dots, 2^{\ell-q}-1$, and $\eta,
\zeta = 0, \dots, 2^\ell-1$,
\begin{equation*}
    \phi_{\overline{\xi_2\eta\zeta}}^{-1}(x_{\overline{\xi_1\xi_2}, \eta, \zeta}) = (2^{-q}\xi_1, 0, 0).
\end{equation*}
Now,  for all $\xi_1\in\{0, \dots, 2^q-1\}$,
\begin{equation*}
     x_{\overline{\xi_1\xi_2}, \eta, \zeta}  \in \phi_{\overline{\xi_2\eta\zeta}}(S_1) =   \overline{I^q_{\xi_1}}\times \{ 2^{-\ell}\eta\}\times \{2^{-\ell}\zeta\}.
\end{equation*}
Then, for each $\xi_2\in \{0, \dots, 2^{\ell-q}-1\}$ and 
for each $\eta, \zeta\in \{0, \dots, 2^\ell-1\}$ 
there exists a piecewise polynomial $p_{\overline{\xi_2\eta\zeta}}\in X_{S_1}^{q}$ 
such that
\begin{equation*}
\begin{aligned}
\vqttell \left(x_{\overline{\xi_1\xi_2}, \eta, \zeta}\right) &=
\left( \vqttell \circ \phi_{\overline{\xi_2\eta\zeta}} \right)  \left( \phi^{-1}_{\overline{\xi_2\eta\zeta}}\left(x_{\overline{\xi_1\xi_2}, \eta, \zeta}\right)  \right)
\\
&= \left( \vqttell \circ \phi_{\overline{\xi_2\eta\zeta}} \right)  \left((2^{-q}\xi_1, 0, 0)\right) \\
&= 
p_{\overline{\xi_2\eta\zeta}}\left((2^{-q}\xi_1, 0, 0)\right) 
\end{aligned}
\end{equation*}
for all $\xi_1= 0, \dots, 2^{q}-1$.
The piecewise polynomial $p_{\overline{\xi_2\eta\zeta}}$ is constructed as follows.
For each $K = (y_0^K, y_1^K)\in \cG^q_{\mathrm{1d}}\setminus (0, 2^{-q})$ the function $\vqttell\circ
  \phi_{\overline{\xi_2\eta\zeta}}$ is a polynomial 
  of degree $p$ with the first variable in the interval $(y_0^K - 2^{-\ell\xi_2}, y_1^K -
  2^{-\ell\xi_2})$, therefore, \emph{a fortiori}, denoting $\widetilde{J}^K = (y_0^K, y_1^K - 2^{-\ell\xi_2})$
  \begin{equation*}
    \vqttell\circ \phi_{\overline{\xi_2\eta\zeta}} \in \mathbb{Q}_p(\widetilde{J}^K\times \{0\}\times \{0\}).
  \end{equation*}
   Hence, by extension, there exists a polynomial $p^K_{\overline{\xi_2\eta\zeta}}$
   such that $p^K_{\overline{\xi_2\eta\zeta}}\in
   \mathbb{Q}_p(K\times\{0\}\times\{0\})$ and
  \begin{equation*}
   p^K_{\overline{\xi_2\eta\zeta}} = \vqttell\circ \phi_{\overline{\xi_2\eta\zeta}} \text{ in }
   \widetilde{J}^K\times \{0\}\times \{0\}.
  \end{equation*}
  When $K=(0, 2^{-q})$, we let $p^K_{\overline{\xi_2\eta\zeta}}$ be any polynomial of degree $p$ satisfying $p^K_{\overline{\xi_2\eta\zeta}} (0,0,0)= (\vqttell\circ \phi_{\overline{\xi_2\eta\zeta}})(0,0,0)$.
  Finally, $p_{\overline{\xi_2\eta\zeta}} \in X^q_{S_1}$ is defined piecewise as
  $p^K_{\overline{\xi_2\eta\zeta}} $  in each element. Note that the property of
  right-continuity is crucial for the exactness of the piecewise polynomial at
  mesh nodes.
% Let us define a piecewise polynomial $\widetilde p_{\overline{\xi_2\eta\zeta}} = \vqttell{|_{\phi^\ell_{\overline{\xi_2\eta\zeta}}(S_1)}}\in X_{S_1}^{\ell}$.
% \CM{$\vqttell{|_{\phi_{\overline{\xi_2\eta\zeta}}(S_1)}}$ is not defined on $S_1$}
% %$[2^{i-q-1}, 2^{i-q} + 2^{-\ell}\xi_2)$
% For $j=0,\dots,q$, let $g_{\overline{\xi_2\eta\zeta}}^j$ be a polynomial defined
% on $J_{1,j}^{\ell} + (0, 2^{-\ell}\xi_2)$
% \CM{do you really only want to consider $J^{\ell}_{1, 0}, \dots, J^\ell_{1, q}$
%   or did you mean $J^{q}_{1, j}$? Is including $j=0$ correct?}
% such that $g_{\overline{\xi_2\eta\zeta}}^j(x) = \widetilde p_{\overline{\xi_2\eta\zeta}}(x), x\in J_{1,j}^{\ell} \times \{0\} \times \{0\}$.
% As $g_{\overline{\xi_2\eta\zeta}}^{q+1}$ we allow any polynomial on $(J_{0,0}^{\ell} + (0, 2^{-\ell}\xi_2))\times \{0\} \times \{0\}$ of degree $p$ satisfying $g_{\overline{\xi_2\eta\zeta}}^{q+1}((2^{-\ell} \xi_2,0,0)) = \widetilde p_{\overline{\xi_2\eta\zeta}}((2^{-\ell} \xi_2,0,0))$.
% \CM{wouldn't one take $g_{\overline{\xi_2\eta\zeta}}^{q+1}= \widetilde p_{\overline{\xi_2\eta\zeta}}$?}
% Then, we obtain the required piecewise polynomial by letting $p_{\overline{\xi_2\eta\zeta}}(x) = g_{\overline{\xi_2\eta\zeta}}^j(x + 2^{-\ell} \xi_2)$ for $x\in J_{1,j}^{\ell}$, $j=0,\dots,q$ and $p_{\overline{\xi_2\eta\zeta}}(x) = g_{\overline{\xi_2\eta\zeta}}^{q+1}(x + (2^{-\ell} \xi_2,0,0))$ for $x\in J_{0,0}^{\ell}\times \{0\} \times \{0\}$.
% }

Remarking that there exists a constant $C>0$ such thath for every $\pmax, q$ holds
$\dim(X_{S_1}^\ell)\leq C q \pmax$ and taking a basis $\{e_n\}_n$ of $X_{S_1}^\ell$, 
we can write
\begin{equation*}
    V^{(q)} = BW, 
\end{equation*}
where $B_{\xi_1, n}= e_n((2^{-q}\xi_1, 0, 0))$, for
$n=1, \dots, \dim(X_{S_1}^\ell)$ and $\xi_1$ as above, and $W$
is a $\dim(X_{S_1}^\ell) \times 2^{3\ell-q}$ matrix of coefficients. 
Hence, there exists $\widetilde C>0$ such that for all $q=1, \dots, \ell$ it holds
\begin{equation*}
    r_q = \rank(V^{(q)}) \leq \dim(X_{S_1}^\ell)\leq \widetilde C  \ell^{\gev+1} .
\end{equation*}
We now consider the case where $\ell < q < 2\ell$ and denote $\tq = q-\ell$. 
Then, 
  \begin{equation*}
    V^{(q)}_{\overline{\xi\eta_1}, \overline{\eta_2\zeta}} 
    =
    \vqttell(x_{\xi, \overline{\eta_1\eta_2}, \zeta}),
  \end{equation*}
  for $\xi, \zeta = 0, \dots, 2^\ell-1$, $\eta_1 = 0, \dots, 2^{\tq}-1$, and
  $\eta_2 = 0, \dots, 2^{\ell-\tq}-1$. 
  We introduce the two-dimensional slice
  \begin{equation*}
    S_2^{\tq} = \{0\}\times(0, 2^{-\tq})\times(0,1),
  \end{equation*}
  with associated mesh
\begin{equation*}
  \cG^{\tq}_{S_2^{\tq}} 
  = 
  % \{\{0\}\times (0, 2^{-\tq})\times K\text{ for all }K\in \cG^{\tq}_{\mathrm{1d}}\},
  \{\{0\} \times K\text{ for all }K\in \cG^{\ell}_{\mathrm{2d}}\text{ such that } K \subset (0, 2^{-\tq})\times (0,1)\},
\end{equation*}
and the corresponding FE space
\begin{equation*}
  X^{\tq}_{S_2^{\tq}} 
  = 
  \{v\in H^1(S_2^{\tq}): v_{|_K} \in \mathbb{Q}_{\pmax}(K) \text{ for all }K\in \cG^{\tq}_{S_2^{\tq}}\}.
\end{equation*}
  Consider the affine transformations
\begin{equation*}
    \psi_{i, j} : (0, x_2, x_3) \mapsto ( 2^{-\ell}i, x_2+2^{-\tq}j, x_3).
\end{equation*}
  Then,
  $x_{\xi, \overline{\eta_1\eta_2}, \zeta} \in \psi_{\xi, \eta_1}(S^{\tq}_2)$.
  Moreover, for each $\xi, \eta_1$ there exists a
  piecewise polynomial $p_{\overline{\xi\eta_1}}\in X^{\tq}_{S_2}$ such that 
  \begin{equation*}
    \left( \vqttell \circ \psi_{\xi, \eta_1} \right)\left( \left(\psi_{\xi, \eta_1}\right)^{-1}(x_{\xi, \overline{\eta_1\eta_2}, \zeta})  \right)
    =
     p_{\overline{\xi\eta_1}}(x_{0, \eta_2, \zeta}) 
  \end{equation*}
     for all $\eta_2= 0, \dots, 2^{\ell-\tq}-1$ and $\zeta=0, \dots 2^\ell-1$,
  see Figure \ref{fig:slice2}.
%%%%%%%%%%%%%%%%%%%%%%%%%%%%%%%%%%%%%%%%%%%%%%%%%%%%%%
  \begin{figure}
    \centering
    \includegraphics[width=.6\textwidth]{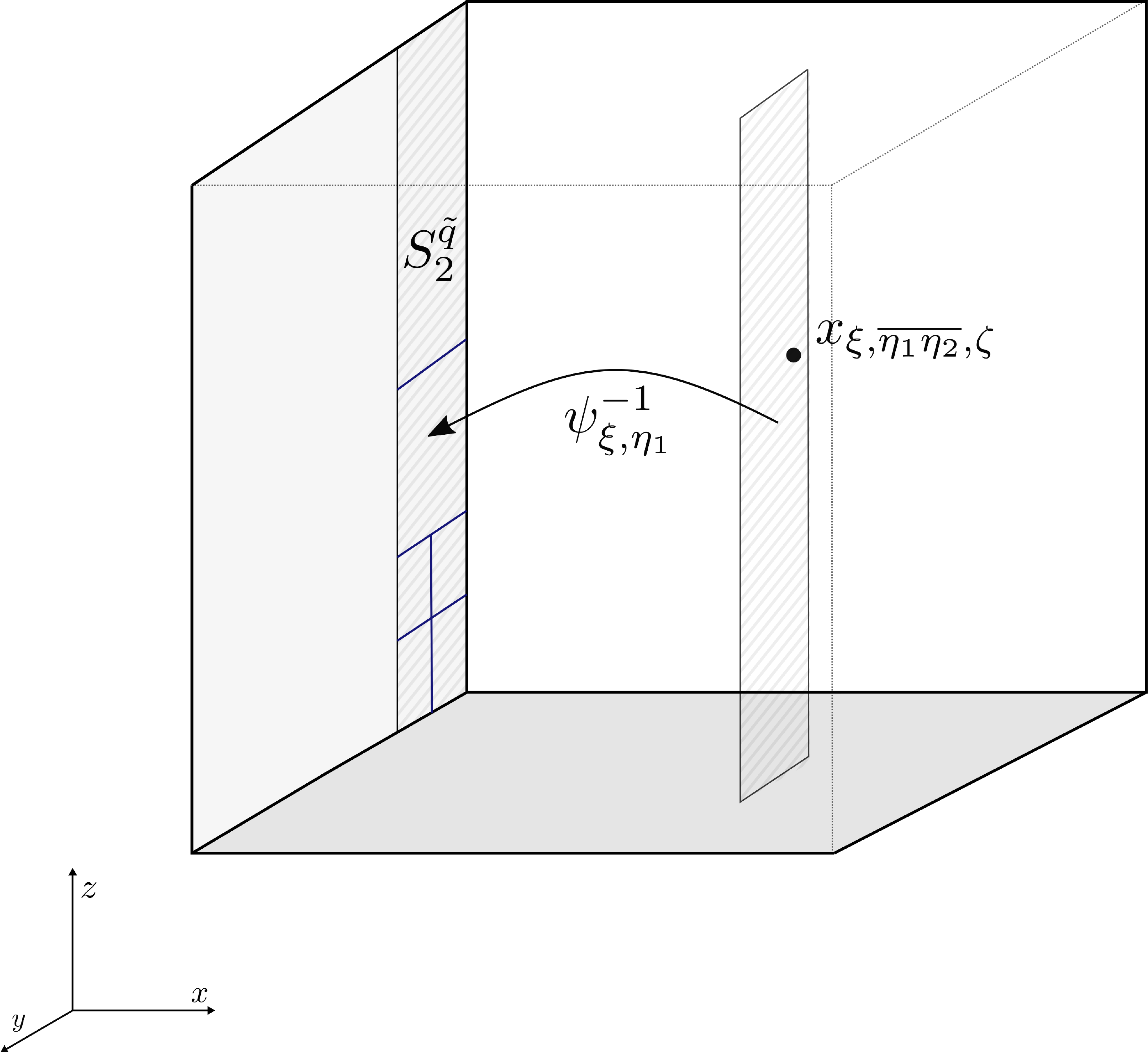}
    \caption{Slice $S_2^\tq$, with geometric mesh $\cG^{\tq}_{S_2}$ (in blue) 
    and action of domain mapping $\psi_{\xi, \eta_1}$.}
    \label{fig:slice2}
  \end{figure}
%%%%%%%%%%%%%%%%%%%%%%%%%%%%%%%%%%%%%%%%%%%%%%%%%%%%%
  Since $\dim (X^{\tq}_{S_2})\leq C\ell\pmax^2$, we obtain, reasoning as before,
  \begin{equation*}
    r_q = \rank(V^{(q)}) \leq C \ell p^2\leq C \ell^{2\gev+1} \qquad q = \ell+1,\dots, 2\ell-1.
  \end{equation*}

  It remains to consider $q$ with $2\ell\leq q< 3\ell$. 
  We sketch their treatment, which follows the same line of reasoning as in the preceding cases. 
  Every row of the
  unfolding matrix $V^{(q)}$ contains the evaluation of $\vqttell$ on 
  $2^{3\ell - q}$ equispaced 
  points belonging to a line parallel to the $z$-axis. 
  Hence, there exists
  a space of piecewise polynomials with less than $C (3\ell - q)\pmax$ degrees of
  freedom such that each row of $V^{(q)}$ can be written as linear combination of elements
  of the space, thus implying the existence of a constant $C>0$ such that
  \begin{equation*}
    r_q = \rank(V^{(q)}) \leq C \ell \pmax \leq  C \ell^{\delta+1} \qquad q = 2\ell,\dots, 3\ell-1.
  \end{equation*}
  The proof is concluded by remarking that
  \begin{equation*}
    \Ndof \leq C \sum_{q=1}^{3\ell-1} r_qr_{q+1} \leq C \ell^{4\gev +3},
  \end{equation*}
  choosing $\ell \simeq |\log\epsilon|$ and using Proposition \ref{prop:hp-reinterp}.
\myqed
\end{proof}

\subsection{Rank bounds for transposed order QTT representations}
\label{sec:RkBdTrOrdQTT}
%%%%%%%%%%%%%%%%%%%%%%%%%%%%%%%%%%%%%%%%%%%%%%%%%%%%%%%%%%%%%%%%%%%%
\begin{lemma}
  \label{lemma:qt3-rank}
  Let $0<\epsilon_0 < 1$ and $u\in X$. 
  Then, for all $0 < \epsilon\leq \epsilon_0$ 
  there exists $\ell\in\mathbb{N}$ and $\vqtttell = \fP^\ell u$ 
  such that $\| u - \vqtttell\|_{H^1(Q)} \leq \epsilon$
   and $\vqtttell$ admits a transposed QTT representation with
  \begin{equation*}
    % \Ndof \leq C \ell^{3\gev+3}
    \Ndof \leq C |\log\epsilon|^{6\gev+1}
  \end{equation*}
  degrees of freedom, with $C>0$ independent of $\epsilon$. 
\end{lemma}
\begin{proof}
By Proposition \ref{prop:hp-reinterp}, 
for all $0 < \epsilon \leq \epsilon_0$ 
exists $\ell\in \mathbb{N}$ such that 
$\| u - \fP^\ell u \|_{H^1(Q)}\leq \epsilon$,
with
  \begin{equation*}
    \ell \leq C |\log\epsilon|,
  \end{equation*}
  and $C$ independent of $\epsilon$.
Let then $q\in\{1, \dots, \ell-1\}$; we consider the $q$th unfolding matrix of the
transposed QTT representation of $\vqtttell$, i.e., the $q$th unfolding matrix of
$\fT^{\qttt}(\scrA^\ell(\vqtttell))$, as defined in \eqref{eq:transposedQTT-tensorization}.
This is the matrix with entries
  \begin{equation}
    \label{eq:qt3-unfolding}
    U^{(q)}_{\overline{\xi_1\eta_1\zeta_1}, \overline{\xi_2\eta_2\zeta_2}} 
    = 
    \scrA^\ell(\vqtttell)_{\overline{\xi_1\xi_2}, \overline{\eta_1\eta_2}, \overline{\zeta_1\zeta_2}} 
    = 
    \vqtttell(x_{\overline{\xi_1\xi_2}, \overline{\eta_1\eta_2}, \overline{\zeta_1\zeta_2}})
  \end{equation}
  with $\xi_1, \eta_1, \zeta_1 \in \{0, \dots, 2^q-1\}$ and $\xi_2,\eta_2,
  \zeta_2\in \{0, \dots, 2^{\ell-q}-1\}$. Following the proof of Lemma
  \ref{lemma:qtt-rank}, we introduce a reference cube
  \begin{equation*}
    S^q = (0, 2^{-q})^3
  \end{equation*}
  and a reference space
  \begin{equation*}
    X_{S^q} = \mathbb{Q}_{\pmax}(S^q).
  \end{equation*}
  Then, let $\phi_{i,j,k}$ be the map from the reference square
  to the space of points of the row $\overline{ijk}$ of the
  unfolding matrix, i.e.,
  \begin{equation*}
    \phi_{i,j,k} :(x_1,x_2,x_3)\mapsto (x_1 + 2^{-q}i, x_2 + 2^{-q}j, x_3 + 2^{-q}k),
  \end{equation*}
  so that $x_{\overline{\xi_1\xi_2}, \overline{\eta_1\eta_2},
    \overline{\zeta_1\zeta_2}} \in \phi_{\xi_1, \eta_1, \zeta_1}(S^q)$ for all
  $\xi_1, \eta_1, \zeta_1\in \{0, \dots, 2^q-1\}$, 
see Figure \ref{fig:slice-trans}.
%%%%%%%%%%%%%%%%%%%%%%%%%%%%%%%%%%%%%%%%%%%%%%%%%%%%
\begin{figure}
    \centering
    \includegraphics[width=.6\textwidth]{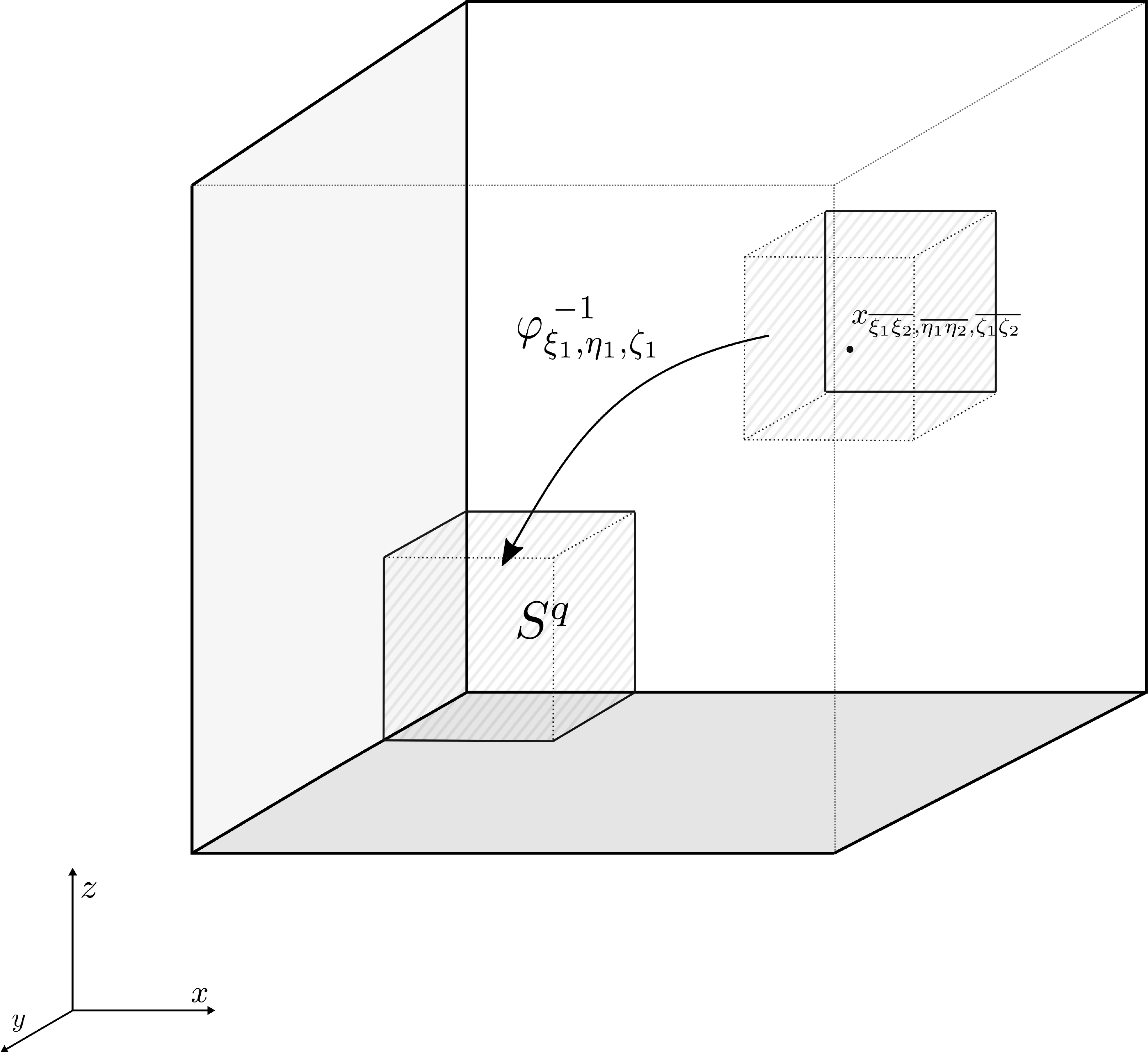}
    \caption{Reference cube $S^q$ for transposed order QTT and
      representation of $\phi_{\xi_1, \eta_1, \zeta_1}$ .}
    \label{fig:slice-trans}
  \end{figure}
%%%%%%%%%%%%%%%%%%%%%%%%%%%%%%%%%%%%%%%%%%%%%%%%%%%%%
  Remark now that 
  \begin{equation*}
    \left\{x_{\overline{\xi_1\xi_2}, \overline{\eta_1\eta_2},
    \overline{\zeta_1\zeta_2}}\right\}_{\substack{\xi_1, \eta_1,\zeta_1\in \{0, \dots, 2^q-1\}\\\xi_2, \eta_2, \zeta_2\in \{0, \dots, 2^{\ell-q}-1\}}} \subset \overline{I^q_{\xi_1}}\times\overline{I^q_{\eta_1}}\times\overline{I^q_{\zeta_1}}.
  \end{equation*}
  Suppose now $\overline{\xi_1\eta_1\zeta_1}>0$: 
  then, there exists a single element $K\in \cG^\ell$ 
  such that 
  \begin{equation*}
    \overline{I^q_{\xi_1}}\times\overline{I^q_{\eta_1}}\times\overline{I^q_{\zeta_1}}
  \subset \overline{K},
  \end{equation*}
  thus,
  \begin{equation*}
  \vqtttell \circ \phi_{\xi_1, \eta_1, \zeta_1} \in X_{S^q}.
  \end{equation*}
  % where $S^q_{\xi_1, \eta_1, \zeta_1} = \phi_{\xi_1, \eta_1, \zeta_1}(S^q)$.
  This implies that, for each $\xi_1, \eta_1, \zeta_1$: $\overline{\xi_1\eta_1\zeta_1}>0$, there exists a polynomial
  $p_{\overline{\xi_1\eta_1\zeta_1}}\in X_{S^q}$ that interpolates
  $\vqtttell\circ\phi_{\xi_1, \eta_1, \zeta_1}$ in the reference cube, i.e., 
  \begin{equation*}
  \left(\vqtttell \circ \phi_{\xi_1, \eta_1, \zeta_1} \right)\left( \phi_{\xi_1, \eta_1, \zeta_1}^{-1} \left(x_{\overline{\xi_1\xi_2}, \overline{\eta_1\eta_2},
    \overline{\zeta_1\zeta_2}} \right)\right) = p_{\overline{\xi_1\eta_1\zeta_1}}(x_{\xi_2, \eta_2, \zeta_2}).
  \end{equation*}
  Note that for $\xi_1 = \eta_1 = \zeta_1 = 0$, the function $\vqtttell \circ
  \phi_{\xi_1, \eta_1, \zeta_1}$ is not a polynomial, which increases the
  dimension of the row space of $U^{(q)}$ by $1$.
  Since $\dim(X_{S^q}) \leq C \pmax^3$, and using the same arguments as in the proof of Lemma \ref{lemma:qtt-rank} it can be concluded that
  \begin{equation*}
    r_q = \rank(U^{(q)}) \leq\dim(X_{S^q}) + 1 \leq  C \ell^{3\gev}
  \end{equation*}
which gives, due to \eqref{eq:Ndof-TT}, a total number of degrees of freedom
$\Ndof \leq C \ell^{6\gev+1}$. 
The fact that $\ell\lesssim |\log\epsilon|$ concludes the proof. \myqed
\end{proof}
%%%%%%%%%%%%%%%%%%%%%%%%%%%%%%%%%%%%%%%%%%%%%%%%%%%%%%%%%%%%%%%%%%%%%%%%%%%%%%%%%%%%%%
\subsection{Rank bounds of TQTT approximations}
\label{sec:TQTT}
%%%%%%%%%%%%%%%%%%%%%%%%%%%%%%%%%%%%%%%%%%%%%%%%%%%%%%%%%%%%%%%%%%%%%%%%%%%%%%%%%%%%%%
In this section, we prove rank bounds for the TQTT approximation. We start 
  by proving, in Lemma \ref{lemma:qtt1d} below, rank bounds for the \emph{block QTT decomposition}
  of collections of piecewise polynomial functions.
  Block QTT decompositions are precisely defined in the following.
\begin{definition}[Block QTT decomposition]
  \label{def:bQTT}
  Let $\ell\in \mathbb{N}$ and let
  $A_\alpha : \{0,\dots,2^\ell -1\} \to \mathbb{R}$ for every $\alpha=1,\dots,s$.
  We say that the collection $\{A_\alpha\}_{\alpha}$ admits a 
  \emph{block QTT decomposition} with ranks $r_0, \dots, r_{\ell}$
  and cores $U^1, U^2, \dots, U^\ell$, if
  \begin{equation*}
    A_\alpha \left(\overline{i_1\dots i_\ell}\right)
    = 
    U^1_\alpha(i_1) U^2(i_1)\cdots U^\ell(i_\ell) 
    \quad \forall (i_1, \dots, i_\ell)\in \{0,1\}^\ell, \, \forall \alpha\in \{1, \dots, s\},
  \end{equation*}
  for all $i_n\in \{0,1\}$, and with $U^n:\{0, 1\}\to \mathbb{R}^{r_{n-1}\times r_n}$
  for all $n = 1, \dots, \ell$. By $U^1_\alpha(i_1)$ we indicate the
    $\alpha$th row of $U^1(i_1)$.
  We have the restriction on the ranks $r_0 = s$, $r_\ell = 1$.
\end{definition}
We also need the definition, 
on the geometric mesh $\cG^\ell_{\mathrm{1d}}$ (see \eqref{eq:hpmesh1d}), 
of the univariate $hp$-FE space 
\[
\Xhpelloned =  \{v\in H^1((0,1)): 
v_{|_{K}}\in \mathbb{Q}_{p}(K), \text{ for all }K\in \cG^\ell_{\mathrm{1d}}\}\;.
\]
\begin{lemma}
\label{lemma:qtt1d}
Let $\{w_\alpha \}_{\alpha=1}^s\subset \Xhpelloned$, 
and let $W_\alpha : \{0,\dots,2^\ell -1\} \to \mathbb{R}$ 
be such that $W_\alpha(i) =  w_\alpha(2^{-\ell}i)$, 
for all $\alpha=1,\dots,s$ and $i=0,\dots,2^\ell-1$.
Then the collection $\{W_\alpha\}_{\alpha=1}^s$
admits a block QTT representation with ranks $r_n \leq s + p + 1$ 
for all $n=1, \dots, \ell-1$.
\end{lemma}
\begin{proof}
	We provide a constructive proof with explicit formulas for the QTT cores.
In the proof, the $m$-by-$n$ matrix with zero entries will be denoted by
$\mathbf{O}_{m\times n}$, while the $n$-by-$n$ identity matrix will be written
as $\mathbf{I}_n$. Let 
	\[
		q(x) = \pc^\top \mathbf{m}(x),
	\] 
        where
	\[
		\pc = 
		\begin{bmatrix}
			a_0 & a_1 & \dots & a_p
		\end{bmatrix}^\top
		\quad
		\mathbf{m}(x) = 
		\begin{bmatrix}
			1 & x & x^2 & \dots & x^p
		\end{bmatrix}^\top
	\]
	be a polynomial of degree $\leq p$ and 
        $\mathbf{q} = \{q(x_i)\}_{i=0}^{2^\ell-1} \in\mathbb{R}^{2^\ell}$, $x_i = 2^{-\ell}\, i$.
	Then $\mathbf{q}$ admits the \emph{exact, explicit QTT representation}~\cite{IOConstRepFns}, 
        \cite{Khoromskij:2011:QuanticsApprox}, \cite{Grasedyck:2010:HTens} 
        with ranks $r_k = p+1$:
	\[
	\mathbf{q}_i = Q^1(i_1) Q^2(i_2) \dots Q^\ell (i_\ell), \quad 
        i = \overline{i_1\dots i_\ell},
	\]
	where 
	\begin{equation}\label{eq:qttpoly}
	\begin{split}
		&
    Q^1(i_1) = 
		\pmb{\phi}^\mathbf{a}(2^{-1}i_1) \equiv 
		\begin{bmatrix}
			\varphi_0^\mathbf{a} \left(2^{-1} i_1\right) & \dots & \varphi_p^\mathbf{a} \left(2^{-1} i_1 \right)
		\end{bmatrix}, 
		\\
    & \qquad\qquad\text{with } \varphi_m^\mathbf{a}(x) = a_m + \sum_{k=m+1}^p a_k C_k^m x^{k-m}, \ \ m = 0,\dots,p, \\
		&Q^k(i_k) = \mathbf{Q}(2^{-k} i_k),\ \ k=2,\dots,\ell-1,
    \\
    & \qquad \qquad\text{with }
			\mathbf{Q}(x)_{ij} = 
			\begin{cases}
				C_i^{i-j} x^{i-j} &i > j, \\
				C_i^{0} & i = j, \\
				0 & \text{otherwise},
			\end{cases}
			\quad i,j=0,\dots,p, \\
		& Q^\ell(i_\ell) = \mathbf{m} (2^{-\ell} i_\ell).
	\end{split}
	\end{equation}

Let now $w_\alpha\in \Xhpelloned$, $\alpha=1,\dots,s$ be given by 
polynomials with the coefficients ${\pc_k}^{\mkern-5mu (\alpha)}$, $k=1,\dots,\ell + 1$ 
on subintervals $[2^{-k},2^{1-k})$ for $k = 1,\dots,\ell$ and $[0,2^{-\ell})$ for $k=\ell + 1$.

Consider the points $x_{\overline{i_1\dots i_\ell}}$.
The case of $i_1=1$ and any $i_n\in\{0,1\}$, $n=2,\dots,\ell$ corresponds to the 
equispaced mesh points $x_{\overline{i_1\dots i_\ell}}$ from the interval $[1/2,1)$. 
Similarly, the case $i_1=\dots=i_{k-1} =0$, $i_k=1$ and any $i_n\in\{0,1\}$, $n=k+1,\dots,\ell$ 
corresponds to $x_{\overline{i_1\dots i_\ell}}\in [2^{-k},2^{1-k})$.

We conclude that $w_\alpha(x_{\overline{0\dots 0 1 i_{k+1}\dots i_\ell}})$ 
are polynomials sampled in equidistant points.
By utilizing this fact, explicit formulas~\eqref{eq:qttpoly} 
for the polynomial parts and combining expressions for each of them, we obtain:
	\[
	W_\alpha (i) = 
%	\begin{bmatrix} 
% 		 u_1(2^{-\ell}i) \\ \vdots \\  u_s(2^{-\ell}i)
%        \end{bmatrix}
w_\alpha(2^{-\ell}i)
        = W^1_\alpha(i_1) W^2(i_2) \dots W^\ell (i_\ell), \quad i = \overline{i_1\dots i_\ell},
	\]
	where
	\[
	\begin{split}
		&W^1_:(i_1) =
		\begin{cases}
		\begin{bmatrix}
			\mathbf{\Phi}_1 (2^{-1}i_1) & \mathbf{O}_{s\times s}
		\end{bmatrix},
		& i_1 = 1 \\[10pt]
		\begin{bmatrix}
			\mathbf{O}_{s\times (p+1)} & \mathbf{I}_s
		\end{bmatrix}, 
		& i_1 = 0
		\end{cases}
		\\
		& W^k (i_k) = 
		\begin{cases}
		\begin{bmatrix}
			\mathbf{Q}(2^{-k} i_k) & \mathbf{O}_{(p+1)\times s} \\
			\mathbf{\Phi}_k (2^{-1}i_k) & \mathbf{O}_{s\times s}
		\end{bmatrix},
		& i_k = 1\\[15pt]
		\begin{bmatrix}
			\mathbf{Q}(2^{-k} i_k)  & \mathbf{O}_{(p+1)\times s} \\
			\mathbf{O}_{s \times (p+1)} & \mathbf{I}_s
		\end{bmatrix},
		& i_k = 0
		\end{cases}
		\qquad k = 2,\dots, \ell-1
		\\
		&W^\ell(i_\ell) =
		\begin{cases}
		\begin{bmatrix}
			\mathbf{m}(2^{-\ell} i_\ell) \\ 
			(\mathbf{a}_\ell^{(1)})^\top \mathbf{m}(2^{-\ell} i_\ell) \\
			\dots \\
			(\mathbf{a}_\ell^{(s)})^\top \mathbf{m}(2^{-\ell} i_\ell) 
		\end{bmatrix}, 
		& i_\ell = 1 \\[15pt]
		\begin{bmatrix}
			\mathbf{m}(2^{-\ell} i_\ell) \\ 
			(\mathbf{a}_{\ell+1}^{(1)})^\top \mathbf{m}(2^{-\ell} i_\ell) \\
			\dots \\
			(\mathbf{a}_{\ell+1}^{(s)})^\top \mathbf{m}(2^{-\ell} i_\ell) 
		\end{bmatrix}, 
		& i_\ell = 0
		\end{cases}
	\end{split}
	\]
and where we used the notation
\[
\mathbf{\Phi}_k(x) \equiv
	\begin{bmatrix}
		\pmb{\phi}^{\mathbf{a}_k^{(1)}}(x) \\
		\vdots\\
		\pmb{\phi}^{\mathbf{a}_k^{(s)}}(x) \\
	\end{bmatrix}\in\mathbb{R}^{s\times (p+1)}, \quad k = 1,\dots,\ell.
\]
For fixed $i_k$, $k=2,\dots,\ell-1$, 
the matrices $W^k(i_k)$ are of size $(p+s+1)\times (p+s+1)$. 
We conclude that the ranks of the collection $\{W_\alpha \}_\alpha$ are bounded from above by $p+s+1$.
This completes the proof.
\myqed
\end{proof}

\begin{lemma}
  \label{lemma:tqtt-rank}
  Let $0<\epsilon_0< 1$ and $u\in X$. 
  Then, there exists a constant $C>0$ (depending on $u$ and on $\epsilon_0$) 
  such that for all $0 < \epsilon\leq \epsilon_0$ 
  there exists $\ell\in\mathbb{N}$ and $\vqttell = \fP^\ell u$, 
  such that $\| u - \vqttell\|_{H^1(Q)} \leq \epsilon$.
  Furthermore, $\vqttell$ admits a TQTT representation with
  \begin{equation*}
    % \Ndof \leq C \ell^{3\gev+3}
    \Ndof \leq C |\log\epsilon|^{3\gev+3}
  \end{equation*}
  degrees of freedom.
\end{lemma}
\begin{proof}
Consider the three-dimensional array $V = \scrA^\ell(\vtqttell) $, with entries
\begin{equation*}
  V_{\xi, \eta, \zeta} = \vtqttell(x_{\xi, \eta, \zeta}).
\end{equation*}
We start by showing that the $1$-rank of $V$, i.e., the column rank of
\begin{equation*}
  V_{(1)} = V_{\xi, \overline{\eta\zeta}}
\end{equation*}
is bounded by $C\ell p$, for $C>0$ independent of $\ell, p$. Indeed, for all $\eta,
\zeta \in \{0, \dots, 2^\ell-1\}$, the column $\{V_{\xi,
  \overline{\eta\zeta}}\}_\xi$ contains the evaluation of a piecewise polynomial 
  in the finite dimensional space
\begin{equation*}
  \Xhpelloned 
  = 
  \{v\in H^1((0,1)): v_{|_K}\in \mathbb{Q}_{\pmax}(K), 
   \text{ for all }K\in \cG^\ell_{\mathrm{1d}}\}, 
\end{equation*}
i.e., 
there exists $p_{\overline{\eta\zeta}}\in \Xhpelloned$ 
such that
\begin{equation*}
  V_{\xi, \overline{\eta\zeta}} = p_{\overline{\eta\zeta}}(2^{-\ell}\xi) 
  \qquad \text{for all }\xi\in \{0, \dots , 2^\ell-1\}.
\end{equation*}
Since 
$\dim\left(\Xhpelloned \right) \leq C \ell\pmax $, 
we
have that $\rank(V_{(1)}) \leq C\ell^{\gev+1}$.
We write $R = \dim\left(\Xhpelloned \right)$, 
denote by $\{b_n\}_{n=1}^R$ a basis for $\Xhpelloned$,
and repeat the argument above in the other two cardinal directions. 
It follows that there exists a Tucker decomposition
such that, for $\xi, \eta, \zeta\in \{0, \dots, 2^\ell-1\}$,
\begin{equation}
\label{eq:Tucker}
 V_{\xi, \eta, \zeta} = \sum_{\beta_1 = 1}^{R}\sum_{\beta_2 = 1}^{R}\sum_{\beta_3 = 1}^{R} 
G_{\beta_1, \beta_2, \beta_3} U_{\beta_1} (\xi)V_{\beta_2}(\eta) W_{\beta_3}(\zeta) 
\end{equation}
where $R\leq C\ell^{\gev+1}$ and such that
\begin{equation}
  \label{eq:Tucker-cores}
  U_{\beta_1} (\xi) = b_{\beta_1}(2^{-\ell }\xi),\qquad
  V_{\beta_2} (\eta) = b_{\beta_2}(2^{-\ell }\eta),\qquad
  W_{\beta_3} (\zeta) = b_{\beta_3}(2^{-\ell }\zeta),
\end{equation}
for all $\beta_1, \beta_2. \beta_3\in \{1, \dots, R\}$,
see \cite[Chapter 8]{Hackb_TensorBook}.%\cite[Section 4]{Kolda2009}.

Applying Lemma~\ref{lemma:qtt1d} to the Tucker factors $U, V, W$, 
we obtain their block QTT representation with tensor ranks 
$\{R, r_\mathrm{QTT},\dots, r_\mathrm{QTT}, 1\}$ bounded as
\[
	r_\mathrm{QTT} = R + p + 1 \leq C \ell^{\gev+1}.
\]
To store Tucker factors in the block QTT format, we need to store 
\[
{N}_{\mathrm{dof}}^{\mathrm{factors}} 
= 
3\cdot 2 \cdot (r_\mathrm{QTT}R + r_\mathrm{QTT}^2 (\ell-2) + r_\mathrm{QTT}) 
\leq C \ell^{2\gev+3}
\] 
parameters of the decomposition.
Storing the Tucker core requires an 
additional
\[
{N}_{\mathrm{dof}}^{\mathrm{core}} = R^3 \leq C \ell^{3\gev+3}
\]
parameters.
This gives the following 
bound for the overall number of degrees of freedom 
in the TQTT representation
\begin{equation*}
{N}_{\mathrm{dof}} 
= 
{N}_{\mathrm{dof}}^{\mathrm{core}} 
+ 
{N}_{\mathrm{dof}}^{\mathrm{factors}} \leq C \ell^{3\gev+3}\;.
\end{equation*}
Choosing $\ell\simeq |\log\epsilon|$ and using Proposition \ref{prop:hp-reinterp} 
completes the proof.
\myqed
\end{proof}
\subsection{Exponential convergence of QTT approximations of $u \in \cJ^\varpi_\gamma(Q)$}
\label{sec:ExpCnvQTTJgam}
From Proposition \ref{prop:hp-reinterp} and Lemmas \ref{lemma:qtt-rank}, 
\ref{lemma:qt3-rank}, and \ref{lemma:tqtt-rank}, we
obtain the following estimate for the QTT-Finite Element 
approximation of functions in $\cJ^\varpi_\gamma(Q)$.
In the following theorem, 
we introduce a tag $\texttt{qtd} \in \{ \texttt{qtt}, \texttt{tqtt} , \texttt{qt3}\}$, 
which generically denotes quantized tensor decomposition.
\begin{theorem}
  \label{th:QTT-analytic}
  Assume $\gamma>3/2$, $C_u>0$, $A_u>0$, $\gev\geq 1$, and $0< \epsilon_0 \ll 1$. 
  Furthermore, assume the function $u$ belongs to the weighted Gevrey class 
  $u\in \cJ^\varpi_\gamma(Q; C_u, A_u, \gamma, \gev)\cap H^1_0(Q)$.
  Then, there exists a constant $C>0$ such that,
  for all $0< \epsilon\leq \epsilon_0$, 
  there exists 
  $\ell \in \mathbb{N}$ and $\vqtdell \in \Xqttell$ 
  such that
  \begin{equation*}
    \| u - \vqtdell \|_{H^1(Q)} \leq \epsilon
  \end{equation*}
  and $\vqtdell$ admits a representation with
  \begin{equation*}
    \Ndof \leq C |\log\epsilon|^{\kappa}
  \end{equation*}
  degrees of freedom, with 
  \begin{equation*}
\kappa =
\begin{cases}
  4\gev+3 & \text{for classic QTT}\\
  6\gev+1 & \text{for transposed order QTT}\\
  3\gev+3 &  \text{for Tucker-QTT}
\end{cases}
\;.
  \end{equation*}
\end{theorem}
\begin{remark}[Rank bounds of QTT formatted approximations of two-dimensional corner singularities]
\label{rmk:QTTcrnr}
  Using the same techniques for the two-dimensional case 
  (which was already considered for the transposed order QTT in \cite{Kazeev2018} 
  in the analytic class, i.e. for $\gev = 1$) 
  results in the bound
\begin{equation*}
\kappa =
\begin{cases}
  2\gev+3 & \text{for classic QTT,}\\
  4\gev+1 & \text{for transposed order QTT}.
\end{cases}
\end{equation*}
In the case of two spatial variables,
the Tucker-QTT format is easily reduced to the classic QTT format 
for the index ordering $i_\ell,\dots,i_1,j_1,\dots,j_\ell$.
\end{remark}
\input{num_exp.tex}
\section{Conclusion}
\label{sec:Concl}
We considered several formats of quantized tensor-train decompositions and
proved tensor rank bounds for the approximation---with a prescribed error $\varepsilon \in (0,1)$ 
in $H^1(Q)$---of
several classes of Gevrey-smooth functions in the unit cube $Q=(0,1)^3$,
with one point singularity situated at the origin. 
In particular, we considered singularities from  Gevrey-type and 
analytic function spaces with regularity quantified by corresponding derivative bounds 
in weighted Sobolev norms, with radial weight.
For these singularities,
we extended the $hp$ approximation error analysis in \cite{Schotzau2016,Schotzau2015}
to Gevrey-regular solutions with an isolated point singularity.

We then addressed approximation rate bounds in 
three concrete quantized tensor formats: 
the quantized tensor train (QTT),
the transposed quantized TT (QT3) and the Tucker quantized TT (TQTT) format.
Our theoretical TT rank analysis indicated that the tensor ranks 
and number of degrees of freedom necessary to achieve a prescribed accuracy 
$\varepsilon \in (0,1)$ in norm $H^1(Q)$ in these format 
might depend on the format adopted in the quantized approximation
(as no lower bounds were shown, 
 these conclusions might be an artifact of our proofs). 
Numerical results, however, for several model singular functions confirmed 
the relative rank bounds for the three mentioned formats.
These results point the way to QTT structured solvers for electron structure
problems and for other PDE models where solutions exhibit isolated point singularities;
for example, continua with point defects, nonlinear Schr\"odinger and parabolic PDEs with 
blowup, to name but a few.
\emph{Format-adaptive, quantized approximations}
as were recently proposed in \cite{BGK_BlckBxHT2013,BGTreeAdapt2014,NouyNM2019}
might result in further quantitative improvement of TT ranks for the 
presently considered examples.

While our analysis focused only on functions with singular support consisting of
one isolated point, we emphasize that corresponding rank bounds are obtained for 
functions whose singular support consists of a finite number of (well-separated)
isolated points; the present results imply the same rank bounds as shown here
also for such functions, albeit with the constants in the estimates strongly
depending on the separation of the singular supports.
With further analysis,
the present results extend to other forms of singularities, such as line and face
singularities. 
The details on this shall be reported in \cite{MRS_preparation}.

\appendix
\normalsize
\section{$hp$ approximation in weighted Gevrey classes}
\label{sec:gevrey}
We prove, in this section, the exponential convergence of the $hp$
approximations to functions in the weighted Gevrey class 
$\cJ^{\varpi}_\gamma(Q; C, A, \gev)$ for $C,A>0$, $\gamma>3/2$, $\gev\geq 1$.
Specifically, this corresponds to functions $u\in H^1(Q)$ such that
\begin{equation}
  \label{eq:Gev-1}
  \sum_{\alpham=s}\| r^{s - \gamma} \dalpha u \|_{L^2(Q)} 
  \leq 
  C A^s(s!)^\gev \qquad \text{ for all }s\in \mathbb{N}.
\end{equation}
We recall that the $hp$ space is defined as
\begin{equation*}
  \Xhpell  = \{ v\in H^1(Q): v_{|_{K}} \in \mathbb{Q}_{p}(K), \text{ for all } K\in \cG^\ell\}.
\end{equation*}
The central (novel) result of this section is 
the existence---for Gevrey-regular functions in $\cJ^{\varpi}_\gamma(Q)$---of 
an exponentially convergent, $H^1(Q)$ conforming $hp$-projector
on $1$-irregular geometric meshes of hexahedra, as
stated in the following proposition.
\begin{proposition}
  \label{prop:hp-error}
Let $\gamma \geq \gamma_0 > 3/2$ and $\gev \geq 1$. 
Then, there exists $\Pihpell :
\cJ^\infty_\gamma(Q)\to \Xhpell$ 
such that for all $u\in \cJ^{\varpi}_\gamma(Q; C, A, \gev)$ 
there exist constants $C_{\mathsf{hp}}$ and $b_{\mathsf{hp}}$ such that
\begin{equation}
  \label{eq:hp-error}
  \| u - \Pihpell u \|_{H^1(Q)} \leq C_{\mathsf{hp}} e^{-b_{\mathsf{hp}}\ell}, \qquad \ell \in \mathbb{N},
\end{equation}
provided the uniform polynomial degree is 
$p= c_0 \ell^\gev$ for some constant $c_0>0$ which is independent of $\ell$. 
The constants $C_{\mathsf{hp}}$, $b_{\mathsf{hp}}$ depend on 
the constants $C, A$, and $\gev$ in $\cJ^{\varpi}_\gamma(Q)$.
In terms of  $\Ndof= \dim(\Xhpell)\simeq  \ell^{3\gev+1}$, \eqref{eq:hp-error} reads
\begin{equation}\label{eq:hp-errorN}
\| u - \Pihpell u \|_{H^1(Q)} 
\leq C_{\mathsf{hp}} \exp\left(-\hat{b}_{\mathsf{hp}}\Ndof^{1/(1+3\gev)}\right).
\end{equation}

\end{proposition}
The rest of the section will be devoted to an overview of the construction of the conforming
projector $\Pihpell$. This projector has already been exploited and analyzed
in detail, e.g., in \cite{Schotzau2015,Schotzau2017}; here, we wish to sketch
its construction for the sake of self-containedness and to provide the necessary
detail of the treatment of non-analytic Gevrey-type estimates (i.e., of the
cases where $\gev>1$), which requires some minor modification 
with respect to the setting of \cite{Schotzau2015,Schotzau2017}.
For positive integers $p$ and $s$ such that $1\leq s \leq p$, 
we write $\Psi_{p,s} = (p-s)! / (p+s)!$.
%%%%%%%%%%%%%%%%%%%%%%%%%%%%%%%%%%%%%%%%%%%%%%%%%%%%%%%%%%%%%%%%%%%%%%%%%%%%%%%%
\subsection{Discontinuous projector}
\label{sec:hp-disc}
%
%%%%%%%%%%%%%%%%%%%%%%%%%%%%%%%%%%%%%%%%%%%%%%%%%%%%%%%%%%%%%%%%%%%%%%%%%%%%%%%%
We start by introducing a nonconforming projector.
%%%%%%%%%%%%%%%%%%%%%%%%%%%%%%%%%%%%%%%%%%%%%%%%%%%%%%%%%%%%%%%%%%%%%%%%%%%%%%%%
\subsubsection{Local projector}
\label{sec:LocProj}
%%%%%%%%%%%%%%%%%%%%%%%%%%%%%%%%%%%%%%%%%%%%%%%%%%%%%%%%%%%%%%%%%%%%%%%%%%%%%%%%
We denote the reference interval by $I=(-1, 1)$ and 
the reference cube by $\hK = (-1, 1)^3$. 
We write also $\Hmix(\hK) = H^2(I)\otimes H^2(I)\otimes H^2(I)$, 
where $\otimes$ denotes the Hilbertian tensor product. 
Let $p\geq 3$: as constructed in \cite[Appendix A]{Costabel2005}, 
there exist univariate projectors 
$\hpi_p : H^2(I) \to \mathbb{P}_p(I)$
such that
\begin{equation}
  \label{eq:onedim-prop}
 \left( \pi_p v \right)^{(j)} (\pm 1) = v^{(j)} (\pm 1), \qquad j=0,1,
\end{equation}
see \cite[Lemma 4.1]{Schotzau2015} (the projector $\pi_p$ is denoted $\pi_{p,2}$ there). 
Then, 
the Hilbertian tensor product projector given by
\begin{equation} \label{eq:hPi}
  \hPi_p = \hpi_p \otimes \hpi_p \otimes \hpi_p
\end{equation}
has the following property.
\begin{lemma}{\cite[Remark 5.5]{Schotzau2013b}}
  \label{lemma:ref-proj}
  For every $p\geq 3$ exists a projector 
  $\hPi_p:\Hmix(\hK) \to \mathbb{Q}_p(\hK)$ such that for all $v\in \Hmix(\hK)$
  and all integer $s$ such that $2\leq s \leq p$
  % \begin{equation}
  %   \label{eq:vert-ex}
  %   \left(\hPi_p v\right) (N) = v(N) 
  % \end{equation}
  % for all vertex $N$ of $\hK$, and
  \begin{equation*}
    %\label{eq:disc-approx}
    \| v - \hPi_p v\|_{\Hmix(\hK)}^2 \leq C \Psi_{p-1, s-1} \| v \|^2_{H^{s+5}(\hK)},
  \end{equation*}
  with $C$ independent of $p, s$, and $v$.
\end{lemma}
For all $K\in \cG^\ell$, we introduce the affine
transformation from the reference element to $K$ 
\begin{equation*}
  %\label{eq:Phi}
  \Phi_K : \hK\to K\quad \text{such that}\quad\Phi_K(\hK) = K;
\end{equation*}
it follows
that for $v$ defined on $K$ such that $v\circ \Phi_K \in \Hmix(\hK)$ 
we can define the local projector on $K$ so that
\begin{equation}
  \label{eq:PiK}
  \Pi_p^K v = \left(\hPi_p  (v\circ \Phi_K) \right)\circ \Phi_K^{-1}.
\end{equation}
The projector $\Pi_p^K$ is continuous across regular matching faces.
\begin{lemma}{\cite[Lemma 4.2]{Schotzau2015}}
  \label{lemma:reg-cont}
 Let $K_1, K_2$ be two axiparallel cubes that share one regular face $F$ (i.e., $F$ is a
 full face of both $K_1$ and $K_2$). Then, for $v\in
 H^6(\mathrm{int}(\overline{K}_1\cup \overline{K}_2))$ the piecewise
 polynomial
 \begin{equation*}
   \Pi^{K_1\cup K_2}_p v = 
   \begin{cases}
     \Pi_p^{K_1} v &\text{in }K_1,\\
     \Pi_p^{K_2} v &\text{in }K_2
   \end{cases}
 \end{equation*}
 is continuous across $F$.
\end{lemma}
\begin{remark}
\label{remark:nullface}
By \eqref{eq:onedim-prop} and \eqref{eq:hPi}, 
if a function $v$ on $K$ such that $v\circ\Phi_K\in \Hmix(\hK)$ 
vanishes on a face $F\subset \partial K$, 
then we also have $\left( \Pi^K_pv\right)_{|_F} = 0$.
\end{remark}
%%%%%%%%%%%%%%%%%%%%%%%%%%%%%%%%%%%%%%%%%%%%%%%%%%%%%%%%%%%%%%%%%%%%%%%%%%%%%%%%%%
\subsubsection{Globally discontinuous $hp$ projector}
\label{sec:GlobDishp}
%%%%%%%%%%%%%%%%%%%%%%%%%%%%%%%%%%%%%%%%%%%%%%%%%%%%%%%%%%%%%%%%%%%%%%%%%%%%%%%%%%
Starting from the local, elementwise projector \eqref{eq:PiK},
a global, \emph{discontinuous} projection operator $\Pihpelld$ is defined
in the usual way:
with the non-conforming $hp$-space 
\begin{equation*}
  %\label{eq:Xhpelld}
  \Xhpelld  = \prod_{K\in \cG^\ell} \mathbb{Q}_{p}(K) 
  = \{ v\in L^2(Q): v_{|_{K}} \in \mathbb{Q}_{p}(K), \text{ for all } K\in \cG^\ell\};
\end{equation*}
for all $K\in \cG$ and for $v\in \cJ^\infty_\gamma(Q)$, with $\gamma>3/2$,
\begin{equation}
  \label{eq:Pihpelld}
  \Pihpelld{|_{K}} v{|_{K}} =
  \begin{cases}
    v(0) &\text{if }K\in \cL^\ell_0,\\
    \Pi_p^K v &\text{otherwise}.
  \end{cases}
\end{equation}
Note that $v(0)$ is well defined for $v\in \cJ^1_{\gamma}(Q)$ if $\gamma > 3/2$,
see \cite[Lemma 7.1.3]{Kozlov1997}; hence, \emph{a fortiori}, $\Pihpelld :
\cJ^\infty_\gamma(Q) \to \Xhpelld$ is well defined if $\gamma>3/2$. 
\begin{lemma}
  \label{lemma:hp-d-appx}
  Let $u\in \cJ^\varpi_\gamma(Q; C_u, A_u, \gev)$. 
  Then, if $p\simeq \ell^\gev$, there exist constants $C, b>0$ such that
  \begin{equation*}
    %\label{eq:hp-d-appx}
    \sum_{K\in \cG^\ell} \frac{1}{h_K^2}\| u - \Pihpelld u\|^2_{L^2(K)} + \| \nabla \left(  u - \Pihpelld u\right) \|^2_{L^2(K)} \leq C e^{-b\ell},
  \end{equation*}
  with $\dim\left( \Xhpelld \right)\simeq \ell^{3\gev+1}$.
\end{lemma}
\begin{proof}
  The proof follows along the lines of the proof of \cite[Proposition
  5.13]{Schotzau2013b}. 
Denote $\eta = u - \Pihpelld u$ and 
$N_K[v]^2 = \| v \|^2_{L^2(K)} / h_K^2 + \| \nabla v \|^2_{L^2(K)}$.

  We start by considering $K\in \cL^\ell_j$ for $j\geq 1$ and write $d_K =
  \dist(K, (0,0,0))$.
  % , noting that $r_{|_K}\simeq h_K$, uniformly in $\ell$.
  By Lemma
  \ref{lemma:ref-proj}, scaling inequalities (see \cite[Equations
  (5.26)--(5.31)]{Schotzau2013b}), and the regularity of $u$ (see \eqref{eq:Gev-1}),
  \begin{align*}
    N_K[\eta]^2
    &\leq C \Psi_{p-1, s-1}\sum_{s +1 \leq\alpham \leq s+5} d_K^{2\alpham -2} \|\dalpha u\|^2_{L^2(K)}\\
    &\leq C \Psi_{p-1, s-1}\sum_{s +1 \leq\alpham \leq s+5} d_K^{2\gamma-2} \|r^{\alpham-\gamma}\dalpha u\|^2_{L^2(K)}\\
    &\leq C \Psi_{p-1, s-1} 2^{-2(\ell-j)\gamma+2} A_u^{2(s+5)} \left((s+5)!  \right)^{2\gev}.
  \end{align*}
  Then,
  using the fact that for
  sufficiently large $s$ and $c = 2A_u+1$,
  \begin{equation*}
    \Psi_{p-1, s-1} A_u^{2s} ((s+5)!)^{2\gev} 
    \leq 
    C \left( \frac{2A_u}{2A_u+1} \right)^{2c^{-1/\gev}p^{1/\gev}},
  \end{equation*}
see \cite[Equation (42)]{Feischl2018},
choosing $s = (p/c)^{1/\gev}\simeq \ell$, with $c>1$ sufficiently large,
and summing over all mesh layers not touching the origin (``interior mesh layers''), 
we obtain that there exist $C_1, b_1 >0$ such that for every $\ell\geq 1$ holds
  \begin{equation}
    \label{eq:lifting}
    \begin{aligned}
    \sum_{j=1}^\ell\sum_{K\in \cL^\ell_j} N_K[\eta]^2
    & \leq C \Psi_{p-1, s_-1} A_u^{2(s+5)} \left((s+5)!  \right)^{2\gev}
      \leq C e^{-2 b s} \leq C_1 e^{-2b_1\ell}.
    \end{aligned}
  \end{equation}

  We now consider the element $K \in \cL^\ell_0$, i.e., $K=(0, 2^{-\ell})^3$. By
  Hardy's inequality and choosing $\gamma>1$,
  \begin{align*}
    N_K[\eta ]^2 = \frac{1}{h_K^2}\| u - u(0)\|^2_{L^2(K)} + \| \nabla   u  \|^2_{L^2(K)}
    &\leq \| r^{-1}( u - u(0))\|^2_{L^2(K)} + \| \nabla   u  \|^2_{L^2(K)}\\
    &\leq C \| \nabla   u  \|^2_{L^2(K)}\\
    &\leq C h_K^{2(\gamma-1)}\| r^{1-\gamma}\nabla   u  \|^2_{L^2(K)}\\
    &\leq C 2^{-2(\gamma-1)\ell}\|   u  \|^2_{\cJ^1_{\gamma}(Q)} \leq C_2 e^{-b_2\ell}.
  \end{align*}
  Finally, the dimension of the $\mathbb{Q}_p(K)$ space in each element
  $K\in \cG^\ell$ is given by $(1+p)^3$; 
  since each non-terminal mesh layer $\cL^\ell_j$, $j>0$,
  contains $7$ elements, we have that $\dim(\Xhpelld) = (1+p)^3(1+7\ell)$.
  The observation that $p\simeq \ell^\gev$ concludes the proof.
\end{proof}
\subsection{Conforming $hp$ approximation} 
\label{sec:Confhp}
A conforming $hp$ approximation is obtained 
by locally lifting the polynomial face jumps
of the discontinuous, piecewise polynomial approximation. 
Their construction is detailed in 
\cite[Section 5.3]{Schotzau2015}.
\subsubsection{Edge and face liftings}
\label{sec:EdFacLft}
\begin{figure}
  \centering
  \includegraphics[width=\textwidth]{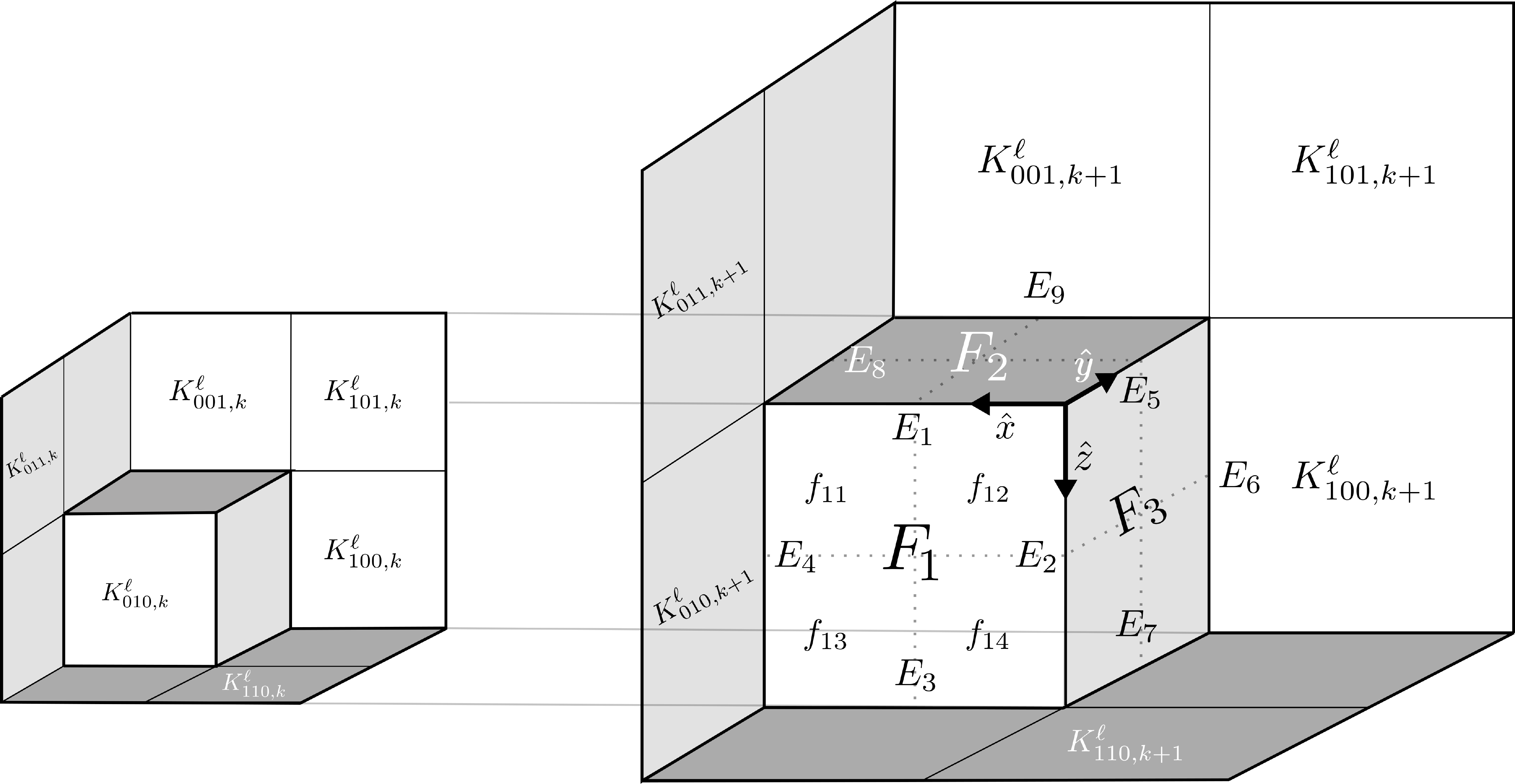}
  \caption{Separation of mesh levels $\cL^\ell_{k}$ (elements moved to the left)
    and $\cL^\ell_{k+1}$, with the interfaces
    $F_1, F_2, F_3$  and edges $E_1, \dots, E_9$ marked. 
    The local 
    system of coordinates is given by $\hat{x}, \hat{y}, \hat{z}$ 
    is also represented (with $\hat{y}$ pointing upwards from the page).}
  \label{fig:liftings}
\end{figure}
Since our discontinuous interpolant is the same as in
\cite{Schotzau2015}, apart from the nonzero constant in $\cL^\ell_0$ 
(see \eqref{eq:Pihpelld} and \cite[Equation (4.10)]{Schotzau2015}),
the construction of the polynomial face-jump liftings can be replicated verbatim 
as in \cite{Schotzau2015}.
We recall it here briefly, 
referring the reader to the aforementioned 
\cite[Section 5.3]{Schotzau2015} for the details.

We start by considering the interface between two mesh levels $\cL^{\ell}_k$ and
$\cL^\ell_{k+1}$, $k\in \mathbb{N}$. 
We introduce a local coordinate system $\hat{x}, \hat{y},
\hat{z}$ and label the faces and edges belonging to the interface as $F_i$,
$i=1,2,3$ and $E_i$, $i=1, \dots, 9$, respectively, see Figure \ref{fig:liftings}. 
Furthermore, we denote by $h_E$ the maximum length of all edges $E_i$.
We refer to Figure \ref{fig:liftings} for the precise numbering of edges and faces 
and for the location of the local system of coordinates.
Given two neighboring elements $K_a$ and $K_b$ with interface $f_{ab} =
\overline{K}_a\cap \overline{K}_b$, the jump of a function
\begin{equation*}
  v =
  \begin{cases}
    v^{K_a}& \text{in }K_a\\
    v^{K_b}& \text{in }K_b
  \end{cases}
\end{equation*}
on $f_{ab}$ is
given by
\begin{equation*}
  \jump{v}_{{f_{ab}}} = v^{K_{a}}_{|_{f_{ab}}} n_{K_a}  + v^{K_{b}}_{|_{f_{ab}}} n_{K_b},
\end{equation*}
where $n_{K_a}$ (resp. $n_{K_b}$) is the normal pointing outwards from element
$K_a$ (resp. $K_b$).

Consider face $F_1$ of Figure \ref{fig:liftings}: 
we define the jump of the discontinuous interpolant on this face as
\begin{equation*}
  \jump{\Pihpelld u}_{F_1} = \jump{\Pihpelld u}_{f_{1j}} \quad \text{on }f_{1j},\, j=1,2,3,4
\end{equation*}
where $f_{1j}$ are the four parts of the face $F_1$, see Figure \ref{fig:liftings}. 
The jumps on the other faces are defined similarly. 
The edge jump, e.g. on $E_1$, is then defined as
$\jump{\Pihpelld u}_{|_{E_1}} = (\jump{\Pihpelld u}_{F_1})_{|_{E_1}}$. 
Let $n$ denote the normal on face $F_1$ pointing outwards from $K^\ell_{010, k+1}$; 
the lifting of the jump on edge $E_1$ is given by
\begin{equation}
  \label{eq:lifting-edge}
  \fL^{E_1} (\Pihpelld u) =
  \begin{cases}
  \begin{aligned}
   \left(\jump{\Pihpelld u}_{|_{E_1}}\cdot n\right)(1-2\hy/h_E&)(1-2\hz/h_E) \\
     &\text{in } K^\ell_{011, k}\cup K^\ell_{111, k}
  \end{aligned}
    \\
    0 \hfill \text{elsewhere.}
  \end{cases}
\end{equation}

After having defined the other edge liftings $\fL^{E_i}$, $i=2,\dots, 9$, 
in the same way, we can introduce the full edge lifting operator
\begin{equation*}
  \fL^{E} = \sum_{i=1}^9 \fL^{E_i}.
\end{equation*}
We now introduce the face lifting operator for the face $F_1$, the other
liftings being derived in the same way. We have
\begin{multline}
  \label{eq:lifting-face}
  \fL^{F_1} (\Pihpelld u)
   \\
  =\begin{cases}
  \begin{aligned}
   \fL^E(\Pihpelld u) + \left(\jump{ \Pihpelld u 
   + \fL^E (\Pihpelld u)}_{|_{F_1}}\cdot n  \right)(1-2\hy/h_E) \\
    \text{in }K^\ell_{n, k}, n\in \{010, 011, 110, 111\}
   \end{aligned}
   \\
   0 \hfill \text{otherwise},
  \end{cases}
\end{multline}

where $n$ is again the normal 
on face $F_1$ pointing outwards from $K^\ell_{010, k+1}$.
Then, the global lifting $\fL^{k}$ 
between mesh levels $\cL^\ell_k$ and $\cL^\ell_{k+1}$ 
is the sum of the local liftings on the three interfaces:
\begin{equation} \label{eq:lifting-def1}
\fL^{k} = \fL^{F_1} + \fL^{F_2} + \fL^{F_3}.
\end{equation}
Note that the lifting thus defined has support only in the elements belonging
to mesh level $\cL^\ell_k$.

We now turn to the terminal layer, i.e., to the jumps of $\Pihpelld u$ between
the element $K^\ell_{000, 0} = (0, 2^{-\ell})^3$ and the elements of
$\cL^\ell_1$. The (three) faces belonging to the interface are all regular,
but $\Pihpelld u$ is defined as a constant in $K^\ell_{000,0}$, see \eqref{eq:Pihpelld}.
One has to lift the nodal jumps at all the nodes of $K^\ell_{000,0}$
except the origin. Then, the same procedure as for the other mesh layers
(applied to the nodally lifted polynomial) gives a lifting operator $\fL^0$.

The full lifting operator is thus given by the sum of the local operators, 
as
  \begin{equation}
    \label{eq:lifting-def}
   \fL = \sum_{k=0}^{\ell-1} \fL^k,
  \end{equation}
  with all $\fL^k$ constructed as in \eqref{eq:lifting-def1}. 
  Such a lifting permits to
  obtain a conforming projector into $\Xhpell$, with approximation error bounded
  by a multiple of the approximation error of the discontinuous operator
  $\Pihpelld$, as stated in the next proposition, that is proven in \cite{Schotzau2015}.
 \begin{proposition}{\cite[Proposition 5.3]{Schotzau2015}}
   \label{prop:lifting}
   The discontinuous projection operator $\Pihpell$ defined in
   \eqref{eq:Pihpelld} and the lifting operator $\mathfrak{L}$ defined in
   \eqref{eq:lifting-def} are such that
   \begin{equation*}
   \Pihpell = \Pihpelld + \fL \Pihpelld : X \to \Xhpell
   \end{equation*}
   is conforming in $H^1(Q)$ and there exists $C>0$ independent of $p$ such that
   \begin{multline*}
  \sum_{K\in \cG^\ell} 
       \frac{1}{h_K^2}\| u - \Pihpell u\|^2_{L^2(K)} 
     + \| \nabla \left(  u - \Pihpell u\right) \|^2_{L^2(K)}     
 \\ \leq C 
p^{18} \sum_{K\in \cG^\ell} \frac{1}{h_K^2}\| \eta\|^2_{L^2(K)} + \| \nabla   \eta \|^2_{L^2(K)}      
   \end{multline*}
 \end{proposition}
Here, $\eta = u - \Pihpelld u$.

The exponential convergence of the conforming approximation, stated in
Proposition \ref{prop:hp-error}, is a direct consequence of the last results.
 \begin{proof}[Proof of Proposition \ref{prop:hp-error}]
   Inequality \eqref{eq:hp-error} follows from Proposition
   \ref{prop:lifting} and Lemma \ref{lemma:hp-d-appx}, once the algebraic term
   in $p$ of inequality \eqref{eq:lifting} has been absorbed in the exponential
   by a change of constants.
 \end{proof}
 \begin{remark}
  \label{remark:nullboundary}
  Recall that
   $\Gamma = \left\{ (x_1,x_2,x_3)\in \partial Q: x_1 x_2 x_3 \not= 0\right\}$
  contains the faces of the boundary of $Q$ not abutting at the singularity.
  All liftings obtained by the operator \eqref{eq:lifting-def} 
  admit traces which vanish on $\Gamma$. 
  I.e., for all $v\in \cJ^\infty_\gamma(Q)$,
  \[
    \left( \Pihpell v \right)_{|_\Gamma} = \left( \Pihpelld \right)_{|_\Gamma}.
    \]
Therefore, by Remark \ref{remark:nullface}, if $v_{|_\Gamma} = 0$,
then also $\left( \Pihpell v \right)_{|_\Gamma} = 0$.
\end{remark}
%%%%%%%%%%%%%%%%%%%%%%%%%%%%%%%%%%%%%%%%%%%%%%%%%%%%%%%%%
\subsection{Combination of patches}
\label{sec:CombPtch}
 We conclude this section by considering the approximation in a domain which
 contains the singular point in its interior. 
 Let then $R = (-1, 1)^3$. The
 definition of the weighted space follows directly from the definition of the
 spaces in $Q$, by keeping the weight $r = |x|$ to be the distance from the origin.

 The
 construction of the graded mesh is done by decomposing $R$ into eight sub-cubes
 of unitary edge and by collecting the elements of the sub-meshes (called here
 ``patches'') obtained by symmetry from $\cG^\ell$. The projector $\Pihpell$ in $R$ 
 can also be straightforwardly constructed by
 combining local projectors obtained by symmetry; 
 we show that it is continuous on inter-patch faces, 
 hence conforming on the whole cube $R$.
 
 We detail the construction for two patches; 
 the rest follows by iterating this argument. 
 Specifically, we consider the two cubes
 \begin{equation*}
  Q^+ = (0,1)^3 \qquad Q^{-} = (-1, 0) \times (0,1)^2,
 \end{equation*}
 and introduce the reflection operator
 \begin{equation*}
 \psi^\pm : Q^+\to Q^- \qquad \psi^\pm :(x_1, x_2, x_3) \mapsto (-x_1, x_2, x_3).
 \end{equation*}
 Note that $(\psi^\pm)^{-1} = \psi^\pm$.
 Then, the mesh on the domain $Q^\pm = \overline{Q^+\cup Q^-}$ 
 is given by
 \begin{equation*}
 \cG^{\ell, \pm} 
 = 
 \cG^\ell \cup \cG^{\ell, -}, \qquad \cG^{\ell, -} 
 = 
 \{\psi^\pm (K) : K\in \cG^\ell\},
 \end{equation*}
see Figure \ref{fig:patches}.
%%%%%%%%%%%%%%%%%%%%%%%%%%%%%%%%%%%%%%%%%%%%%%%%%%%%%%%%%%%%%%%%%%%%%%%%%%%%%
\begin{figure}
  \centering
  \includegraphics[width=.8\textwidth]{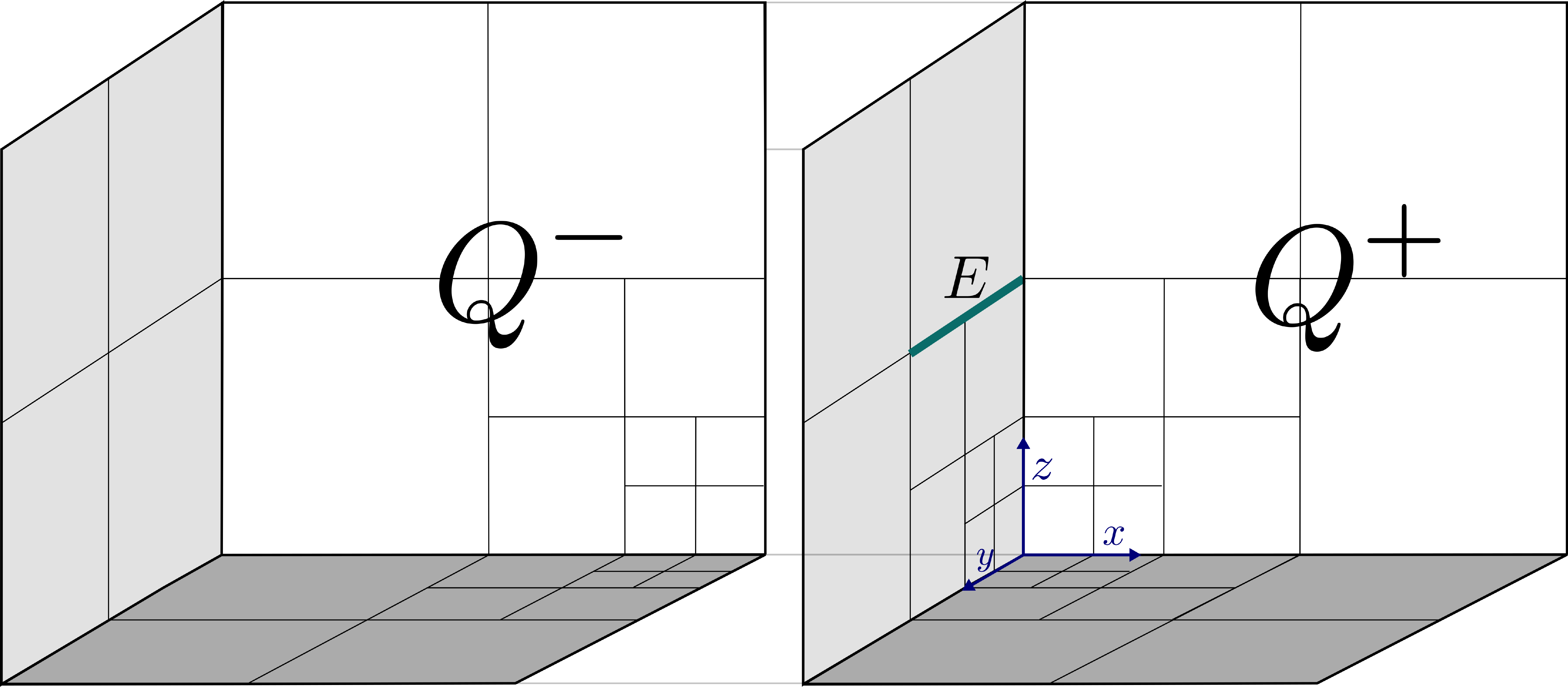}
  \caption{The mesh patches $\cG^{\ell, -}$ on $Q^-$ and $\cG^{\ell}$ on $Q^+$.
    An edge $E$ belonging to the interpatch interface is highlighted.}
  \label{fig:patches}
\end{figure}
%%%%%%%%%%%%%%%%%%%%%%%%%%%%%%%%%%%%%%%%%%%%%%%%%%%%%%%%%%%%%%%%%%%%%%%%%%%%%%
 The projection operator for functions 
 $v\in \cJ^\infty_\gamma(Q^\pm)$ can be easily constructed by reflection
 \begin{equation*}
   %\label{eq:Pihpellpm}
   (\Pihpellpm v)_{|_K} =
   \begin{cases}
     \PihpellK v & \text{if }K\in \cG^\ell\\
     \left(\PihpellK (v\circ\psi^\pm)\right)\circ \psi^\pm & \text{if }K\in \cG^{\ell,-}.
   \end{cases}
 \end{equation*}
 The operator thus obtained is continuous hence conforming, as discussed in the
 next lemma.
 \begin{lemma}
   \label{lemma:patch-pm}
   The operator $\Pihpellpm$ is conforming in $H^1(Q^\pm)$. Furthermore, 
if $\gamma \geq \gamma_0 > 3/2$ and $\gev \geq 1$, then for all $u\in
\cJ^{\varpi}_\gamma(Q^{\pm}; C, A, \gev)$ there exist 
$C_{\mathsf{hp}}^\pm$, $b_{\mathsf{hp}}$ such that, for all $\ell\in \mathbb{N}$,
with $p \geq c_0^\pm \ell^\gev$ for a sufficiently large
$c_0^\pm>0$ independent of $\ell$, 
there holds
\begin{equation}
  \label{eq:hp-error-pm}
  \| u - \Pihpellpm u \|_{H^1(Q^\pm)} 
  \leq 
  C_{\mathsf{hp}}^\pm e^{-b_{\mathsf{hp}}^\pm\ell}.
\end{equation}

Furthermore, there holds $\dim(\Xhpell)\simeq  \ell^{3\gev+1}$.
\end{lemma}
\begin{proof}
  $\Pihpellpm$ is continuous in the sub-patches $Q^+$ and $Q^-$. 
  It remains to check the continuity across the inter-patch interface $F^{\pm} = \{0\}\times
  (0,1)^2$. By construction, all elemental faces belonging to the interface are
  regular, hence, by Lemma \ref{lemma:reg-cont}, the discontinuous projector
  $\Pihpelld$ is continuous across these faces.

  We consider the error contribution from interior mesh layers, i.e., 
  from all elements in $\cL^\ell_j$, $j>0$.
  For all faces $F$ in interior mesh layers which are situated 
  perpendicular to $F^{\pm}$, 
  we have
  \begin{equation*}
      \jump{ \Pihpelld u + \fL^E (\Pihpelld u)}_F\cdot n  = 0.
  \end{equation*}
  We now consider any edge $E$ belonging to $F^{\pm}$ and separating the mesh
  levels $\cL^\ell_{k}$ and $\cL^\ell_{k+1}$, see Figure \ref{fig:patches} for
  an example. 
  By the continuity of the
  discontinuous projector across regular faces
  \begin{equation*}
    \jump{\Pihpelld u }_E = \jump{\Pihpelldm u }_E,
  \end{equation*}
  where $\Pihpelldm$ is the discontinuous projector in patch $\cG^{\ell, -}$.
  Therefore, from definitions \eqref{eq:lifting-edge}, \eqref{eq:lifting-face},
  and \eqref{eq:lifting-def1}, we conclude that the projection operator is
  continuous across interior mesh layers $\cL^\ell_k$, $k>0$.

  When dealing with the terminal layer $\cL^\ell_0$, we note that the
  discontinuous projector is constant hence continuous. The nodal
  liftings are continuous by the symmetry of their construction; 
  the edge liftings are then continuous by the same argument as in interior mesh layers, 
  and this gives the continuity between patches.

  Finally, equation \eqref{eq:hp-error-pm} follows 
  from the application of the corresponding
  approximation results in each patch.
 \end{proof}
 We can directly extend the construction in the proof of the above lemma to the 
 remaining patches $R^m =(0, a_1)\times (0, a_2)\times (0, a_3)$ 
 with $(a_1, a_2, a_3)\in \{-1, 1\}^3$, $m=0, \dots, 7$.
 Recall that $\Pihpellm$ is the conforming $hp$
 projector in patch $R^m$, obtained by reflection from the one defined in
 $(0,1)^3$, see \eqref{eq:Pihpellm}. 
 Recall also that the functions $\psi^m$ are
 the reflections from $(0,1)^3$ to $R^m$. 
 For $\gamma>3/2$, given the finite element space on $R = \bigcup_m R^m$,
 \begin{equation*}
   \XhpellR = \{v\in H^1(R): v\circ \psi^m \in \Xhpell,\, m=0, \dots,7 \},
 \end{equation*}
 we define the global projector 
 \begin{equation}
   \label{eq:PihpellR}
   \PihpellR : \cJ_\gamma^\infty(R) \to \XhpellR \quad \text{such that}\quad \PihpellR v_{|_{R^m}}=\Pihpellm v_{|_{R^m}}.
 \end{equation}
 Then, by the same arguments as in Lemma \ref{lemma:patch-pm} applied to all
 interpatch interface, there holds the
 following result.
 \begin{corollary}
   \label{cor:exponential}
   The operator $\PihpellR$ defined in \eqref{eq:PihpellR} is conforming in $H^1(R)$. Furthermore, 
if $\gamma \geq \gamma_0 > 3/2$ and $\gev \geq 1$, then for all $u\in
\cJ^{\varpi}_\gamma(R; C, A, \gev)$ 
exist constants $C^R_{\mathsf{hp}}$, $b^R_{\mathsf{hp}}$ (that depend on $C, A$, and $\gev$)
such that, for every $\ell \in \mathbb{N}$ there holds, 
with $p \geq c_0^R \ell^\gev$ for some $c_0^R>0$ independent of $\ell$,
the error bound
\begin{equation*}
  %\label{eq:hp-error-R}
  \| u - \PihpellR u \|_{H^1(Q^\pm)} 
  \leq 
  C_{\mathsf{hp}}^R e^{-b_{\mathsf{hp}}^R\ell}.
\end{equation*}
Furthermore, $\dim(\Xhpell)\simeq  \ell^{3\gev+1}$.

\end{corollary}
\section{Extension of rank bounds to domains with internal singularity}
\label{sec:ExtRkBdIntSing}
As a corollary to Theorem \ref{th:QTT-analytic}, we show here how the result can
be generalized to functions that have the singularity in an internal point of
the domain. As an example, we will consider the case of the axiparallel cube
$R = (-1,1)^3$ and of functions in the weighted analytic class 
$\cJ^\varpi_\gamma(R;C,A,\gev)$ with singularity at the origin.
The cube $R$ can be decomposed into eight congruent cubes,
all with the singularity situated at one corner, 
that we will denote by $R^m$, $m=0, \dots, 7$. 
For each $m$, there
exist $(a_1, a_2, a_3)\in \{-1, 1\}^3$ such that 
$R^m = (0, a_1), \times (0, a_2)\times (0, a_3)$.
We do not need to
specify any particular ordering, but choose, without loss of generality $R^0 = Q$.
We will denote $\psi^m : Q \to R^m$ the linear transformation from $Q = R^0$ to $R^m$
such that that for all $(x_1,x_2,x_3)\in Q$
\begin{equation*}
  %\label{eq:psim-def}
  \psi^m :
  \begin{pmatrix}
    x_1\\x_2\\x_3
  \end{pmatrix}
\mapsto
  \begin{pmatrix}
   a_1&& \\
   &a_2& \\
   &&a_3
  \end{pmatrix}
  \begin{pmatrix}
    x_1\\x_2\\x_3
  \end{pmatrix}
, \qquad \text{with }(a_1, a_2, a_3)\in \{-1,1\}^3,
\end{equation*}
i.e., $\psi^m$ only operates reflections with respect to interpatch interfaces.
Note that $\psi^0$ is the identity.

Furthermore, we
define by $\scrA^{\ell.m}$ the analysis operator (see Section
\ref{sec:AnSyntOp}) of patch $R^m$, such that
\begin{equation*}
\scrA^{\ell, m} v_{|_{R^m}} \in \mathbb{R}^{2^\ell\times 2^\ell\times 2^\ell} 
\quad \text{and}\quad
\scrA^{\ell, m}v = \scrA^{\ell} (v\circ \psi^m).
\end{equation*}
%%%%%%%%%%%%%%%%%%%%%%%%%%%%%%%%%%%%%%%%%%%%%%%%%%%%%%%%%%%%%%%%%%%%%%%%%%%%%%%%
\subsection{Quasi interpolation on $R$}
\label{sec:QuasIntR}
%%%%%%%%%%%%%%%%%%%%%%%%%%%%%%%%%%%%%%%%%%%%%%%%%%%%%%%%%%%%%%%%%%%%%%%%%%%%%%%%
We can then define the local $hp$ projection and interpolation operators in the patches $R^m$,
$m=0, \dots, 7$, in the same way, i.e., as
\begin{equation}
  \label{eq:Pihpellm}
 \Pihpellm v = \left(\Pihpell (v\circ \psi^m)\right)\circ (\psi^m)^{-1} \quad {\text{and}}\quad
\cI^{\ell,m }v = \left(\cI^{\ell} (v\circ \psi^m)\right)\circ (\psi^m)^{-1} 
\end{equation}
in each $R^m$. The definition of the local quasi interpolation operators also
follows directly, by setting $\fP^{\ell, m} = \cI^{\ell, m}\Pihpellm$, for $m = 0,
\dots, 7$. Then, the global (on $R$) quasi interpolation operator is the
operator $\fP^{\ell, R}$ such that $\fP^{\ell, R} _{|_{R^m}} = \fP^{\ell, m}$
for all $m =0, \dots, 7$.
%%%%%%%%%%%%%%%%%%%%%%%%%%%%%%%%%%%%%%%%%%%%%%%%%%%%%%%%%%%%%%%%%%%%%%%%%%%%%%%%
\subsection{Patchwise QTT formats}
\label{sec:PtchQTT}
%%%%%%%%%%%%%%%%%%%%%%%%%%%%%%%%%%%%%%%%%%%%%%%%%%%%%%%%%%%%%%%%%%%%%%%%%%%%%%%%
It is now straightforward to consider the ``patchwise QTT'' formats which are
constructed by adding a patch index to the formats considered so far. 
For a function $u\in \cJ^\infty_\gamma(R)$,
we consider the tensor 
$A \in \mathbb{R}^{8\times 2^\ell\times 2^\ell\times 2^\ell}$ 
such that for $m=0, \dots, 7$
\begin{equation*}
  A_{m,:,:,:} = \scrA^{\ell, m} \fP^{\ell,m}u.
\end{equation*}
Then, writing with the usual notation 
$i=\overline{i_1\dots i_\ell}$, $j=\overline{j_1\dots j_\ell}$ 
and $k=\overline{k_1\dots k_\ell}$,
\begin{itemize}
\item 
$A$ admits a \emph{patchwise, classic QTT} decomposition if 
  \begin{equation*}
    %\label{eq:patch-qtt}
   A_{m,i,j,k} = U^1_{m, :}(i_1) \cdots U^\ell(i_\ell)V^1(j_1) \cdots V^\ell(j_\ell)W^1(k_1) \cdots W^\ell(k_\ell)
  \end{equation*}
for all $m=0, \dots, 7$, $(i,j,k)\in \{0, \dots, 2^\ell-1\}^3$ and
where $U^1_{m, :}(i_1)$ indicates the $m$th row of $U^1(i_1)$
with cores defined as in \eqref{eq:qtt-cores-def} and the following convention on ranks
\begin{equation*}
    r_0 := 8 \quad  t_\ell := 1.
\end{equation*}
Note that the only modification with respect to Definition \ref{def:QTT} 
is the convention $r_0=8$.
\item 
$A$ admits a \emph{patchwise, transposed order QTT} decomposition if 
  \begin{equation}
    \label{eq:patch-qt3}
   A_{m,i,j,k} = U^1_{m, :}(\overline{i_1j_1k_1}) \cdots U^\ell(\overline{i_\ell j_\ell k_\ell })
  \end{equation}
  with cores as in Definition \ref{def:QT3} and 
  with the restriction on the ranks $r_0 = 8$, $r_\ell = 1$.
\item 
$A$ admits a \emph{patchwise, Tucker-QTT} decomposition if 
\begin{equation}
    \label{eq:patch-tqtt}
\begin{aligned}
  A_{m,i,j,k} 
   = \sum_{\beta_1, \beta_2, \beta_3 = 1}^{R_1, R_2, R_3}
  &G_{\beta_1, \beta_2, \beta_3}^m
 U^1_{{\beta_1}}(i_1) U^2(i_2) \ldots U^\ell(i_\ell) \\
  V^1_{{\beta_2}}(j_1) &V^2(j_2)\ldots V^\ell(j_\ell) 
 W^1_{{\beta_3}}(k_1) W^2(k_2)\ldots W^\ell(k_\ell).
\end{aligned}
\end{equation}

where, clearly, the Tucker core is now a 
four-dimensional array of size $8\times R_1\times R_2\times R_3$.
\end{itemize}
Let $\mathcal{T}^{\ell, R}$ be the tensor product mesh on $R$ given by
\begin{multline*}
\mathcal{T}^{\ell,R} 
= 
\{ (2^{-\ell}i, 2^{-\ell}(i+1))\times (2^{-\ell}j, 2^{-\ell}(j+1)) 
   \times (2^{-\ell}k, 2^{-\ell}(k+1)), \\(i,j,k) \in\{ -2^\ell+1,\dots, 2^\ell-1\}^3\}.
\end{multline*}
We define the the finite element space in $R$ as
\begin{equation*}
\XqttellR = \left\{v\in H^1_0(R): v_{|_K}\in \mathbb{Q}_1(K), \text{for all }K\in \mathcal{T}^{\ell, R} \right\}.
\end{equation*}
The following proposition is then a direct consequence of Theorem
\ref{th:QTT-analytic} and of Corollary \ref{cor:exponential}.
\begin{proposition}
  \label{prop:QTT-R}
Assume $\gamma>3/2$, $C_u>0$, $A_u>0$, $\gev\geq 1$, and $0< \epsilon_0 \ll 1$. 
  Furthermore, assume the function $u$ belongs to the weighted Gevrey class 
  $u\in \cJ^\varpi_\gamma(R; C_u, A_u, \gamma, \gev)\cap H^1_0(R)$.
  Then, for all $0< \epsilon\leq \epsilon_0$, 
  there exists 
  $\ell \in \mathbb{N}$ and $\vqtdell \in \XqttellR$ 
  such that
  \begin{equation*}
    \| u - \vqttell \|_{H^1(Q)} \leq \epsilon
  \end{equation*}
  and the multi-dimensional array 
  $V_{\mathsf{qtd}}^\ell\in \mathbb{R}^{8\times 2^\ell\times 2^\ell\times 2^\ell}$ 
  such that $(V_{{\mathsf{qtd}}}^\ell)_{m, :,:,:} = \scrA^{\ell,m}\vqttell$,
  $m=0,\dots, 7$ admits a patchwise representation with
  \begin{equation*}
    \Ndof \leq C |\log\epsilon|^{\kappa}
  \end{equation*}
  degrees of freedom, with a positive constant $C$ independent of $\epsilon$ and
  \begin{equation*}
\kappa =
\begin{cases}
  4\gev+3 & \text{for patchwise classic QTT}\;,\\
  6\gev+1 & \text{for patchwise transposed order QTT}\;,\\
  3\gev+3 &  \text{for patchwise Tucker-QTT} \;.
\end{cases}
  \end{equation*}
\end{proposition}
\begin{proof}
 Here, 
 we retrace the steps of the proofs of Lemmas \ref{lemma:qtt-rank},
 \ref{lemma:qt3-rank}, and \ref{lemma:tqtt-rank}, 
 generalizing them to the multipatch case.
 \paragraph{\textbf{Patchwise classic QTT}}
The tensor $\Vqttell$ can be written as the product 
\begin{equation*}
    (\Vqtdell)_{m, i, j, k} = U^0(m)U^1(i_1)\cdots U^\ell(i_\ell)V^1(j_1)\cdots V^\ell(j_\ell)W^1(k_1)\cdots W^\ell(k_\ell),
\end{equation*}
where the bounds on the ranks of the cores $U^1, \dots, U^\ell$ and the first rank of the core $V^1$ are multiplied
by $8$, while the other bounds are left unchanged with respect to the single
patch case. 
The multiplication of the cores $U^0$ and $U^1$ gives the multipatch formulation.
 \paragraph{\textbf{Patchwise transposed order QTT}}
The row space of the unfolding matrices
 \begin{equation*}
   V^{(q)}_{\overline{m \xi_1\eta_1\zeta_1}, \overline{\xi_2\eta_2\zeta_2}} 
    = 
   ( \Vqttell)_{m, \overline{\xi_1\xi_2}, \overline{\eta_1\eta_2}, \overline{\zeta_1\zeta_2}}  
 \end{equation*}
 defined for $m\in{0, \dots, 7}$, 
$\xi_1, \eta_1, \zeta_1 \in \{0, \dots, 2^q-1\}$, and $\xi_2,\eta_2,
  \zeta_2\in \{0, \dots, 2^{\ell-q}-1\}$
  is bounded asymptotically by the same quantity as the one of the unfolding matrix
  in \eqref{eq:qt3-unfolding}, by symmetry. 
  Thus, $\Vqtdell$ admits a decomposition such that
  \begin{equation*}
    (\Vqtdell)_{m, i, j, k} = U^0 (m) U^1(\overline{i_1j_1k_1})\cdots U^\ell(\overline{i_\ell j_\ell k_\ell}).
  \end{equation*}
  By multiplying $U_0 (m)$ and $ U_1(\overline{i_1j_1k_1})$ for all $m=0, \dots,
  7$ and $i_1,j_1, k_1\in \{0,1\}$, we obtain a representation of the form \eqref{eq:patch-qt3}.
  \paragraph{\textbf{Patchwise Tucker-QTT}}
 We Tucker-decompose the tensor $\Vqtdell$, thus obtaining, by the same
 arguments that we used for equation \eqref{eq:Tucker},
 \begin{equation*}
   \Vqtdell = \sum_{\beta_1 ,\beta_2, \beta_3 = 1}^{R_T} \sum_{\beta_0=1}^{R_P}
G_{\beta_0, \beta_1, \beta_2, \beta_3} Z_{\beta_0} \otimes U_{\beta_1} \otimes V_{\beta_2}\otimes W_{\beta_3},
 \end{equation*}
 where $R_T\leq C \ell^{\gev+1}$. Then, by contracting the core $G$ and the
 factor $Z$ over the index $\beta_0$ and by deriving the existence of the block QTT decomposition of
 the Tucker factors $U$, $V$, $W$ as in equation \eqref{eq:Tucker-cores}, 
 we obtain the existence of a representation of $\Vqtdell$ of the form \eqref{eq:patch-tqtt}.
 \myqed
\end{proof}
  \begin{remark}
    In Proposition \ref{prop:QTT-R}, we consider the approximation of functions
    in the cube $R=(-1,1)$ for ease of notation. Nonetheless, the argument and the result
    extend, without major modification, to 
    $\widetilde{R} = (-a_1, b_1) \times (-a_2, b_2) \times (-a_3, b_3)$,
    with $a_i, b_i>0$, $i=1,2 ,3$, and with a point singularity at the origin.
    This implies, by translation, that given a cube of fixed size, 
    we can obtain bounds on patchwise quantized tensor representations that are 
    \emph{uniform in the location of the singularity}.
  \end{remark}
%%%%%%%%%%%%%%%%%%%%%%%%%%%%%%%%%%%%%%%%%%%%%%%%%%%%%%%%%%%%%%%%%%%%%%%%%%%%%%%%%%%%%%%%%%%%%%
\section{QTT representation of prolongation matrices.} 
\label{sec:prolong}
%%%%%%%%%%%%%%%%%%%%%%%%%%%%%%%%%%%%%%%%%%%%%%%%%%%%%%%%%%%%%%%%%%%%%%%%%%%%%%%%%%%%%%%%%%%%
In order to evaluate the error $\varepsilon_\ell$,
we need a tensor of $\nabla\uqtd^{\ell,\delta}$ evaluated on
the virtual mesh with $L$ levels of refinement
and have it represented using the 
respective tensor decomposition without accessing all its elements.
This can be implemented as a multiplication by the prolongation matrices 
in the respective tensor format.
To introduce the prolongation matrices, we start by
considering the one dimensional piecewise linear space on the virtual mesh 
with 
$\ell$ levels (recall that $I^\ell_j = (2^{-\ell}j, 2^{-\ell}(j + 1))$)
\[
     \Xqttellonedzero = \{v\in H^1((0,1)): v(1) = 0 \text{ and } v_{|_{I^\ell_j}}\in \mathbb{P}_1(I^\ell_j),
    j=0, \dots, 2^\ell-1\}.
\]
Furthermore, we introduce
the one dimensional analysis operator
$\scrA^\ell_{\mathrm{1d}} : H^1((0,1))\to \mathbb{R}^{2^\ell}$ as 
\begin{equation*}
    \left(\scrA^\ell_{\mathrm{1d}} v \right)_i = v(2^{-\ell}i), \qquad i=0, \dots, 2^\ell-1.
\end{equation*}
Then, for every $L>\ell$, the one dimensional prolongation operator 
$P^{(\ell\to L)}_{\mathrm{p.l.}}  \in \mathbb{R}^{2^L\times 2^\ell}$ 
is realized by the matrix such that
\begin{equation}
  \label{eq:1dprolong-pl}
P^{(\ell\to L)}_{\mathrm{p.l.}} 
\left(\scrA^\ell_{\mathrm{1d}}\vqttell \right) 
   =  
     \scrA^L_{\mathrm{1d}}\vqttell  
\qquad \text{for all }\vqttell\in \Xqttellonedzero.
  \end{equation}
In the same vein, the one dimensional prolongation operator for piecewise
constant function is such that
\begin{equation*}
  P_\mathrm{p.c.}^{(\ell\to L)}\left(\scrA^\ell_{\mathrm{1d}}\vqttell    \right) =  \scrA^L_{\mathrm{1d}}\vqttell
\end{equation*}
for all
\begin{equation*}
\vqttell\in X^\ell_{\mathrm{p.c.}, \mathrm{1d}} 
= 
\{v\in L^\infty((0,1)): v_{|_{[x_j, x_{j+1})}}\in \mathbb{P}_0([x_j, x_{j+1})), j=0, \dots, 2^\ell-1\}.
\end{equation*}

Recall that we consider functions $u$ such that $u_{|_{\Gamma}}=0$, 
where $\Gamma = \partial Q \backslash\{x=(x_1,x_2,x_3)\in\partial Q: x_1 x_2 x_3 =0\}$.
In this case, 
the three-dimensional prolongation matrices from mesh level~$\ell$ to $L>\ell$,
can be written as a tensor product of the one-dimensional piecewise 
linear and piecewise constant prolongation matrices, which read:
\begin{equation}\label{eq:pl_prolong}
	 P_\mathrm{p.l.}^{(\ell\to L)} = 
	2^{\ell-L}
	\biggl(
	I^{(\ell)}
	\otimes 
	\begin{bmatrix}
		2^{L-\ell}  \\
		\vdots \\
		2 \\
		1 
	\end{bmatrix}
	+ 
	J^{(\ell)}
	\otimes 
	\begin{bmatrix}
		0  \\
		1 \\
		\vdots \\
		2^{L-\ell} -1
	\end{bmatrix}
	\biggr) 
	\in 
	\mathbb{R}^{2^L \times 2^\ell}
\end{equation}
and
\begin{equation*}
  %\label{eq:pc_prolong}
	P_\mathrm{p.c.}^{(\ell\to L)} = 
	I^{(\ell)}
	\otimes 
	\begin{bmatrix}
		1  \\
		\vdots \\
		1 
	\end{bmatrix}_{2^{L-\ell}}\in 
	\mathbb{R}^{2^L \times 2^\ell}
	,
\end{equation*}
respectively, 
where we used the notation
\[
I^{(\ell)} =
	\begin{bmatrix}
		1 \\
		& \ddots \\
		& & \ddots \\
		& & & 1
	\end{bmatrix}_{2^{\ell}\times 2^{\ell}}, \quad 
S^{(\ell)} =
	\begin{bmatrix}
		0 & 1 \\
		 & \ddots & \ddots\\
		& & \ddots & 1 \\
		& & & 0 & 
	\end{bmatrix}_{2^{\ell}\times 2^{\ell}}.
\]
The matrix $P_\mathrm{p.c.}^{(\ell\to L)}$ can be represented with QTT ranks
$1,1,\dots, 1$, as it has Kronecker product structure, since $I^{(\ell)} =
  \left( I^{(1)} \right)^{\otimes \ell}$ and 
\[
	P_\mathrm{p.c.}^{(\ell\to L)} = 
	I^{(\ell)} \otimes e^{\otimes (L-\ell)}, \quad 
	% I = \begin{bmatrix} 1 & 0 \\ 0 & 1 \end{bmatrix}, \quad
	e = \begin{bmatrix} 1  \\  1 \end{bmatrix}.
\]

We now show, in Proposition~\ref{prop:prolong_estimate} below, 
that $P_\mathrm{p.l.}^{(\ell\to L)}$ also has low-rank QTT structure.
For convenience, we introduce the matricization operator $\mathcal{M}: \mathbb{R}^{r_1\times m \times n \times r_2} \to \mathbb{R}^{m r_1\times n r_2}$ such that:
\begin{multline*}
	\left(\mathcal{M}\left( X \right)\right)_{\overline{\alpha_1 i}, \overline{\alpha_2 j}} = \left(X (i, j)\right)_{\alpha_1,\alpha_2}, \\ i = 1,\dots,m, \quad j = 1,\dots,n, \quad \alpha_i = 1,\dots, r_i, \ i =1,2,
\end{multline*}
that allows to recast tensor cores as matrices.
The following proposition holds.

\begin{proposition} \label{prop:prolong_estimate}
	The matrix $P_\mathrm{p.l.}^{(\ell\to L)}$, $L>\ell$, defined in~\eqref{eq:pl_prolong} has explicit QTT representation with ranks $2,2,\dots,2$.
	In particular, for each $i_1,\dots,i_L \in  \{0,1\}$, $j_1,\dots,j_\ell \in  \{0,1\}$ and $j_{\ell+1},\dots,j_L \in  \{0\}$:
	\[
		\left(P_\mathrm{p.l.}^{(\ell\to L)}\right)_{\overline{i_1\dots i_{L}}, \overline{j_1\dots j_{L}}}
		= 
		Q_1(i_1,j_1) \dots Q_L(i_L, j_L)
	\]
	where the matricizations read
	\[
	\begin{split}
		&\mathcal{M}(Q_1) = 
		\begin{bmatrix}
			I & J
		\end{bmatrix},		
		\\
		&\mathcal{M}(Q_i) = 
		\begin{bmatrix}
			I & J \\
			 & J^\top 
		\end{bmatrix}, 
		\ i = 2,\dots, \ell, 
		\\
		&\mathcal{M}(Q_{i}) = 
		\frac{1}{2}
		\begin{bmatrix}
			p & \delta_1 \\
			 \delta_2 & q
		\end{bmatrix},
	  	\ i = \ell +1,\dots, L-1,
	  	\\
	  	&\mathcal{M}(Q_{L}) = 
	  	\frac{1}{2}
		\begin{bmatrix}
			p \\ \delta_2
		\end{bmatrix},
	\end{split}
	\]
	with blocks given by
	\[
	I = \begin{bmatrix} 1 & 0 \\ 0 & 1 \end{bmatrix}, \ \ 
	J = \begin{bmatrix} 0 & 1 \\ 0 & 0 \end{bmatrix}, \ \ 
	p = \begin{bmatrix} 2 \\ 1 \end{bmatrix}, \ \ 
	q = \begin{bmatrix} 1 \\ 2 \end{bmatrix}, \ \ 
	\delta_1 = \begin{bmatrix} 1 \\ 0 \end{bmatrix}, \ \ 
	\delta_2 = \begin{bmatrix} 0 \\ 1 \end{bmatrix}.
	\]
\end{proposition}
\begin{proof}
	First, we introduce the notation
	\[
	p^{(i)}
	=
		2^{-i}
	\begin{bmatrix}
		2^{i}  \\
		\vdots \\
		2 \\
		1 
	\end{bmatrix},
	\quad 
	q^{(i)}
	=
	2^{-i}
	\begin{bmatrix}
		0  \\
		1 \\
		\vdots \\
		2^{i} -1
	\end{bmatrix},
	\]
	so that
	\[
		P_\mathrm{p.l.}^{(\ell\to L)} = 
	I^{(\ell)}
	\otimes 
	p^{(L-\ell)}
	+ 
	J^{(\ell)}
	\otimes 
	q^{(L-\ell)}.
	\]
	Since $I^{(\ell)} = I \otimes I^{(\ell-1)}$ and $J^{(\ell)} = I \otimes J^{(\ell-1)} + J \otimes \left( J^\top \right)^{\otimes (\ell -1)}$ and using the operation $\Join$ that denotes the strong Kronecker product between block matrices, in which matrix-matrix multiplication of blocks is replaced by a Kronecker product\footnote{Formally, the strong Kronecker product of two $2\times 2$ block matrices is defined as the following $2\times 2$ block matrix:
	\[
		\begin{bmatrix}
			A_{11} & A_{12} \\
			A_{21} & A_{22}
		\end{bmatrix}
		\Join
		\begin{bmatrix}
			B_{11} & B_{12} \\
			B_{21} & B_{22}
		\end{bmatrix}
		=
		\begin{bmatrix}
			A_{11}\otimes B_{11} + A_{12}\otimes B_{21} & 
			A_{11}\otimes B_{12} + A_{12}\otimes B_{22} \\
			A_{21}\otimes B_{11} + A_{22}\otimes B_{21} & 
			A_{21}\otimes B_{12} + A_{22}\otimes B_{22} \\
		\end{bmatrix}.
	\]
	}
	, we obtain
	\[
	\begin{split}
		P_\mathrm{p.l.}^{(\ell\to L)} = 
		&\begin{bmatrix}
			I & J 
		\end{bmatrix}
		\Join
		\begin{bmatrix}
			I^{(\ell-1)} & J^{(\ell-1)} \\
			& \left( J^\top \right)^{\otimes (\ell -1)}
		\end{bmatrix}
		\Join
		\begin{bmatrix}
			p^{(L-\ell)} \\ q^{(L-\ell)}
		\end{bmatrix}
		= \\
		&\begin{bmatrix}
			I & J 
		\end{bmatrix}
		\Join
		\begin{bmatrix}
			I & J \\
			& J^\top 
		\end{bmatrix}^{\Join (\ell-1)}
		\Join
		\begin{bmatrix}
			p^{(L-\ell)} \\ q^{(L-\ell)}
		\end{bmatrix}.
	\end{split}
	\]
	We complete the proof by the observations that
	\[
		\begin{bmatrix}
			p^{(i)} \\ q^{(i)}
		\end{bmatrix}
		=
		\frac{1}{2}
		\begin{bmatrix}
			p & \delta_1 \\
			 \delta_2 & q
		\end{bmatrix}
		\Join
		\begin{bmatrix}
			p^{(i-1)} \\ q^{(i-1)}
		\end{bmatrix}, 
		\,  i > 1,
		\quad
		p^{(1)} = \frac{1}{2} \begin{bmatrix} 2 \\ 1 \end{bmatrix},
		\quad
		q^{(1)} = \frac{1}{2} \begin{bmatrix} 0 \\ 1 \end{bmatrix}.
	\]
\end{proof}
\begin{corollary}
Let $\vqttell \in \Xqttellonedzero$ and let 
$\scrA^\ell_{\mathrm{1d}} \vqttell$ have QTT ranks $r_1,r_2,\dots,r_{\ell-1}$. 
Then, for any $L>\ell$, 
the vector $\scrA^L_{\mathrm{1d}}\vqttell = \{\vqttell(x_i)\}_{i=0}^{2^L -1}$, $x_i = 2^{-L}\,i$ 
can be represented with QTT ranks equal to $2r_1,2r_2,\dots,2r_{\ell-1}$.
\end{corollary}
\begin{proof}
According to Proposition~\ref{prop:prolong_estimate}, 
the matrix $P_\mathrm{p.l.}^{(\ell\to L)}$ has ranks $2,2,\dots,2$.
The statement then follows from the fact that the multiplication in \eqref{eq:1dprolong-pl} 
of a TT-matrix with ranks $R_1,\dots,R_{\ell-1}$ by a TT-vector 
with the ranks $r_1,\dots,r_{\ell-1}$, leads to the TT representation with ranks
$R_1r_1,\dots,R_{\ell-1}r_{\ell-1}$, see \cite{IO2011a}. 
\end{proof}

The multidimensional prolongation matrices are assembled as 
Kronecker products of the one-dimensional matrices 
$P_\mathrm{p.l.}^{(\ell\to L)}$ and/or $P_\mathrm{p.c.}^{(\ell\to L)}$.
For example, 
to find the values of $\vqttell\in \Xqttell$ on a mesh with $L$ levels,
the matrix 
\[
	P_\mathrm{p.l.}^{(\ell\to L)} \otimes P_\mathrm{p.l.}^{(\ell\to L)} \otimes P_\mathrm{p.l.}^{(\ell\to L)}
\] 
represented in the respective format 
is applied to the coefficient vector $\scrA^\ell \vqttell$.
\bibliographystyle{siam}
\bibliography{library}
\end{document}

%% file: preamble.tex
\usepackage[dvipsnames]{xcolor}
\usepackage[utf8]{inputenc}
\usepackage[T1]{fontenc}
\usepackage{amsmath, amsfonts,amssymb}
\usepackage{amsthm}
\usepackage[foot]{amsaddr}

\usepackage[hidelinks]{hyperref}
\usepackage{graphicx}

\usepackage{stmaryrd}
\usepackage[position=t]{subcaption}

\usepackage[color=black!10!white]{todonotes}
\usepackage[showonlyrefs]{mathtools}
\usepackage{mathtools}
\usepackage{bm}
\usepackage{mathrsfs}

\usepackage{algorithm}
\usepackage[noend]{algpseudocode}

\usepackage{tgpagella}
\usepackage{microtype}
\usepackage{enumitem}
\newcommand{\jump}[1]{\llbracket #1 \rrbracket}

\newcommand{\fullnorm}[2][]{\left\vert\kern-0.25ex\left\vert\kern-0.25ex\left\vert #2 
      \right\vert\kern-0.25ex\right\vert\kern-0.25ex\right\vert_{\mathrm{DG} #1}}
\newcommand{\fullnormfixedsize}[2][]{\vert\kern-0.25ex\vert\kern-0.25ex\vert #2 
      \vert\kern-0.25ex\vert\kern-0.25ex\vert_{\mathrm{DG} #1}}

\newtheorem{theorem}{Theorem}

\newtheorem{lemma}{Lemma}
\newtheorem{remark}{Remark}

\newtheorem{ind-assumption}{Induction Assumption}

\newtheorem{proposition}[theorem]{Proposition}
\newtheorem{corollary}[lemma]{Corollary}

\theoremstyle{definition}
\newtheorem{definition}{Definition}
\DeclareMathOperator{\spn}{span}
\DeclareMathOperator{\rank}{rank}

\DeclareMathOperator{\diag}{diag}

\DeclareMathOperator{\dist}{dist}

\DeclareMathOperator{\Id}{Id}

\let\epsilon\varepsilon

\let\phi\varphi

\let\theta\vartheta

\newcommand{\spacenameinnorm}{\mathcal{D}}
\newcommand{\spacenameinnormhom}{\mathcal{G}}

\newcommand{\wnormdg}[5]{\left\vert\kern-0.25ex\left\vert\kern-0.25ex\left\vert #1 
      \right\vert\kern-0.25ex\right\vert\kern-0.25ex\right\vert_{\spacenameinnorm^{#3}_{#4}(#5)} }
\newcommand{\wnormdghom}[5]{\left\vert\kern-0.25ex\left\vert\kern-0.25ex\left\vert #1 
      \right\vert\kern-0.25ex\right\vert\kern-0.25ex\right\vert_{\spacenameinnormhom^{#3}_{#4}(#5)} }

\newcommand{\hK}{{\widehat{K}}}
\newcommand{\Hmix}{H^2_{\mathrm{mix}}}
\newcommand{\hPi}{{\widehat{\Pi}}}
\newcommand{\hpi}{{\widehat{\pi}}}

\newcommand{\fc}{{\mathfrak{c}}}

\newcommand{\tq}{{\tilde{q}}}

\newcommand{\dalpha}{\partial^\alpha}

\newcommand{\alpham}{{|\alpha|}}

\newcommand{\fL}{{\mathfrak{L}}}

\newcommand{\fP}{{\mathfrak{P}}}

\newcommand{\cK}{\mathcal{K}}
\newcommand{\cN}{\mathcal{N}}
\newcommand{\cL}{\mathcal{L}}
\newcommand{\cG}{\mathcal{G}}
\newcommand{\cI}{\mathcal{I}}
\newcommand{\cT}{\mathcal{T}}
\newcommand{\fT}{\mathfrak{T}}
\newcommand{\scrA}{\mathscr{A}}
\newcommand{\scrS}{\mathscr{S}}

\newcommand{\bv}{\bm{v}}
\newcommand{\bx}{\bm{x}}

\newcommand{\cJ}{\mathcal{J}}

\newcommand{\Xhpell}{X_\mathsf{hp}^{\ell, p}}
\newcommand{\XhpellR}{X_\mathsf{hp}^{\ell, p}(R)}
\newcommand{\Xhpelld}{X_{\mathsf{hp}, \mathrm{disc}}^{\ell, p}}
\newcommand{\Xhpelloned}{X^{\ell,p}_{\mathsf{hp},\mathrm{1d}}}

\newcommand{\Pihpell}{\Pi_\mathsf{hp}^{\ell, p}}
\newcommand{\PihpellR}{\Pi_{\mathsf{hp}}^{\ell, p, R}}
\newcommand{\PihpellK}{\Pi^{\ell, p}_{\mathsf{hp}_{|_K}}}

\newcommand{\Pihpellm}{\Pi_\mathsf{hp}^{\ell, p, m}}
\newcommand{\Pihpellpm}{\Pi_\mathsf{hp}^{\ell, p, \pm}}
\newcommand{\Pihpelld}{\Pi_{\mathsf{hp}, \mathrm{disc}}^{\ell, p}}
\newcommand{\Pihpelldm}{\Pi_{\mathsf{hp}, \mathrm{disc}}^{\ell, p, -}}

\newcommand{\Ndof}{N_{\mathrm{dof}}}
\newcommand{\pmax}{p}

\newcommand{\qtt}{{\mathsf{qtt}}}
\newcommand{\qttt}{{\mathsf{qt3}}}
\newcommand{\tqtt}{{\mathsf{tqtt}}}

\newcommand{\Xqttell}{X^\ell}

\newcommand{\Xqttellonedzero}{X_{\mathrm{1d}, 0}^\ell}

\newcommand{\XqttellR}{X^{\ell,R}}

\newcommand{\uqtd}{u_\mathsf{qtd}}

\newcommand{\vqttell}{v_\mathsf{qtt}^\ell}
\newcommand{\vqtttell}{{v_\mathsf{qt3}^\ell}}
\newcommand{\vtqttell}{{v_\mathsf{tqtt}^\ell}}
\newcommand{\vqtdell}{{v_\mathsf{qtd}^\ell}}
\newcommand{\Vqttell}{V_\mathsf{qtt}^\ell}
\newcommand{\Vqtdell}{V_\mathsf{qtd}^\ell}

\newcommand{\hy}{\hat{y}}
\newcommand{\hz}{\hat{z}}

\graphicspath{{./figures/}}

\renewcommand{\bx}{x}

\newcommand{\pc}{\mathbf{a}}

\newcommand{\gev}{{\mathfrak{d}}}
\newcommand{\myqed}{}

%% file: num_exp.tex
\section{Numerical experiments}
\label{sec:NumExp}
In this section, we support the obtained theoretical results with numerical experiments.
First, 
in Section~\ref{sec:num:fe}, 
we construct FE approximants to functions defined 
in $Q=(0,1)^3$ with point singularities in three quantized tensor formats: 
QTT, transposed order QTT (QT3)  and Tucker QTT (TQTT), 
see Figure~\ref{fig:networks} for their tensor network representations. 
We note that for all formats under consideration, the 
numerically observed asymptotic behavior of rank versus error 
is better than that of theoretical estimates.

In Section~\ref{sec:num:evp}, we consider an elliptic eigenvalue problem 
with a singular potential -- the Schr\"odinger equation for a hydrogen atom. 
We approximate the solution using the finite element method with a 
tensor of coefficients represented the TQTT format.
The numerical results suggest that convergence rates of QTT 
formatted approximations are slightly higher than 
those achieved by the $hp$-FEM.
%
%%%%%%%%%%%%%%%%%%%%%%%%%%%%%%%%%%%%%%%%%%%%%%%%%%%%%%%%%%%%%%%%%%%%%%%%%%%%
\subsection{QTT-FE approximation of functions with point singularities}
\label{sec:num:fe}
%%%%%%%%%%%%%%%%%%%%%%%%%%%%%%%%%%%%%%%%%%%%%%%%%%%%%%%%%%%%%%%%%%%%%%%%%%%%
\begin{figure}
    \centering
    \includegraphics[width=0.55\textwidth]{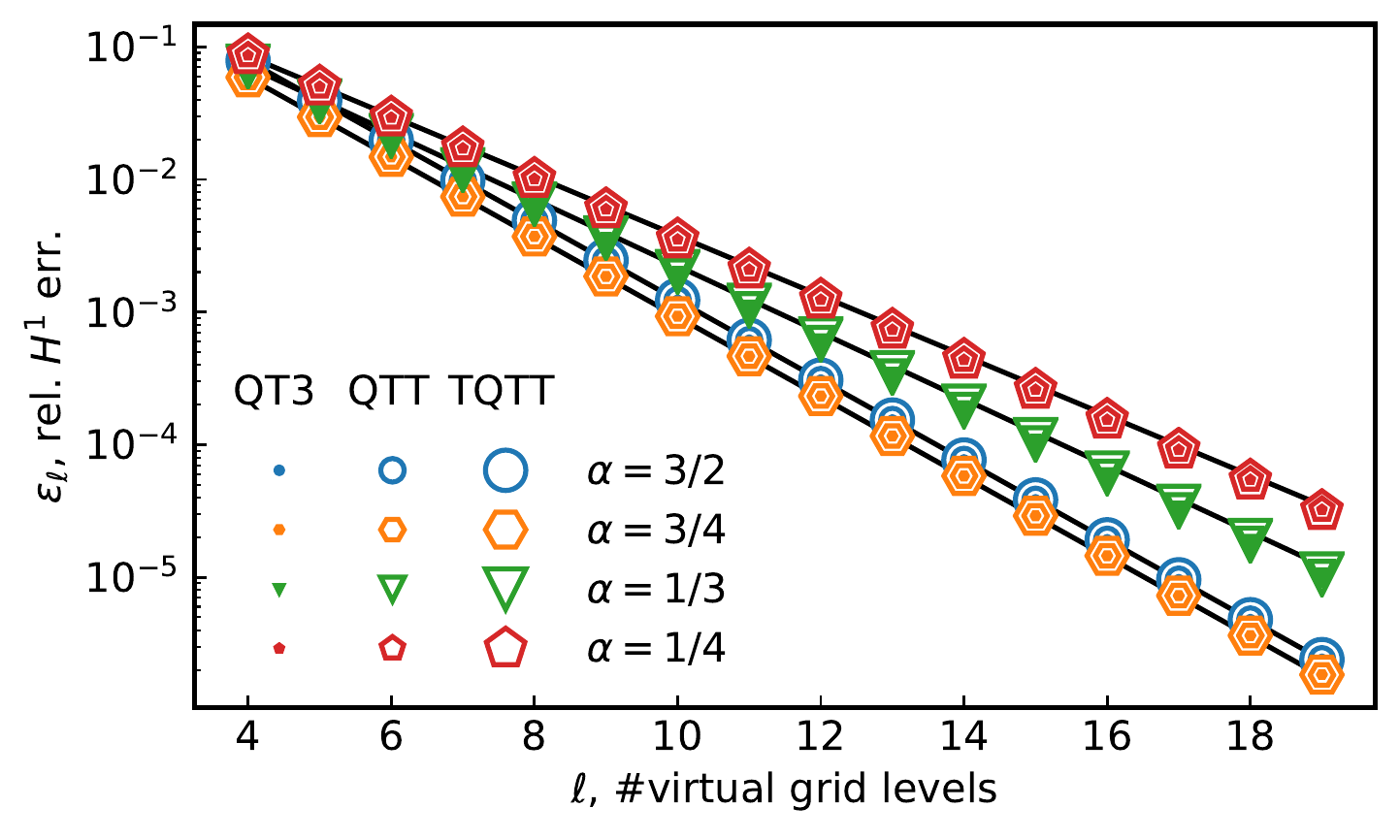}
    \caption{Number of virtual mesh levels versus the relative error in $|\cdot|_{H^1(D)}$ 
      seminorm for singularity exponent $\alpha = 3/2, 3/4, 1/3, 1/4$ and 
      for different quantized tensor formats: 
      QTT, QT3 (transposed QTT), TQTT (Tucker QTT). 
      The black lines correspond to 
      $\varepsilon_\ell =\mathcal{O}(2^{-\min\{\alpha + 1/2,\, 1\} \ell })$ convergence.}
    \label{fig:conv3d}
  \end{figure}

In this Section, we present the numerical results on function approximation. 
We will detail the approximation technique in
Remark~\ref{rmk:ExpSums}, while the details on the explicit construction of
prolongation matrices for the computation of the error will be postponed 
to Section~\ref{sec:prolong} in the appendix.

Let us consider the following smooth functions in $Q=(0,1)^3$ 
that exhibit singularities at the origin $x = (0, 0, 0)$:
\[
u(x) = |x|^\alpha\, m(x), \quad x = (x_1, x_2, x_3)\in Q,
\]
where $\alpha>0$ defines the strength of the singularity 
and $m(x) = (1-x_1^2)(1-x_2^2)(1-x_3^2)$ is chosen to ensure zero values of the function on $\Gamma$.
Note that the function $m(x)$ does not affect the singularity at 
the origin and can be represented with tensor ranks bounded from above by $3$ 
for QTT and TQTT formats and by $9$ for QT3 format.

Recall that by $\cI^\ell$ we denote the Lagrange interpolation operator on the uniform tensor mesh 
$\cT^\ell$:
\begin{equation*}
%\label{eq:Idef}
\mathcal{I}^\ell v 
= 
\sum_{(i,j,k)\in \{0, \dots, 2^\ell-1\}^3}v(\bx_{i,j,k})\,\phi_{i,j,k}.
\end{equation*}
%Let us introduce the $L^2$ projection operator 
%$\Pi^\ell: L^2(Q) \to \Xqttell$ such that for $v\in L^2(Q)$
%\[
%\Pi^\ell (v) 
%= \mr{2^{-3\ell}} \sum_{i=0}^{2^{\ell}-1} \sum_{j=0}^{2^{\ell}-1} 
%  \sum_{k=0}^{2^{\ell}-1} (v, \phi_{i,j,k}^{\ell})_{L^2(Q)} \phi_{i,j,k}^{\ell}.
%\]
In practice, $\uqtd^\ell$ will be an approximation of $\mathcal{I}^\ell u$ 
obtained by applying to $\scrA^\ell \mathcal{I}^\ell u$ 
the exponential sums representation (see Remark~\ref{rmk:ExpSums}) and by
interpolating on a staggered grid (see Remark~\ref{rmk:staggered}).
We introduce the rank-truncated representation of $\uqtd^\ell$, $\mathsf{qtd}\in \{ \mathsf{qtt}, \mathsf{tqtt}, \mathsf{qt3}\}$
based on the rounding procedure:
\[
\uqtd^{\ell,\delta} 
= 
\scrS^\ell\left(\texttt{round}_{\mathsf{qtd}}(\scrA^\ell \uqtd^\ell,\,\delta)\right),
\]
where $\texttt{round}_{\mathsf{qtd}}$ is a rounding operation that aims at reducing the numerical $\mathsf{qtd}$-rank 
of $\scrA^\ell \uqtd^\ell$ with the relative Euclidean error threshold $\delta$.
The rounding procedure is based on a sequence of QR and SVD decompositions, see~\cite[Alg. 2]{IO2011a} 
for TT (covers QTT and QT3 cases) and~\cite[Alg. 1]{DKQTT13} for two-level QTT-Tucker (covers the TQTT case with minor modifications).

For given $\uqtd^{\ell,\delta}$, 
we approximate the error $\widehat \varepsilon_\ell$ in the seminorm $|\cdot|_{H^1(Q)}$:
\[
\widehat \varepsilon_\ell = \frac{| \uqtd^{\ell,\delta} -
  u|_{H^1(Q)}}{|u|_{H^1(Q)}}= \frac{\| \nabla\uqtd^{\ell,\delta} - \nabla u
    \|_{L^2(Q)}}{\|\nabla u\|_{L^2(Q)}},
\]
by using the respective quantized tensor approximation of $u$ 
obtained on an equispaced mesh of axiparallel cubes 
with $L := 30$ levels of binary (virtual) refinement of $Q=(0,1)^3$:
\begin{equation}\label{eq:epsilon_def}
\widehat \varepsilon_\ell \approx \varepsilon_\ell \equiv 
\frac{\| \nabla\uqtd^{\ell,\delta} - \mathcal{I}^L_0(\nabla u)\|_{L^2(Q)}}{\|\mathcal{I}^L_0(\nabla u)\|_{L^2(Q)}}.
%\frac{\sqrt{\sum_{|\alpha|=1}\|\partial^\alpha \uqtd^{\ell,\delta} - \Pi^L_0(\partial^\alpha u)\|_{L^2(Q)}^2}}
%     {\sqrt{\sum_{|\alpha|=1}\|\Pi^L_0(\partial^\alpha u)\|_{L^2(Q)}^2}}.
\end{equation}
Here $\mathcal{I}^L_0(\nabla u) = (\mathcal{I}^L_{0,1}(\partial_{x_1} u), \mathcal{I}^L_{0,2}(\partial_{x_2} u), \mathcal{I}^L_{0,3}(\partial_{x_3} u))^\top$ with $\mathcal{I}^L_{0,\beta}$, $\beta=1,2,3$ being interpolation operators on a span of $\{\partial_{x_\beta}\phi_{i,j,k}\}_{i,j,k}$. 
%is an $L^2$ projection operator on the space of piecewise constant basis functions $\chi_{i,j,k}$, $(i,j,k)\in \{0,\dots,2^{L -1}\}^3$ which are associated with the mesh with $L$ virtual refinements of $Q=(0,1)^3$.

In Figure~\ref{fig:conv3d}, 
we present the convergence plots of the relative error 
$\varepsilon_\ell$ defined in~\eqref{eq:epsilon_def} for $\delta = 10^{-10}$ 
versus the virtual mesh levels~$\ell$ for different $\alpha$ and 
for different quantized tensor formats.
In all the cases, 
we observe empirical convergence in close correspondence with the rate
\[
\varepsilon_\ell 
= 
\mathcal{O}\left(2^{-\min\{\alpha + 1/2,\, 1\} \ell }\right), 
\quad \alpha > - \frac{1}{2}.
\]
This can be anticipated from classical FE interpolation error
bounds on an equispaced, cartesian mesh in $Q$, for 
functions
$[x\mapsto |x|^\alpha]$ in three spatial dimensions.

Let us first fix $\alpha = 3/2$.
In Figure~\ref{fig:fixedalpha}, 
we present $\varepsilon_\ell$ 
versus the effective number of degrees of freedom 
$N_\mathrm{dof}$ for three different tensor formats. 
On each gray dotted line we plot the error $\epsilon_\ell$ for 
one fixed~$\ell$ and for various values of~$\delta$.
The envelopes of the computed errors with respect to $\Ndof$ 
are highlighted with large empty markers.

\begin{figure}
  \centering
  \includegraphics[width=1.0\textwidth]{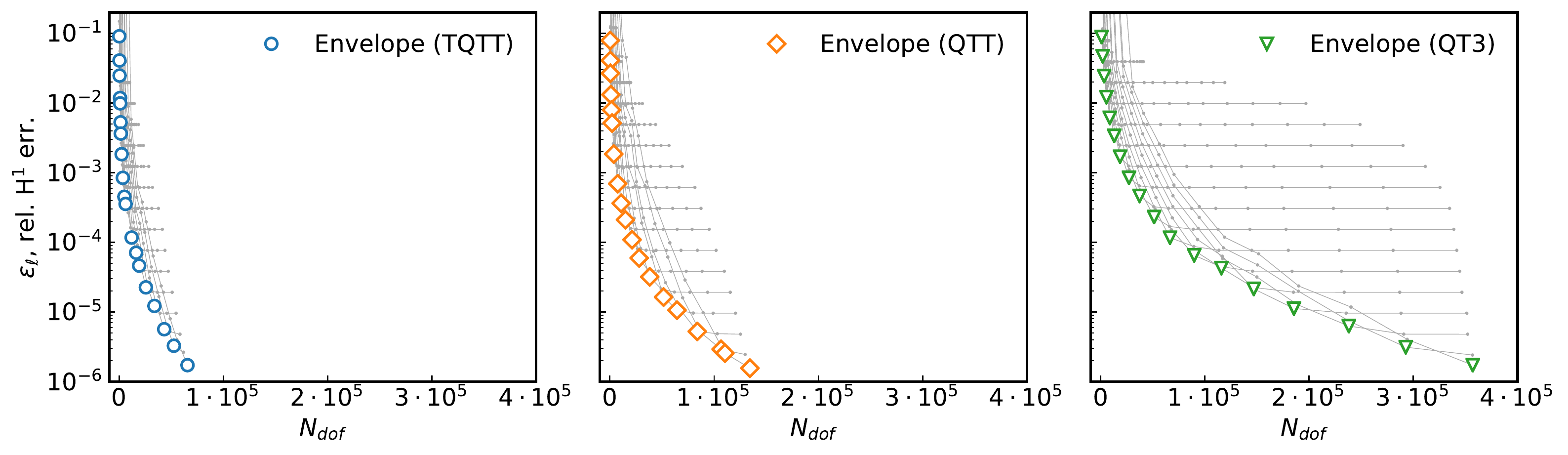}
  \caption{Each gray dotted line represents dependence of the estimated 
    relative seminorm $|\cdot|_{H^1(D)}$ error values $\varepsilon_\ell$ 
    on the rounding parameter $\delta$ for fixed $\ell$ as 
    suggested by~\eqref{eq:epsilon_def}. 
   Empty markers represent convex envelope of the gray dotted lines. 
   The limits for both axes coincide for each of the plots.}
    \label{fig:fixedalpha}
\end{figure}

In Figure~\ref{fig:conv3d_dof}, 
we depict $\varepsilon_\ell$ versus $N_\mathrm{dof}$ for $\alpha= 3/2,3/4,1/3$, and $1/4$ 
obtained as envelopes of the set of points obtained for different $\delta$ 
(see Figure~\ref{fig:fixedalpha} for $\alpha=3/2$).
By plotting $\log_{10}\log_2 \varepsilon_\ell^{-1}$ against $\log_{10} N_{\mathrm{dof}}$, 
we numerically estimate the constant $\kappa$ 
in the empirical exponential rate of convergence
\begin{equation}\label{eq:expconvemp}
	\varepsilon_\ell = C\,\mathrm{exp}(-b N_\mathrm{dof}^{1/\kappa}),
\end{equation}
for some positive constants $b$ and $C$.
Indeed, by first applying $\log_2$ to both sides of \eqref{eq:expconvemp}, 
we arrive at 
$\log_2 \varepsilon_{\ell}^{-1} = \tilde b N_\mathrm{dof}^{1/\kappa} - \log_2 C$, 
$\tilde b = -b \log 2$. 
Assuming $\log_2 C$ is small compared with $N_\mathrm{dof}^{1/\kappa}$ 
and taking $\log_{10}$ of both sides, we obtain 
\[
\log_{10}\log_2 \varepsilon_\ell^{-1} \approx \kappa^{-1} \log_{10} N_\mathrm{dof} + \log_{10} \tilde b. 
\]
In all the numerical examples considered, we observe $\kappa < 6$, i.e., 
higher convergence rates than those predicted by our 
quantized tensor rank bounds. 
We also observe
lower rates of convergence than those of $hp$-FE approximations of corner singularities 
in three spatial dimensions (see \eqref{eq:hp-errorN}), i.e., 
we find $\kappa > 4$.

Figures~\ref{fig:fixedalpha} and~\ref{fig:conv3d_dof} illustrate the fact that
in the range of $\ell$ considered, the transposed order QTT representation
requires more degrees of freedom to achieve a given accuracy $\epsilon$
than the two other formats even though it has empirical convergence
  \eqref{eq:expconvemp} with slightly smaller values of $\kappa$.
Among the tensor formats 
and for the examples considered, the TQTT format 
requires the smallest number of degrees of freedom
to achieve a prescribed accuracy $\epsilon$.
\begin{figure}
    \centering
    \includegraphics[width=1.0\textwidth]{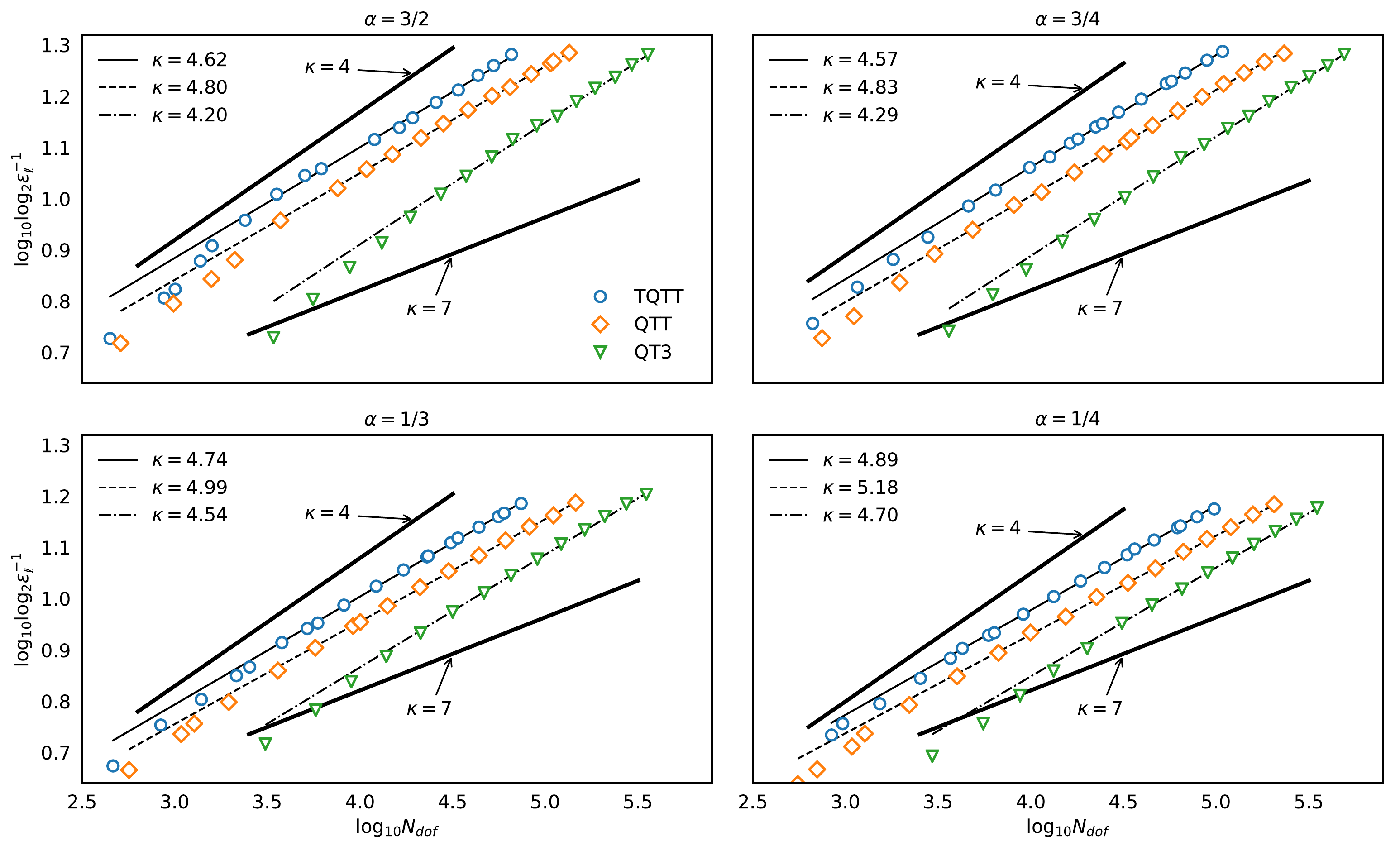}
    \caption{Effective number of degrees of freedom w.r.t. the estimated 
      relative seminorm $|\cdot|_{H^1(D)}$ error values $\varepsilon_\ell$ 
      for $\alpha = 3/2, 3/4, 1/3$ and $1/4$. 
      Reference lines with $\kappa=4$ and $\kappa=7$ correspond to the $hp$ approximation 
      and the obtained (for QTT and QT3) theoretical convergence bounds respectively.}
    \label{fig:conv3d_dof}
  \end{figure}
\begin{remark}[Approximation of singular functions by exponential sums]
 \label{rmk:ExpSums} 
To numerically evaluate 
the relative errors ${\epsilon}_\ell$ for all functions
under consideration we used the following procedure.
For each virtual mesh level $\ell$, 
we approximated the function using the exponential sums representation.
Specifically, 
we obtain the quantized tensor representations by applying the quadrature rule 
on a uniform mesh to the following integral~\cite{HK06,BM2010EXP}
\begin{equation}\label{eq:exp_int}
(\sqrt{y})^{-\beta} 
= 
\frac{1}{\Gamma (\beta/2)} 
\int_{-\infty}^{\infty} 
e^{-y e^t + \beta t/2}\, dt, \quad \beta > 0, \quad y > 0,
\end{equation}
for different values of $\beta$.
A quadrature rule on a uniform mesh applied to~\eqref{eq:exp_int} 
leads to an approximate, separated representation:
\begin{equation}\label{eq:expsums}
|x|^{-\beta} \approx \sum_{\alpha} \omega_\alpha\, 
e^{-x_1^2 e^{t_\alpha}} e^{-x_2^2 e^{t_\alpha}} e^{-x_3^2 e^{t_\alpha}}, \quad x = (x_1,x_2,x_3)
\end{equation}
where $|x| = (x_1^2+x_2^2+x_3^2)^{1/2}$.
The size of the integration interval and the number of points 
was tuned separately for each beta to ensure the desired accuracy.
In this way, an approximation for $|x|^\beta$, with $\beta \in (0,2)$, 
is found by first approximating the radial function
$|x|^{\beta - 2}$ since $\beta - 2 < 0$ and~\eqref{eq:expsums} is applicable, 
and subsequently by multiplying this function
by $|x|^2 = x_1^2 + x_2^2 + x_3^2$,
which has bounded TT ranks:
$ |x|^{\beta} = |x|^2\, |x|^{\beta - 2}$ for $\beta \in (0,2)$.
This allows us to avoid using cross approximation techniques 
which may experience stability issues at high accuracies 
(using exponential sums, we obtain approximations with relative accuracy $10^{-11}$ in $L^2$ norm).
Note that the exponential sums approach can be applied to any of the considered TT 
formats: QTT, QT3 and TQTT.
In this section, for the QTT-, QT3-formatted arrays and for
intermediate computations in TQTT we utilized the 
\texttt{ttpy} library\footnote{\url{https://github.com/oseledets/ttpy/tree/develop/tt} 
in the develop branch (latest commit: \texttt{ac03657})}.
\end{remark}
\begin{remark}[Interpolation on staggered grid]
  \label{rmk:staggered}
To conveniently assemble $|x|^{\beta - 2}$ for $\beta \in (0,2)$ using
exponential sums, while avoiding evaluation at the origin where the
function has a singularity, we approximate each $u(x_{i,j,k})$ as an average of
the neighboring points on a staggered grid. Let $h_\ell = 2^{-\ell}$ and denote by
$\widetilde{\mathcal{N}} = \{x_{i,j,k} + h_\ell/2\}_{i,j,k=0}^{2^\ell-1}$ the
nodes on the staggered grid. Then, for each $x_{i,j,k}$, the set of neighboring
points to $x_{i,j,k}$ on
the staggered grid is $\widetilde{\mathcal{N}}_{i,j,k} = \{x\in
\widetilde{\mathcal{N}}: |x-x_{i,j,k}|\leq h_\ell/2\}$. We then approximate
\[
	u(x_{i,j,k}) \approx \frac{1}{\# \widetilde{\mathcal{N}}_{i,j,k}}
	\sum_{x\in \widetilde{\mathcal{N}}_{i,j,k}}u(x),
\]
where $\# \widetilde{\mathcal{N}}_{i,j,k}$ is the number of points of
$\widetilde{\mathcal{N}}_{i,j,k}$---equal to $8$ except for points that lie on
$\partial \Omega \setminus \Gamma$.
We need therefore function evaluations only in the points of a mesh shifted 
by $h_{\ell}/2$ with respect to the original mesh $\cT^\ell$, and avoid
evaluations at the singularity.
\end{remark}

After $\uqtd^\ell$ is accurately approximated for every virtual mesh level $\ell$ 
using exponential sums, 
we reduce the number of parameters in the corresponding quantized tensor representation 
$\mathsf{qtd}$ of $\scrA^\ell\uqtd^\ell$ by using $\texttt{round}_\mathsf{qtd}$.
%%%%%%%%%%%%%%%%%%%%%%%%%%%%%%%%%%%%%%%%%%%%%%%%%%%%%%%%%%%%%%%%%%%%%%%
\subsection{QTT-FEM for eigenvalue problems with singular potential}
\label{sec:num:evp}
%%%%%%%%%%%%%%%%%%%%%%%%%%%%%%%%%%%%%%%%%%%%%%%%%%%%%%%%%%%%%%%%%%%%%
We apply QTT-formatted compression to the numerical solution of
the eigenvalue problem \eqref{eq:NLSch}, linearized and with singular potential $V$:
\begin{equation*}
  %\label{eq:hydrogeneq}
	\left(-\frac{1}{2}\Delta - \frac{1}{|x|} \right) u(x) = \lambda\, u(x), \quad x\in\mathbb{R}^3.
\end{equation*}
This is, essentially, Schr\"odinger's equation for the hydrogen atom.
It is well-known (e.g., \cite[Chapter 10]{LL81Q}) 
that the eigenfunctions $u_{n,l,m}$ can be
enumerated by three integer quantum numbers: 
$n=1,2,\dots$---principal quantum number, 
$l=0,1,\dots,n-1$---orbital quantum number, and 
$m=-l,\dots,l$, magnetic quantum number.
The corresponding eigenvalues are $\lambda_n = 1/(2n^2)$.
We aim at approximating the $3$ smallest eigenvalues $\lambda_n$ and
their respective $N_\mathsf{ev}=14$ eigenvectors~$u_{n,m,l}$, $n=1,2,3$.

To solve the problem numerically, we replace $\mathbb{R}^3$ with a finite domain $\Omega = (-a, a)^3$, $a=100$ and impose homogeneous Dirichlet boundary conditions.
To discretize the problem, we introduce a mesh with the nodes
\[
	x_{i,j,k} = -a + (i,j,k)h_\ell, \quad h_\ell = \frac{2a}{2^{\ell}+1},
\]
where $(i,j,k)\in \{0,1,\dots,2^{\ell}+1\}^3$.
For $(i,j,k) \in \{1, \dots, 2^\ell\}^3$, we denote by $\phi_{i,j,k}^\ell$ the 
piecewise trilinear, continuous nodal Lagrange functions satisfying
\begin{equation*}
  \phi_{i,j,k}^\ell (\bx_{p,q,r}) = \delta_{ip}\delta_{jq}\delta_{kr},
\quad 
(i,j,k)\in \{1, \dots, 2^\ell\}^3, \ 
(p,q,r) \in \{0, \dots, 2^\ell+1\}^3,
\end{equation*}
and introduce the associated finite element space $\spn\{\phi^\ell_{i,j,k}\}$.
We discretize the problem using the finite element method.
The discretized eigenvalue problem reads
\begin{equation}\label{eq:eigval_discr}
 \left(\frac{1}{2}D^{\ell} + V^{\ell}\right) u^{\ell} 
    = \lambda^{\ell}\, M^{\ell} u^{\ell},  
 \quad 
 u^{\ell}\in \mathbb{R}^{2^\ell\times 2^\ell \times 2^\ell},
\end{equation}
where $D^{\ell}$ and $M^{\ell}$ are, respectively, 
stiffness and mass linear operators\footnote{Here, 
linear operators are mappings 
$A:\mathbb{R}^{2^\ell \times 2^\ell \times 2^\ell} \to \mathbb{R}^{2^\ell \times 2^\ell \times 2^\ell}$
given as $6$-dimensional arrays such that the action on $u\in \mathbb{R}^{2^\ell \times 2^\ell \times 2^\ell}$ 
is defined by 
\[
(Au)_{i,j,k} = \sum_{p,q,r=1}^{2^\ell}A_{i,j,k,p,q,r} u_{p,q,r}, \quad (i,j,k)\in \{1,\dots,2^\ell\}^3. 
\] 
}:
\begin{align*}
  &  (D^{\ell})_{i,j,k,p,q,r} = \int_{\Omega} \nabla\phi_{i,j,k}^\ell (x) \nabla\phi_{p,q,r}^\ell (x)\, dx,
  \\ & (M^{\ell})_{i,j,k,p,q,r} = \int_{\Omega} \phi_{i,j,k}^\ell(x) \phi_{p,q,r}^\ell(x)\, dx
\end{align*}
for $(i,j,k),(p,q,r)\in \{1, \dots, 2^\ell\}^3$, and $V_\ell$ is the matrix of
the FE discretization of $V(x) = -|x|^{-1}$:
 \[
 	(V^\ell)_{i,j,k,p,q,r}  = -\int_{\Omega} \frac{1}{|x|}\, \phi_{i,j,k}^\ell (x)\, \phi_{p,q,r}^\ell (x)\, d x.
 \]
assembled with the exponential sums approach.

 \begin{algorithm}[t] %H
 \caption{Block eigensolver in TQTT format based on derivative-free formulas. 
The algorithm is formulated for three-dimensional arrays, 
implying that all the operations are performed within the TQTT format.}
\label{alg:eigenblock}
\begin{algorithmic}[1]
 \Require 
  Initial guess to eigenvectors $u^{\ell,0}_\alpha$ and to eigenvalues $\lambda^{\ell,0}_\alpha, \alpha = 1,\dots,N_\mathsf{ev}$, 
  tolerance parameter $\delta$.
 \Ensure 
  Approximation to eigenvectors $u^{\ell}_\alpha(\delta)$ 
  and to eigenvalues $\lambda^{\ell}_\alpha(\delta), \alpha = 1,\dots,N_\mathsf{ev}$.
 \For{$k = 1, 2,\dots$ {until converged}}
  \For{$\alpha=1,\dots,N_\mathsf{ev}$}
 \State Approximate $V^{\ell} u^{\ell,k-1}_\alpha$ using algorithm \texttt{mvrk2} from TT-Toolbox.
 \State With tolerance $\delta$ using the ADI-based solver~\cite{R19}, solve 
 \[
 \left(\frac{1}{2}D^{\ell} - \lambda^{\ell,k-1}_\alpha\, M^{\ell} \right)u^{\ell,k}_\alpha = -V^{\ell}\, u^{\ell,k-1}_\alpha. %\quad \alpha=1,\dots,b.
 \]
 \State Approximate $V^{\ell} u^{\ell,k}_\alpha$ using algorithm \texttt{mvrk2} from TT-Toolbox.
 \EndFor
   \For{$\alpha=1,\dots,N_\mathsf{ev}$}
     \For{$\beta=1,\dots,N_\mathsf{ev}$}
% \State Using derivative-free formulas for $F^{k}_{\alpha\beta}$ \eqref{eq:fockmat}, calculate:  
  \State Calculate \label{alg:line:derfree} 
  {\small $
 	F_{\alpha \beta}^{k} = \lambda_\beta^{\ell,k-1} \left<u^{\ell,k}_\alpha, u^{\ell,k}_\beta \right> + \left<u^{\ell,k}_\alpha, V^{\ell} u^{\ell,k}_\beta \right> - \left<u^{\ell,k}_\alpha, V^{\ell}u^{\ell,k-1}_\beta \right>.
 $} 
 \State Calculate {\small $G_{\alpha \beta}^{k} = \left<u^{\ell,k}_\alpha, u^{\ell,k}_\beta \right>$}.
  \EndFor
   \EndFor
 \State Solve the generalized eigenvalue problem
 \[
 	F^{k} S = G^{k} S \Lambda, \quad S\in \mathbb{R}^{N_\mathsf{ev}\times N_\mathsf{ev}}, \quad \Lambda = \diag(\lambda^{\ell,k}_1,\dots,\lambda^{\ell,k}_{N_\mathsf{ev}}) \in \mathbb{R}^{N_\mathsf{ev}\times N_\mathsf{ev}}.
 \]
 \For{$\alpha=1,\dots,N_\mathsf{ev}$}
 \State Calculate ${\widetilde u}^{\ell,k}_\alpha = \texttt{round}(\sum_{\beta=1}^{N_\mathsf{ev}}S_{\alpha\beta}\,u^{\ell,k}_\beta,\  \delta)$.
 \State Calculate $u^{\ell,k}_\alpha = {\widetilde u}^{\ell,k}_\alpha / \|{\widetilde u}^{\ell,k}_\alpha\|_2$.
  \EndFor
 \EndFor
\State Set $u^{\ell}_\alpha(\delta)=u^{\ell,k}_\alpha$, $\lambda^{\ell}_\alpha(\delta)=\lambda^{\ell,k}_\alpha$, $\alpha =1,\dots, N_\mathsf{ev}$.
\end{algorithmic}
\end{algorithm}

To solve the problem, we approximate the eigenvectors corresponding to
the smallest eigenvalues in the TQTT format that yields the smallest amount
of degrees of freedom for a given error (compared with the QTT and QT3 formats) according to Figures~\ref{fig:conv3d} and~\ref{fig:conv3d_dof}.
Note that due to the extremely refined underlying virtual meshes with
$2^{3\ell}$ internal equispaced points, 
the stiffness matrix $D^{\ell}$ becomes
severely ill-conditioned (its condition number scales as $h^{-2}_\ell$, 
i.e., it grows exponentially in $\ell$).
Besides, there arises an effect of ill-conditioning for large $\ell$ connected
purely with the structure of tensor decompositions, see~\cite{BK18PREC}.
Therefore, in order to overcome the effect of algebraic and representation
  ill-conditioning and to accurately approximate the eigenvalues and
corresponding eigenvectors of \eqref{eq:eigval_discr}, 
particular attention has to be devoted to
technical details of the computation.
The overall procedure---based on the preconditioned gradient descent
method and on the Rayleigh-Ritz procedure---is summarized in Algorithm~\ref{alg:eigenblock}.
In the algorithm, we utilize ``derivative-free'' formulas~\cite{RO16} (that
avoid multiplications by $D^{\ell}$, see Algorithm~\ref{alg:eigenblock}, line~\ref{alg:line:derfree}) for calculating the 
$N_\mathsf{ev}\times N_\mathsf{ev}$ matrix $F$ 
given by
\begin{equation}\label{eq:fockmat}
F_{\alpha \beta}^{k} 
= 
\left<u^{\ell,k}_\alpha, \left(\frac{1}{2}D^{\ell} 
+
V^{\ell}\right) u^{\ell,k}_\beta \right>, \quad \alpha,\beta=1,\dots,N_\mathsf{ev}
\end{equation}
where $u^{\ell,k}_\alpha$ are three-dimensional arrays represented in the 
TQTT format that approximate $u^{\ell}_\alpha$ on the $k$-th step 
of the iterative process and $\left<\cdot,\cdot \right>$ denotes scalar products of three-dimensional arrays:
\[
	\left<u, v\right> = \sum_{i,j,k=1}^{2^\ell} u_{i,j,k} v_{i,j,k}, \quad u,v\in\mathbb{R}^{2^\ell\times 2^\ell \times 2^\ell}.
\]
To solve the screened Poisson's equations arising in Algorithm~\ref{alg:eigenblock}, we utilize the algorithm proposed in~\cite{R19}, which is based on the alternating direction implicit method and allows to approximate the solution without conditioning issues.

Let $\lambda^{\ell}_{n,l,m} (\delta)$, $n=1,2,3$ ($N_\mathsf{ev}=14$) be the
eigenvalues obtained by using Algorithm~\ref{alg:eigenblock} with a tolerance
parameter $\delta$ and sorted by their quantum numbers.
Let us calculate an average numerical eigenvalue for fixed $n$ and $l$ 
\begin{equation}\label{eq:avlam}
	\overline{\lambda}_{n,l}^\ell (\delta) = \frac{1}{2l + 1} \sum_{m=-l}^{l} \lambda^{\ell}_{n,l,m}(\delta), \quad n=1,2,3.
\end{equation}
To each $\overline{\lambda}_{n,l}$, we associate a number of degrees of freedom, 
which is averaged in $m$ by analogy with~\eqref{eq:avlam}.
For every $\ell$, 
select the parameters $\delta^\ell$ as the largest numbers satisfying
\[
|\overline{\lambda}^{\ell}_{n,l} (\delta^\ell) - \lambda_n| 
\leq 
c_{n,l} |\lambda^{\ell}_{n,l}(\delta_{\mathrm{ref}}) - \lambda_n|,
\]
where we chose $\delta_{\mathrm{ref}} = 10^{-10}$ and where the
constants $c_{n,l}$ satisfy $c_{n,l} > 1$  (the practical choice is $c_{n,l} = 1.01$).
In Figure~\ref{fig:eigval_dof}, we present the errors 
\[
\varepsilon_\ell = \frac{|\overline{\lambda}^{\ell}_{n,l}(\delta^\ell)  - \lambda_{n}|}{|\lambda_{n}|},
\] 
in eigenvalues $\lambda_n$, $n=1,2,3$ with respect to the effective 
number of degrees of freedom for the eigenvalue problem~\eqref{eq:eigval_discr}.
 
Note that in this section, 
the implementation is done using the open source library 
TT-Toolbox\footnote{\url{https://github.com/oseledets/TT-Toolbox}}, 
which contains the implementation of the two-level QTT Tucker format~\cite{DKQTT13}. 
In three space dimensions, this format is equivalent to the TQTT format with
negligible overhead\footnote{In the two-level QTT Tucker format,
    the Tucker core of size $R_1\times R_2 \times R_3$ is additionally
    decomposed using the TT decomposition, which leads to TT-cores of sizes
    $R_1\times R_1$, $R_1\times R_2 \times R_3$, $R_3\times R_3$. So, compared
    with TQTT, the two-level QTT Tucker format leads to the storage of
    $\mathcal{O}(R_1^2 + R_3^2)$ additional degrees of freedom.}.

\begin{figure}
  \centering
\includegraphics[width=1.0\textwidth]{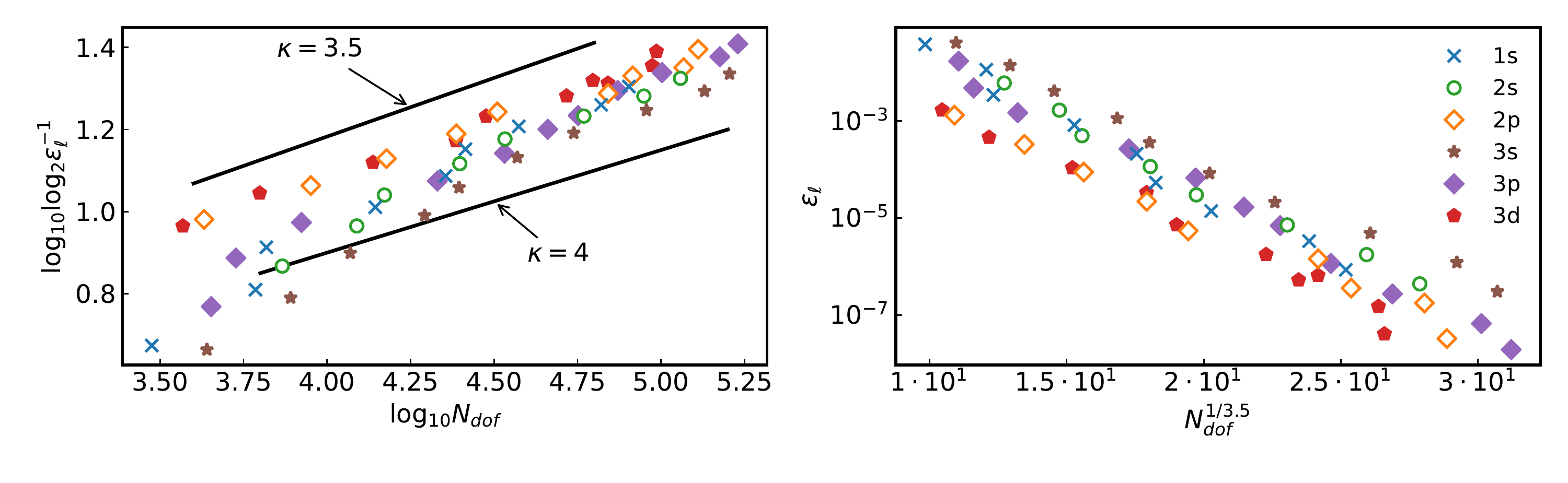}
  \caption{Relative errors $\varepsilon_\ell = |\overline{\lambda}^{\ell}_{n,l}
    (\delta) - \lambda_n|/|\lambda_n|$, $n=1,2,3$, $l=0,1,\dots,n-1$ in double
    logarithmic scale (left) and single logarithmic scale (right) with respect
    to the averaged number of degrees of freedom for the eigenvalue
    problem~\eqref{eq:eigval_discr}. In the legend, the numbers $1,2,3$ denote
    the principal quantum number $n$ and the letters $s,p,d$ correspond to $l=0,1,2$ respectively.}
  \label{fig:eigval_dof}
\end{figure}